\documentclass[a4paper, fleqn]{article}
\usepackage[english]{babel}
\usepackage{graphicx}
\usepackage{amsmath, amssymb, graphics, float, fullpage}
\usepackage{amsthm}
\usepackage{makeidx}
\usepackage{lipsum}
\usepackage{tikz}
\usepackage{amssymb, amsfonts, amsmath, amsthm, latexsym}
\usepackage{epsfig} 
\usepackage{graphicx}
\usepackage{float,fullpage}
\usepackage{mathtools}
\usepackage{xparse}
\usepackage{makeidx}
\usepackage{lipsum}
\usepackage{tikz}
\usepackage{comment}

\numberwithin{equation}{section}

\newtheorem{theorem}{Theorem}[section]
\newtheorem{lemma}[theorem]{Lemma}
\newtheorem{proposition}[theorem]{Proposition}
\newtheorem{corollary}[theorem]{Corollary}
\newtheorem{definition}[theorem]{Definition}
\newtheorem{remark}[theorem]{Remark}

\newtheorem{rhproblem}[theorem]{Riemann-Hilbert Problem}
\newtheorem{assumption}[theorem]{Assumption}
\newtheorem{convention}[theorem]{Convention}

\providecommand\phantomsection{}

\makeindex

\newcommand\MeijerG[7]{ G^{\,#1,#2}_{#3,#4}\MeijerM*{#5}{#6}{#7}}

\begin{document}

\DeclarePairedDelimiterX\MeijerM[3]{\lparen}{\rparen}{\begin{matrix}#1 \\ #2\end{matrix}\delimsize\vert\,#3}


\title{The local universality of Muttalib-Borodin ensembles when the parameter $\theta$ is the reciprocal of an integer}
\author{L. D. Molag}
\maketitle
\begin{center}
KU Leuven, Department of Mathematics,\\
Celestijnenlaan 200B box 2400, BE-3001 Leuven, Belgium.\\
E-mail: leslie.molag@kuleuven.be
\vspace{1.0cm}
\begin{abstract}
The Muttalib-Borodin ensemble is a probability density function for $n$ particles on the positive real axis that depends on a parameter $\theta$ and a weight $w$. We consider a varying exponential weight that depends on an external field $V$. 
In a recent article, the large $n$ behavior of the associated correlation kernel at the hard edge was found for $\theta=\frac{1}{2}$, where only few restrictions are imposed on $V$. In the current article we generalize the techniques and results of this article to obtain analogous results for $\theta=\frac{1}{r}$, where $r$ is a positive integer. The approach is to relate the ensemble to a type II multiple orthogonal polynomial ensemble with $r$ weights, which can then be related to an $(r+1)\times (r+1)$ Riemann-Hilbert problem. The local parametrix around the origin is constructed using Meijer G-functions. We match the local parametrix around the origin with the global parametrix with a double matching, a technique that was recently introduced.
\end{abstract}
\end{center}

\tableofcontents


\section{Introduction and main result}

\subsection{Introduction} 

The Muttalib-Borodin ensemble (MBE) with parameter $\theta>0$ and weight $w$ is defined by the following joint probability density function for particles on the positive half-line. 
\begin{align} \label{ch4:eq:defMBE}
& \frac{1}{Z_n} \prod_{1\leq i<j\leq n} (x_i-x_j)(x_i^\theta-x_j^\theta) \prod_{j=1}^n w(x_j), & x_1,\ldots,x_n> 0.
\end{align}
Here $Z_n>0$ is a normalization constant. We consider an $n$-dependent weight
\begin{align} \label{ch4:eq:defMBEw}
w(x)=w_\alpha(x)= x^\alpha e^{-n V(x)},
\end{align}
where $\alpha>-1$ and $V:[0,\infty)\to \mathbb R$ is an external field that has enough increase at infinity. The latter is imposed to assure that \eqref{ch4:eq:defMBE} is integrable and thus normalizable. A sufficient condition would be to have $V(x)\geq \frac{1+\theta}{2} \log(x)$ for $x$ big enough. We put $V(0)=0$ without loss of generality.

In 1995 Muttalib introduced the model as a simplified model for disordered conductors in the metallic regime \cite{Mu}. This type of disordered conductors was not accurately described by the existing random matrix models. A few years later Borodin obtained interesting results for several specific choices of the weight $w_\alpha$ \cite{Bo}, most notably for the Laguerre case, i.e., when $V$ is linear. For linear external fields he found a new scaling limit, that we turn to in a second. The model has seen a revival of interest, as it became clear in recent years that the MBE is connected to several random matrix models \cite{FoWa, Ch, BeGeSz, AkIpKi}, where it describes either the eigenvalue density or the density of the squared singular values. We also mention the recent results on the corresponding large gap probabilities of the MBE \cite{ClGiSt,ChLeMa}. See \cite{YaAlMuWa} for a recent attempt of Yadav, Muttalib et. al. to model certain physical systems with a generalization of the MBE.

The MBE is a determinantal point process and thus it has an associated correlation kernel $K_{V,n}^{\alpha,\theta}$. In fact, it is a biorthogonal ensemble \cite{Bo}, and this implies that we may take
\begin{align}
\label{ch4:defK}
K_{V,n}^{\alpha,\theta}(x,y) = w_\alpha(x) \sum_{j=0}^{n-1} p_j(x) q_j(x^\theta),
\end{align}
where $p_j$ and $q_j$ are polynomials of degree $j$ that satisfy
\begin{align}
\label{ch4:defpnqn}
\int_0^\infty p_j(x) q_k(x^\theta) w_\alpha(x) dx &= \delta_{ij}, & j=0,1,\ldots
\end{align}
In the large $n$ limit the particles, corresponding to the weight \eqref{ch4:eq:defMBEw}, behave almost surely according to a limiting empirical measure $\mu_{V,\theta}^*$ that minimizes a corresponding equilibrium problem, as was shown by Claeys and Romano in \cite{ClRo}. Namely, $\mu_{V,\theta}^*$ minimizes the functional
\begin{align} 
\label{ch4:defmuVtheta}
\frac{1}{2} \iint \log \frac{1}{|x-y|} d\mu(x) d\mu(y)
+ \frac{1}{2} \iint \log \frac{1}{|x^\theta-y^\theta|} d\mu(x) d\mu(y)
+ \int V(x) d\mu(x).
\end{align}
The Euler-Lagrange variational conditions (see \cite{ClRo}) corresponding to this minimization problem take the form
\begin{align} \label{ch4:eq:varCon}
\int\log|x-s| d \mu_{V,\theta}^* + \int\log|x^\theta -s^\theta| d \mu_{V,\theta}^* 
\left\{\begin{array}{ll}
= V(x) + \ell, & x\in\operatorname{supp}(\mu_{V,\theta}^*),\\
\leq V(x) + \ell, & x\in [0,\infty),
\end{array}\right.
\end{align}
where $\ell$ is a real constant. For specific choices of $V$ we know how the correlation kernel behaves around the origin in the large $n$ limit. In particular, Borodin \cite{Bo} calculated the hard edge scaling limit of \eqref{ch4:defK} for the Laguerre case, i.e., where $V(x)=x$. His result translates to
\begin{align} \label{ch4:eq:scalingLimitK}
\lim_{n\to\infty} \frac{1}{n^{1+\frac{1}{\theta}}} K_{V,n}^{\alpha,\theta}\left(\frac{x}{n^{1+\frac{1}{\theta}}},\frac{y}{n^{1+\frac{1}{\theta}}}\right)
= \mathbb{K}^{(\alpha,\theta)}(x,y),
\end{align}
where 
\begin{align} \label{ch4:eq:scalingLimitIK}
\mathbb{K}^{(\alpha,\theta)}(x,y) = 
\theta y^\alpha \int_0^1 J_{\frac{\alpha+1}{\theta},\frac{1}{\theta}}(ux) J_{\alpha+1,\theta}\left((uy)^\theta\right) u^\alpha du,
\end{align}
and $J_{a,b}$ is Wright's generalized Bessel function. The scaling limit \eqref{ch4:eq:scalingLimitK} is valid for any fixed $\theta>0$. When either $\theta$ or $1/\theta$ is a positive integer the limiting kernel coincides (up to rescaling and a gauge factor) with the so-called Meijer G-kernel \cite{AkIpKi, KuSt}. In \cite{KuMo} we conjectured that one would obtain the scaling limit \eqref{ch4:eq:scalingLimitK} for a much larger class of external fields, for any fixed $\theta>0$. Such universality was well-known for $\theta=1$, where one obtains the Bessel kernel. Indeed, the Bessel kernel coincides with \eqref{ch4:defK} when $\theta=1$. In \cite{KuMo}, the conjecture was proved for $\theta=\frac{1}{2}$. 

\subsection{Statement of results}

In this paper we will go a step further and prove the conjecture for all $\theta=\frac{1}{r}$ with $r$ a positive integer. There are several advantages when we restrict to such $\theta$. First of all, it is then known that the biorthogonal ensemble can be related to a multiple orthogonal polynomial ensemble (MOP) with $r$ weights $w_\alpha, w_{\alpha+\frac{1}{r}}, \ldots, w_{\alpha+\frac{r-1}{r}}$ (see \cite{KuMo}), Lemma 2.1). That is, we can take $p_n$, as in \eqref{ch4:defpnqn}, to be the unique monic polynomial that satisfies
\begin{align}
\label{ch4:defMOP}
\int_0^\infty p_n(x) x^k w_{\alpha+\frac{j-1}{r}}(x) dx &= 0, & j=1,2,\ldots,r, \quad k = 0,1,\ldots,\left\lfloor \frac{n-j}{r} \right\rfloor.
\end{align} 
Secondly, it was shown \cite{Ku} by Kuijlaars that there is, besides the equilibrium problem as in \eqref{ch4:defmuVtheta}, also a corresponding \textit{vector equilibrium problem} consisting of $r$ measures. We describe this vector equilibrium problem in Section \ref{ch4:sec:normalization2}. It is interesting that such a vector equilibrium problem also exists when $\theta$ is assumed to be rational (although it is unclear which multiple orthogonal polynomials, if any, would correspond to that situation). 

Our result will be valid under a generic restriction (see \cite{ClRo}, Theorem 1.8) on the external field. As in \cite{KuMo}, we call $V$ one-cut $\theta$-regular when the equilibrium measure $\mu_{V,\theta}^*$ is supported on one interval $[0,q]$, for some $q>0$, has a density that is positive on $(0,q)$ and that behaves near the endpoints as
\begin{align} \label{ch4:eq:onecutthetabehav}
\frac{d \mu_{V,\theta}^*(s)}{ds} =
\left\{\begin{array}{ll}
c_{0,V} (1+o(1)) s^{-\frac{1}{\theta+1}}, & s\downarrow 0,\\
c_{1,V} (1+o(1)) (q-s)^\frac{1}{2}, & s\uparrow q
\end{array}\right. 
\end{align}
for some positive constants $c_{0,V}$ and $c_{1,V}$, and in addition we demand that the inequality in \eqref{ch4:eq:varCon} is strict for $x>q$.  This last condition is not essential, but it will make our derivation cleaner. The one-cut condition,  added for convenience as well, is also not absolutely necessary. The main result holds as long as the support of the equilibrium measure contains a closed interval with left-end point $0$ (and \eqref{ch4:eq:onecutthetabehav} is satisfied). A sufficient condition for $V$ to be one-cut $\frac{1}{r}$-regular is that it is twice differentiable on $[0,\infty)$ and that $x V'(x)$ is increasing for $x>0$. A proof for this can be found in Proposition 2.4 in \cite{KuMo}, the proof is for $\theta=\frac{1}{2}$ but with a mild modification it also works for all rational $\theta>0$. Notice in particular, that linear external fields are one-cut $\frac{1}{r}$-regular. The main result of this paper is the following. 

\begin{theorem}
\label{ch4:mainThm}
Let $\alpha>-1$ and let $\theta=\frac{1}{r}$, where $r$ is a positive integer.  Let $V:[0,\infty)\to\mathbb{R}$ be a one-cut $\theta$-regular external field which is real analytic on $[0,\infty)$. Then for $x,y\in (0,\infty)$ we have 
\begin{align} \label{ch4:eq:mainResult}
\lim_{n\to\infty} \frac{1}{(c n)^{r+1}} K_{V,n}^{\alpha,\frac{1}{r}}\left(\frac{x}{(c n)^{r+1}},\frac{y}{(c n)^{r+1}}\right)
& = \mathbb K^{(\alpha,\frac{1}{r})}(x,y),
\end{align}
uniformly on compact sets, where $c = \pi c_{0,V}/\sin\frac{\pi}{r+1}$ with $c_{0,V}$ as in \eqref{ch4:eq:onecutthetabehav}.
\end{theorem}
We remark that the substitution $x\to x^{\frac{1}{\theta}}$ changes the MBE to one with a different input. Namely, we should then substitute $\theta, \alpha$ and $V(x)$ by $1/\theta, (1+\alpha)/\theta -1$ and $V(x^\frac{1}{\theta})$ respectively. This means that the main result is also true when we replace $r$ by $\frac{1}{r}$ everywhere in its formulation, but with the altered condition that $V(x^\frac{1}{r})$ should be one-cut $\frac{1}{r}$-regular. Then $V(x)$ should have a power series expansion evaluated in $x^r$, and this severely restricts what type of external fields can be treated. There does not appear to be a simple way to relax this restriction, although we believe that the main result should hold without it. 

Our approach is conceptually the same as in \cite{KuMo}. The main difference is that the calculations become more technical and involved. The MOP ensemble \eqref{ch4:defMOP} is related to an $(r+1)\times (r+1)$ Riemann-Hilbert problem (RHP) which we will present in the next section. We analyze this RHP using the Deift-Zhou method of nonlinear steepest descent and this will allow us to prove Theorem \ref{ch4:mainThm}. To make our derivation cleaner we will assume that $n$ is divisible by $r$, but we explain in Appendix \ref{ch:appendixA} how the case where $n$ is not divisible by $r$ is treated. Notice that $\left\lfloor \frac{n-j}{r} \right\rfloor=\frac{n}{r}-1$ for all $j=1,2,\ldots,r$ in the case that $n$ is divisible by $r$. 

The local parametrix at the hard edge is constructed with the help of Meijer G-functions. This was to be expected, Zhang showed in \cite{Zh} that the limiting correlation kernel in \eqref{ch4:eq:scalingLimitIK} can be expressed with the help of Meijer G-functions when either $\theta$ or $1/\theta$ is a positive integer. The local parametrix that we find shows great similarity with the bare Meijer G-parametrix from \cite{BeBo}, although there does not appear to be a simple transformation that relates these two local parametrix problems. As in other larger size RHPs (e.g., see \cite{BeBo} and \cite{KuMFWi}), we will not be able to match the local parametrix with the global parametrix in the usual way. In order to match the global and local parametrix we are going to use a double matching. This double matching procedure was introduced in \cite{KuMo} and was later applied in \cite{SiZh}. In \cite{Mo}, the double matching procedure was refined and a general framework was put forward. The current paper will be the first instance were this general framework for the double matching procedure is utilized. As a convenience to the reader, we repeat the main result concerning the double matching of \cite{Mo} in Section \ref{ch4:sec:matching} (see \text{Theorem \ref{lem:matching}}).

Having done the RHP analysis, one can also calculate the scaling limits of the correlation kernel in the bulk and at the soft edge $q$ to be the sine and Airy kernel respectively. We omit the details. 

We do not believe it to be reasonable to expect that our method can be generalized to all $\theta>0$, but it might be possible to adapt it to obtain the same results for rational $\theta$. In particular, it was shown in \cite{Ku} that there exists an underlying vector equilibrium problem when $\theta$ is rational. The measures of its solution might be used to construct $g$-functions for a corresponding RHP, although, at the moment, it is not clear to us what this RHP would look like. To prove the conjecture for irrational $\theta$ we would suspect that an entirely new approach has to be invented, although a density or continuity argument might do the trick once the conjecture is proved for rational $\theta$.

Another question for further research, is whether the real analyticity of $V$ can be relaxed. Our current approach can not deal with the situation where $V$ is is not real analytic. To be specific, equation \eqref{ch4:eq:varphijcrelation} would not necessarily be valid anymore for $j=0$. This means that our Szeg\H{o} function as defined in \eqref{ch4:eq:defD0} does not actually lead to a local parametrix problem with constant jumps (see Section \ref{ch4:sec:szegoconstantjumps}). Furthermore, the map $f$ as defined in \eqref{ch4:eq:conformalf} will not be analytic, hence not conformal. It has been suggested to us that the $\overline{\partial}$-method as introduced by McLaughlin and Miller \cite{McMi}, adapted to larger size RHPs, might be able to deal with more general external fields.\\

The following three expressions will be used repeatedly throughout this paper.
\begin{align}
\label{ch4:defab}
\beta = \alpha + \frac{r-1}{2 r}, \quad\quad\Omega = e^\frac{2\pi i}{r}, \quad\quad\omega = e^\frac{2\pi i}{r+1}.
\end{align}
We will use these notations without reference henceforth. Throughout this paper we use principle branches for fractional powers and logarithms, i.e., we pick the argument of $z$ between $-\pi$ and $\pi$. In the few cases were we have no choice but to deviate from this convention, it will be explicitly mentioned what branch we take.




\section{The Riemann-Hilbert problem}

\subsection{Introduction of the Riemann-Hilbert problem} \label{ch4:sec:theRHP}

As mentioned in the introduction we will assume that $n$ is divisible by $r$. This choice is made because the intuition behind some of the formulae that we will encounter might be obscured if we include the $n$ that are not divisible by $r$. Most of the RH analysis is identical for such $n$ though, see Appendix \ref{ch:appendixA}. The MOP ensemble defined in \eqref{ch4:defMOP} is related to an $(r+1)\times (r+1)$ RHP \cite{VAGeKu}, which takes the form

\begin{rhproblem} \label{ch4:RHPforY} \
\begin{description}
\item[RH-Y1] $Y : \mathbb{C}\setminus [0,\infty)\to \mathbb{C}^{(r+1)\times (r+1)}$ is analytic.
\item[RH-Y2] $Y$ has boundary values for $x\in (0,\infty)$, denoted by $Y_{+}(x)$ (from the upper half plane) and $Y_{-}(x)$ (from the lower half plane), and for such $x$ we have the jump condition
\begin{align} \label{ch4:RHY2}
Y_{+}(x) = Y_{-}(x) \begin{pmatrix} 1 & w_\alpha(x) & w_{\alpha+\frac{1}{r}}(x) & \ldots &w_{\alpha+\frac{r-1}{r}}(x) \\ 
0 & 1 & 0 & \ldots & 0\\
0 & 0 & 1 & \ldots & 0\\
\vdots & & & & \vdots\\
0 & 0 & 0 & \ldots & 1 \end{pmatrix}.
\end{align}
\item[RH-Y3] As $|z|\to\infty$
\begin{align} \label{ch4:RHY3}
Y(z) = \left(\mathbb{I}+\mathcal{O}\left(\frac{1}{z}\right)\right) 
\begin{pmatrix} 
z^{n} & 0 & 0 & \ldots & 0\\ 
0 & z^{-\frac{n}{r}} & 0 & \ldots & 0\\ 
0 & 0 & z^{-\frac{n}{r}} & \ldots & 0\\
\vdots & & & & \vdots\\
0 & 0 & 0 & \ldots & z^{-\frac{n}{r}}
\end{pmatrix}.
\end{align}
\item[RH-Y4] As $z\to 0$
\begin{align}
\label{ch4:RHY4}
Y(z) = \mathcal{O}\begin{pmatrix} 1 & h_{\alpha}(z) & h_{\alpha+\frac{1}{r}}(z) & \ldots & h_{\alpha+\frac{r-1}{r}}(z)\\ 
1 & h_{\alpha}(z) & h_{\alpha+\frac{1}{r}}(z) & \ldots & h_{\alpha+\frac{r-1}{r}}(z)\\ 
\vdots & & & & \vdots\\
1 & h_{\alpha}(z) & h_{\alpha+\frac{1}{r}}(z) & \ldots & h_{\alpha+\frac{r-1}{r}}(z)\end{pmatrix}
\text{ with } \, h_{\alpha}(z) = 
\begin{cases} |z|^{\alpha}, & \text{if } \alpha < 0, \\ 
	\log{|z|}, & \text{if } \alpha = 0,\\ 
	1, & \text{if } \alpha > 0.
	\end{cases}
\end{align}
\end{description}
The $\mathcal{O}$ condition in \eqref{ch4:RHY3} and \eqref{ch4:RHY4} is to be taken entry-wise.
\end{rhproblem}

It will be convenient to use the following convention for our RH analysis, which is indeed seen to be in agreement with RH-Y2. 
\begin{convention} \label{ch4:con:orientation}
Any jump curve that touches the origin is oriented away from the origin.
\end{convention}


Notice that this means, perhaps contrary to intuition, that the $+$ and $-$ signs are in the lower and upper half-plane respectively, when we consider jumps on the negative real axis. We will never deviate from Convention \ref{ch4:con:orientation}. \\

RH-Y has a unique solution $Y(z)$, which is related to the multiple orthogonal polynomials in the following way. The first row of $Y(z)$ can be expressed as
\begin{multline*}
\hspace{-0.4cm}\left(
p_n(z) \quad \displaystyle\frac{1}{2\pi i} \int_0^\infty \frac{p_n(x) w_\alpha(x)}{x-z} dx
\quad \displaystyle\frac{1}{2\pi i} \int_0^\infty \frac{p_n(x) w_{\alpha+\frac{1}{r}}(x)}{x-z} dx
\quad \hdots\right. \\
\left. \hdots\quad \displaystyle\frac{1}{2\pi i} \int_0^\infty \frac{p_n(x) w_{\alpha+\frac{r-1}{r}}(x)}{x-z} dx\right).
\end{multline*}
The other rows are similar, but are expressed with different though similar multiple orthogonal polynomials (see \cite[Theorem 3.1]{VAGeKu}). 
It is known \cite{DaKu} that the correlation kernel \eqref{ch4:defK} can be conveniently expressed in terms of $Y$ by
\begin{align} \label{ch4:eq:KinY}
K_{V,n}^{\alpha,\frac{1}{r}}(x,y)
&= \frac{1}{2\pi i(x-y)} 
\begin{pmatrix}
0 & w_\alpha(y) & w_{\alpha+\frac{1}{r}}(y) & \hdots & w_{\alpha+\frac{r-1}{r}}(y)
\end{pmatrix}
Y_+^{-1}(y) Y_+(x)
\begin{pmatrix}
1 \\ 0 \\ \vdots \\ 0
\end{pmatrix}\\ \nonumber
&= \frac{w_0(y)}{2\pi i(x-y)}
\begin{pmatrix}
0 & y^\alpha & y^{\alpha+\frac{1}{r}} & \hdots & y^{\alpha+\frac{r-1}{r}}
\end{pmatrix}
Y_+^{-1}(y) Y_+(x)
\begin{pmatrix}
1 \\ 0 \\ \vdots \\ 0
\end{pmatrix}
\end{align}
for $x,y>0$. Hence, to obtain the large $n$ behavior of the correlation kernel it will suffice to determine the large $n$ behavior of $Y$. 

\subsection{First transformation $Y\mapsto X$} \label{ch4:sec:firstTransformation}

In order to be able to normalize the RHP properly we will need a first transformation that will turn the jumps into a direct sum of $1\times 1$ and $2\times 2$ jumps. Here and in the rest of this paper we shall often use a block form notation, similarly as in \cite{BeBo}. As in \cite{BeBo}, we also often write a diagonal matrix as a direct sum of $1\times 1$ blocks, in cases where formulae tend to become big. We mention that (matrix) multiplication has a higher precedence than the direct sum in the order of operations.  

We remind the reader that $\Omega=e^\frac{2\pi i}{r}$. We introduce the $r\times r$ matrices
\begin{align}
\label{ch4:defU+}
U^+ &= \frac{1}{\sqrt r}
\begin{pmatrix}
1 & 1 & 1 & 1 & 1 & 1 & \hdots \\
1 & \Omega & \Omega^{-1} & \Omega^2 & \Omega^{-2} & \Omega^3 & \hdots \\
1 & \Omega^2 & \Omega^{-2} & \Omega^4 & \Omega^{-4} & \Omega^6 & \hdots \\
\vdots & & & & & & \vdots\\
1 & \Omega^{r-1} & \Omega^{-(r-1)} & \Omega^{2(r-1)} & \Omega^{-2(r-1)} & \Omega^{3(r-1)} & \hdots 
\end{pmatrix}\\ \label{ch4:defU-}
U^- &= \overline{U^+}.
\end{align}
Here $\overline{U^+}$ denotes the complex conjugate of $U^+$. Since $U^+$ is unitary we have that $\overline{U^+}$ can also be written as $(U^+)^{-t}$ where $t$ denotes transposition. We also define the $r\times r$ diagonal matrices
\begin{align} \label{ch4:eq:defDpm}
D^+ &=  \operatorname{diag}\left(1,\Omega^{\frac{1}{2}},-\Omega^{-\frac{1}{2}}, -\Omega^{\frac{2}{2}}, \Omega^{-\frac{2}{2}}, \Omega^{\frac{3}{2}}, -\Omega^{-\frac{3}{2}}, \ldots\right),\\
D^- &=  \operatorname{diag}\left(1,-\Omega^{-\frac{1}{2}},-\Omega^{\frac{1}{2}}, \Omega^{-\frac{2}{2}}, \Omega^{\frac{2}{2}},-\Omega^{-\frac{3}{2}}, -\Omega^{-\frac{3}{2}},\ldots\right).
\end{align}
\begin{definition} \label{ch4:def:X}
We define $X:\mathbb C\setminus \mathbb R\to\mathbb C^{(r+1)\times (r+1)}$ by the transformation 
\begin{align} \label{ch4:defX}
X(z) = Y(z) 
\left(r^\frac{r}{2r+2} \oplus r^{-\frac{1}{2r+2}} z^{\frac{r-1}{2 r}} \bigoplus_{j=1}^r z^{-\frac{j-1}{r}} \right)
\times \left\{\begin{array}{ll}
1 \oplus U^+ D^+, & \operatorname{Im}(z) >0,\\
1 \oplus U^- D^-, & \operatorname{Im}(z) <0.
\end{array}
\right.
\end{align}
\end{definition}

For clarity, we emphasize that $1\oplus U^\pm D^\pm$ is the direct sum of the $1\times 1$ block with component $1$ and the $r\times r$ block $U^\pm D^\pm$. Now $X$ satisfies

\begin{rhproblem} \label{ch4:RHPforX} \
\begin{description}
\item[RH-X1] $X$ is analytic on $\mathbb C\setminus \mathbb R$.
\item[RH-X2] $X$ has boundary values for $x\in (-\infty, 0)\cup (0,\infty)$.
\begin{align} \nonumber
X_+(x) &= X_-(x) \\ \nonumber
&\hspace{-0.5cm}\times \left\{\begin{array}{ll} 
\begin{pmatrix}
1 & w_{\beta}(x)\\
0 & 1
\end{pmatrix} \oplus
\bigoplus_{j=1}^{\frac{r}{2}-1} 
\begin{pmatrix}
0 & 1\\
-1 & 0
\end{pmatrix}
\oplus 1, & r \equiv 0 \mod 2,\\
\begin{pmatrix}
1 & w_{\beta}(x)\\
0 & 1
\end{pmatrix} \oplus
\bigoplus_{j=1}^{\frac{r-1}{2}} 
\begin{pmatrix}
0 & 1\\
-1 & 0
\end{pmatrix}
, & r \equiv 1 \mod 2,
\end{array}
\right. \\ \label{ch4:RHX2}
&\hspace{6.5cm} x > 0,\\ \nonumber
X_+(x) &= X_-(x)\\ \nonumber
&\hspace{-0.5cm}\times \left\{\begin{array}{ll}
1\oplus
\bigoplus_{j=1}^{\frac{r}{2}} 
\begin{pmatrix}
0 & 1\\
-1 & 0
\end{pmatrix}, & r \equiv 0 \mod 2,\\
1\oplus
\bigoplus_{j=1}^{\frac{r-1}{2}} 
\begin{pmatrix}
0 & 1\\
-1 & 0
\end{pmatrix}
\oplus 1, & r \equiv 1 \mod 2,\\
\end{array}
\right. \\ \label{ch4:RHX2b}
&\hspace{6.5cm} x < 0.
\end{align}
\item[RH-X3] As $|z|\to\infty$
\begin{align} \label{ch4:RHX3}
X(z) = \left(\mathbb{I}+\mathcal{O}\left(\frac{1}{z}\right)\right) 
\left(1 \oplus z^{\frac{r-1}{2 r}} 
\bigoplus_{j=1}^r z^{-\frac{j-1}{r}}\right) 
\times \left\{\begin{array}{ll}
z^n \oplus z^{-\frac{n}{r}} U^+ D^+
, & \operatorname{Im}(z) >0,\\
z^n \oplus z^{-\frac{n}{r}} U^- D^-
, & \operatorname{Im}(z) <0.
\end{array}
\right.
\end{align}
\item[RH-X4] As $z\to 0$
\begin{align}
\label{ch4:RHX4}
X(z) = \mathcal{O}\begin{pmatrix} 1 & z^{-\frac{r-1}{2r}} h_{\alpha+\frac{r-1}{r}}(z) & \ldots & z^{-\frac{r-1}{2r}} h_{\alpha+\frac{r-1}{r}}(z)\\ 
\vdots & & & \vdots\\
1 & z^{-\frac{r-1}{2r}} h_{\alpha+\frac{r-1}{r}}(z) & \ldots & z^{-\frac{r-1}{2r}} h_{\alpha+\frac{r-1}{r}}(z)\end{pmatrix}.
\end{align}
\end{description}
\end{rhproblem}

\begin{proof}
RH-X2 requires verification. As an intermediate step we define
\begin{align}
\label{ch4:defZ}
Z(z) = Y(z) 
\left(r^\frac{r}{2r+2} \oplus r^{-\frac{1}{2r+2}} z^{\frac{r-1}{2 r}} 
\bigoplus_{j=1}^r z^{-\frac{j-1}{r}}\right).
\end{align}
Then we have
\begin{align} \label{ch4:eq:defXasZ}
X(z) = Z(z) 
\times \left\{\begin{array}{ll}
1 \oplus U^+ D^+, & \operatorname{Im}(z) >0\\
1 \oplus U^- D^-, & \operatorname{Im}(z) <0.
\end{array}
\right.
\end{align}
$Z$ has a jump on $(0,\infty)$ and $(-\infty,0)$ which we denote by $J$. One easily checks that
\begin{align} \label{ch4:eq:jumpZ}
J = \left\{\begin{array}{ll}
\begin{pmatrix} 
1 & \frac{1}{\sqrt r} w_{\beta}(x) & \frac{1}{\sqrt r} w_{\beta}(x) & \hdots & \frac{1}{\sqrt r} w_{\beta}(x)\\
0 & 1 & 0 & \hdots & 0\\
0 & 0 & 1 & \hdots & 0\\
\vdots & & & \ddots & \vdots\\
0 & 0 & 0 & \hdots & 1
\end{pmatrix}, & x >0,\\
\begin{pmatrix} 
1 & 0 & 0 & \hdots & 0\\
0 & -\Omega^\frac{1}{2} & 0 & \hdots & 0\\
0 & 0 & -\Omega^\frac{1}{2} \Omega & \hdots & 0\\
\vdots & & & \ddots & \vdots\\
0 & 0 & 0 & \hdots & -\Omega^\frac{1}{2} \Omega^{r-1}
\end{pmatrix}, & x <0.
\end{array} \right.
\end{align}
We remind the reader, once more, that we use the convention that jump curves that touch the origin are oriented away from the origin. Let us now prove the jumps for $X$. In what follows we let the $(r+1)\times (r+1)$ matrices have indices ranging from $0$ to $r$, we make this choice because then we can label the entries of the $r\times r$ matrices $U^\pm$ and the $r\times r$ diagonal matrices $D^\pm$ with indices ranging from $1$ to $r$. Notice that the entries of $U^+$ can be written explicitly as follows.
\begin{align} \label{ch4:eq:U+-evenandodd}
\begin{array}{lll}
U^+_{i,2j} &= \frac{1}{\sqrt r}\Omega^{(i-1)j}, & i=1,2,\ldots,r; j=1,2,\ldots, \left\lfloor \frac{r}{2}\right\rfloor\\
U^+_{i,2j-1} &= \frac{1}{\sqrt r} \Omega^{-(i-1)(j-1)}, & i=1,2,\ldots,r; j=1,2,\ldots, \left\lfloor \frac{r+1}{2} \right\rfloor.
\end{array}
\end{align}
Let us check the jump of $X$ for $x>0$. By using the block form we notice, using \eqref{ch4:eq:defXasZ} and \eqref{ch4:eq:jumpZ}, that
\begin{align*} 
\left(X_-(x)^{-1} X_+(x)\right)_{00} &= \sum_{k,l=0}^r (1\oplus U^- D^-)^{-1}_{0k} J_{kl} (1\oplus U^+ D^+)_{l0}
= J_{00} = 1,\\ 
\left(X_-(x)^{-1} X_+(x)\right)_{01} &= \sum_{k,l=0}^r (1\oplus U^- D^-)^{-1}_{0k} J_{kl} (1\oplus U^+  D^+)_{l1}
= \sum_{l=1}^r J_{0l} U^+_{l1}= w_{b-a}(x).
\end{align*}
Here we used in the last line that $J_{0l}=\frac{1}{\sqrt r} w_{\beta}(x)$ and that $U^+_{l1}=\frac{1}{\sqrt r}$ for all $l=1,2,\ldots,r$. We have also used the diagonal form of the matrices $D^\pm$ in both lines. On the other hand, we have for $i=1,2,\ldots,r$ and $j=2,\ldots,r$ that
\begin{align*} 
\left(X_-(x)^{-1} X_+(x)\right)_{i0} &= \sum_{k,l=0}^r (1\oplus U^- D^-)^{-1}_{ik} J_{kl} (1\oplus U D^+)_{l0}\\
&= (D^-)^{-1}_{ii} (1\oplus U^-)^{-1}_{i0} J_{00} = 0,\\ 
\left(X_-(x)^{-1} X_+(x)\right)_{0j} &= \sum_{k,l=0}^r (1\oplus U^- D^-)^{-1}_{0k} J_{kl} (1\oplus U^+ D^+)_{lj}\\
&= D^+_{jj} \sum_{l=0}^r J_{0l} U^+_{lj} = D_{jj}^+ \frac{w_{\beta}(x)}{\sqrt r} \sum_{l=1}^r  U^+_{lj} = 0.
\end{align*}
For indices $i,j=1,2,\ldots,r$ we have that
\begin{align*} 
\left(X_-(x)^{-1} X_+(x)\right)_{ij} &= \sum_{k,l=0}^r (1\oplus U^- D^-)^{-1}_{ik} J_{kl} (1\oplus U^+ D^+)_{lj}\\ 
&= \sum_{k,l=1}^r (1\oplus U^- D^-)^{-1}_{ik} J_{kl} (1\oplus U^+ D^+)_{lj}\\
&= (D^-)^{-1}_{ii} D^+_{jj} \sum_{k=1}^r (U^+)^t_{ik} U^+_{kj}\\
&= (D^-)^{-1}_{ii} D^+_{jj} \sum_{k=1}^r U^+_{ki} U^+_{kj}.
\end{align*}
Here we have used that $(U^-)^{-1} = (U^+)^t$. Now we have four cases depending on the parity of $i$ and $j$. We will use \eqref{ch4:eq:U+-evenandodd} for all of them. Suppose that $i$ and $j$ are both odd. Then we can write $i=2A-1$ and $j=2B-1$. Thus we find
\begin{align*}
\sum_{k=1}^r U^+_{ki} U^+_{kj} = \frac{1}{r} \sum_{k=1}^r \Omega^{-(k-1)(A+B-2)}
= \left\{
\begin{array}{ll}
1, & i=j=1\\
0, & \text{otherwise.}
\end{array}
\right.
\end{align*}
When $i=j=1$ we indeed have $(D^-)^{-1}_{ii}D^+_{jj} = 1$. 
Now let us suppose that $i=2A$ and $j=2B$. Then we have
\begin{align*}
\sum_{k=1}^r U^+_{ki} U^+_{kj} = \frac{1}{r} \sum_{k=1}^r \Omega^{(k-1)(A+B)}
= \left\{
\begin{array}{ll}
1, & i=j=r\text{ and }r\equiv 0\mod 2\\
0, & \text{otherwise.}
\end{array}
\right.
\end{align*}
When we are in the situation of $i=j=r$ and $r$ is even we obtain $(D^-)^{-1}_{ii}D^+_{jj} = (-1)^\frac{r}{2} \Omega^{\frac{r}{4}} (-1)^{\frac{r}{2}-1}\Omega^{\frac{r}{4}} = -\Omega^{\frac{r}{2}}=1$, as we should have.
Lastly, but perhaps most importantly, we check the case where $i$ and $j$ have a different parity. Let us write $i=2A-1$ and $j=2B$. Then we have
\begin{align*}
\sum_{k=1}^r U^+_{ki} U^+_{kj} = \frac{1}{r} \sum_{k=1}^r \Omega^{(k-1)(B-A+1)}
= \left\{
\begin{array}{ll}
1, & i=j+1\\
0, & \text{otherwise.}
\end{array}
\right.
\end{align*}
When $i=j+1$ we have $$(D^-)^{-1}_{ii} D^+_{jj} = (-1)^{A-1} \Omega^{-\frac{A-1}{2}} (-1)^{B-1} \Omega^{\frac{B}{2}} = -\Omega^{\frac{j+1-i}{4}} = -1.$$
For $i=2A$ and $j=2B-1$ we get
\begin{align*}
\sum_{k=1}^r U^+_{ki} U^+_{kj} = \frac{1}{r} \sum_{k=1}^r \Omega^{(k-1)(A-B-1)}
= \left\{
\begin{array}{ll}
1, & i=j-1\\
0, & \text{otherwise.}
\end{array}
\right.
\end{align*}
When $i=j-1$ we have $$(D^-)^{-1}_{ii} D^+_{jj} = (-1)^A \Omega^{\frac{A}{2}} (-1)^{B-1} \Omega^{-\frac{B-1}{2}} = \Omega^{\frac{i-j+1}{4}} = 1.$$
Having found all the components of the jump matrix for $x>0$, we conclude that \eqref{ch4:RHX2} holds.\\

Let us now turn to the jump for $x<0$. Here we may ignore the $0$ indices altogether due to the particular block form of $J$. For $i,j=1,2,\ldots,r$ we have
\begin{align*}
\left(X_-(x)^{-1} X_+(x)\right)_{ij} &= (D^+)^{-1}_{ii} D^-_{jj}\sum_{k,l=1}^r (U^+)^{-1}_{ik} J_{kl} U^-_{lj}
= - (D^+)^{-1}_{ii} D^-_{jj} \Omega^\frac{1}{2} \sum_{k=1}^r U^-_{ki} \Omega^{k-1} U^-_{kj}.
\end{align*}
Again we treat the four cases for $i$ and $j$ depending on the parity. Let $i=2A-1$ and $j=2B-1$, then
\begin{align*}
\sum_{k=1}^r U^-_{ki} \Omega^{k-1} U^-_{kj} &= \frac{1}{r} \sum_{k=1}^r \Omega^{(k-1)(A+B-1)}
= \left\{
\begin{array}{ll}
1, & i=j=r\text{ and }r=1\mod 2\\
0, & \text{otherwise.}
\end{array}
\right.
\end{align*}
In the case that $i=j=r$ and $r$ odd we indeed have $$(D^+)^{-1}_{ii} D^-_{jj} = (-1)^\frac{r-1}{2} \Omega^{\frac{r-1}{4}} (-1)^\frac{r-1}{2} \Omega^{\frac{r-1}{4}} = -\Omega^{-\frac{1}{2}},$$ as it should be (it should exactly cancel the $-\Omega^\frac{1}{2}$ factor).\\ 
Let $i=2A$ and $j=2B$, then
\begin{align*}
\sum_{k=1}^r U^-_{ki} \Omega^{k-1} U^-_{kj} = \frac{1}{r} \sum_{k=1}^r \Omega^{-(k-1)(A+B-1)}
=0.
\end{align*}
Let $i=2A-1$ and $j=2B$, then
\begin{align*}
\sum_{k=1}^r U^-_{ki} \Omega^{k-1} U^-_{kj} = \frac{1}{r} \sum_{k=1}^r \Omega^{(k-1)(A-B)}
= \left\{
\begin{array}{ll}
1, & i=j-1\\
0, & \text{otherwise.}
\end{array}
\right.
\end{align*}
In the case that $i=j-1$ we have $$(D^+)^{-1}_{ii} D^-_{jj} = (-1)^{A-1} \Omega^{\frac{A-1}{2}} (-1)^B \Omega^{-\frac{B}{2}} = -\Omega^\frac{i-j+1}{4} = -\Omega^{-\frac{1}{2}}.$$
Let $i=2A$ and $j=2B-1$, then
\begin{align*}
\sum_{k=1}^r U^-_{ki} \Omega^{k-1} U^-_{kj} = \frac{1}{r} \sum_{k=1}^r \Omega^{(k-1)(B-A)}
= \left\{
\begin{array}{ll}
1, & i=j+1\\
0, & \text{otherwise.}
\end{array}
\right.
\end{align*}
In the case that $i=j+1$ we have $$(D^+)^{-1}_{ii} D^-_{jj} = (-1)^{A-1} \Omega^{-\frac{A}{2}} (-1)^{B-1} \Omega^{\frac{B-1}{2}} = \Omega^\frac{j-i-1}{4} = \Omega^{-\frac{1}{2}}.$$
We conclude that we get the jump on $x<0$ as in \eqref{ch4:RHX2b}.

RH-X4 follows from the observation that as $z\to 0$
\begin{align} \label{ch4:eq:RH-X4proof}
z^{-\frac{j}{r}} h_{\alpha+\frac{j}{r}}(z) = \mathcal O\left(z^{-\frac{r-1}{r}} h_{\alpha+\frac{r-1}{r}}(z)\right),
\hspace{2cm}j=0,1,\ldots,r-1.
\end{align} 
Indeed, we have as $z\to 0$ that
\begin{align*}
z^{-\alpha-\frac{j}{r}} h_{\alpha+\frac{j}{r}}(z) = h_{-\alpha-\frac{j}{r}}(z) &= \mathcal O\left(h_{-\alpha-\frac{r-1}{r}}(z)\right)
= \mathcal O\left(z^{-\alpha-\frac{r-1}{r}} h_{\alpha+\frac{r-1}{r}}(z)\right)
\end{align*}
and this leads to \eqref{ch4:eq:RH-X4proof} after multiplication by $z^\alpha$. 
\end{proof}

\section{Normalization and opening of the lens} \label{ch4:sec:normalization}

Our next task is to normalize the RHP. That is, we should eliminate the $z^n$ behavior in RH-X3, without making the jumps too cumbersome. As usual, we will use $g$-functions to perform this normalization. 

\subsection{Vector equilibrium problem and definition of the $g$-functions} \label{ch4:sec:normalization2}

In order to construct the $g$-functions we need a vector equilibrium problem corresponding to the MBE. Fortunately, this has been studied by Kuijlaars in detail in \cite{Ku}. According to \cite{Ku} we have a vector of measures $(\mu_0,\mu_1,\ldots,\mu_{r-1})$ that minimizes the energy functional
\begin{align} \label{ch4:eq:eqProblem}
\sum_{j=0}^{r-1} \iint \log \frac{1}{|x-y|} d\mu_j(x) d\mu_j(y)
- \sum_{j=0}^{r-2} \iint \log \frac{1}{|x-y|} d\mu_j(x) d\mu_{j+1}(y)
+ \int V(x) d\mu(x).
\end{align}
under the condition that $\mu_j$ has support in $[0,\infty)$ for even $j$ and $(-\infty,0]$ for odd $j$, and the condition that the total mass of the measures is given by
\begin{align} \label{ch4:eq:totalMass}
\mu_j(\operatorname{supp}(\mu_j)) = \frac{r-j}{r}.
\end{align} 
The main result of \cite{Ku} is then that $\mu_0$ coincides with the equilibrium measure $\mu_{V,\theta}^*$ from \eqref{ch4:defmuVtheta}. Furthermore, we have $\operatorname{sup}(\mu_j)=\Delta_{j}$, where
\begin{align} \label{ch4:eq:defDeltaj}
\Delta_j = \left\{\begin{array}{rl}
[0,q] \text{ for some }q>0, & j=0,\\
(-\infty,0], & j \equiv 1 \mod 2,\\ 
{}[0,\infty), & j \equiv 0 \mod 2 \text{ and }j\neq 0.
\end{array}\right.
\end{align}
That the support for $\mu_0$ is an interval $[0,q]$ is dictated by the one-cut $\frac{1}{r}$-regularity assumption on $V$. For convenience we also define $\mu_r=0$. 
In addition, the measures satisfy the following variational conditions.
\begin{align} \label{ch4:varCon1}
& 2\int \log \frac{1}{|s-x|} d\mu_0(s) - \int \log \frac{1}{|s-x|} d\mu_1(s) + V(x) 
\left\{\begin{array}{ll}
= -\ell, & x\in [0,q]\\ 
\geq -\ell, & x>q
\end{array}\right.\\ \label{ch4:varCon2}
&- \int \log \frac{1}{|s-x|} d\mu_{j-1}(s)+2\int \log \frac{1}{|s-x|} d\mu_j(s) 
- \int \log \frac{1}{|s-x|} d\mu_{j+1}(s) = 0, \quad j=1,\ldots,r-1.
\end{align}
Here $\ell$ is some real constant (that does not necessarily equal the one in \eqref{ch4:eq:varCon}). The derivation for all the properties in this section up to this point can be found in \cite{Ku}. As stated in the introduction, we make an additional assumption on the solution to the variational equations. Namely, we assume that the inequality in \eqref{ch4:eq:varCon} is strict, which is equivalent to the following assumption. 

\begin{assumption} \label{ch4:assump:varStrict}
The inequality in \eqref{ch4:varCon1} for $x>q$ is strict.
\end{assumption}

We now define the $g$-functions by
\begin{align} \label{ch4:eq:defgfunctions}
g_j(z) = \int_{\Delta_{j}} \log(z-s) d\mu_j(s), \quad j=0,1,\ldots,r.
\end{align}
For convenience we also put $g_{-1}=g_r=0$. We remind the reader that the logarithm is taken with the principle branch, as always. It follows immediately from the definition \eqref{ch4:eq:defgfunctions} that for real $x$
\begin{align}
\label{ch4:g0jump}
g_{0,+}(x) - g_{0,-}(x) &= 2\pi i \int_x^q d\mu_0(s), & x\in [0,q],
\end{align}
and 
\begin{align} \label{ch4:g0jump2}
g_{0,+}(x) - g_{0,-}(x) &\equiv 0 \mod 2\pi i, & x\in \mathbb R\setminus [0,q],\\
\label{ch4:g1jump}
g_{1,+}(x) - g_{1,-}(x) &= 0, & x\in \mathbb R\setminus \Delta_1,\\
\label{ch4:grjump2}
g_{r-1,+}(x) - g_{r-1,-}(x) &= ((-1)^r-1) \frac{\pi i}{r}, & x\in \mathbb R\setminus \Delta_{r-1}.
\end{align}
Similar formulae hold for the other $g$-functions but we will not need these. One can also deduce from the variational conditions that
\begin{align}
\label{ch4:g1g2jump}
& g_{0,-}(x)+g_{0,+}(x) - g_{1,+}(x) - V(x) 
\left\{\begin{array}{ll}
= \ell, & x\in [0,q]\\ 
< \ell, & x>q
\end{array}\right.\\ \label{ch4:g1g2jumpb}
&- g_{j-1,-}(x)+g_{j,-}(x)+g_{j,+}(x) - g_{j+1,+} 
= ((-1)^j-1) \frac{\pi i}{r}, \qquad j=1,\ldots,r-1; x\in\Delta_{j}. 
\end{align}
In particular, the equations on the negative real axis, i.e., for odd $j$, yield an equality with $-\frac{2\pi i}{r}$. 

\subsection{Asymptotic behavior of the $g$-functions}

In general one uses the matrix
\begin{align}
\label{ch4:defG}
G(z) = \bigoplus_{j=0}^{r} e^{n (g_{j-1}(z)-g_j(z))}
\end{align}
to normalize the RHP. Then we need to understand the asymptotics of the $g$-functions as $z\to\infty$. The following two propositions will provide these. 

\begin{proposition} \label{ch4:asympgfunctions}
For $a>0$ let
\begin{align} \label{ch4:eq:defma}
m_a = \int_0^q s^a d\mu_{V,\theta}^*(s).
\end{align}
As $z\to \infty$ we have
\begin{align} \label{ch4:eq:asympg0infty}
g_0(z) &= \log(z) + O\left(\frac{1}{z}\right),
\end{align}
and for $j=1,\ldots,r$ and $\pm \operatorname{Im}(z)>0$, we have
\begin{align} \label{ch4:eq:asympgjinfty}
g_{j-1}(z)-g_{j}(z) &= \frac{1}{r} \log(z) - \sum_{k=1}^{r-1} m_{\frac{k}{r}} \Omega^{\pm (-1)^{j}\lfloor \frac{j}{2}\rfloor k} z^{-\frac{k}{r}} + O\left(\frac{1}{z}\right).
\end{align}
\end{proposition}

\begin{proof}
For the asymptotics of $g_0$ we can use the compactness of the support of $d\mu_{V,\theta}^*$ to immediately conclude that
\[
g_0(z) = \int_{0}^q \left(\log(z) + \mathcal O(1/z))\right) d\mu_{V,\theta}^*(s) = \log(z) + \mathcal O(1/z),
\]
as $z\to\infty$. For the asymptotics of $g_j$ with $j=1,\ldots,r-1$ we will have to look at the specific construction of the measures as presented in \cite{Ku}. For a fixed $a>0$ one considers the rational function
\begin{align*} 
\Psi^a(w) = \frac{1}{r \Omega^{r-1} (w-a^\frac{1}{r})}, \quad\quad z=w^r,
\end{align*}
on an $r$-sheeted Riemann surface, that has cuts $\Delta_j$ (as defined in \eqref{ch4:eq:defDeltaj}, but excluding $\Delta_0$), which connect the $j$-th sheet to the $(j+1)$-st sheet, for $j=1,\ldots,r-1$. This is done in such a way that $w=z^\frac{1}{r}$ is taken with the principle branch on the first sheet. This uniquely determines the branches that we should take on the other sheets. Explicitly, we have for $j=1,\ldots,r$
\begin{align} \label{ch4:eq:defPsijExplicit}
\Psi^a_{j}(z) = \left\{\begin{array}{ll}
\displaystyle\frac{1}{r z^{1-\frac{1}{r}}(z^\frac{1}{r}- \Omega^{(-1)^{j}\lfloor \frac{j}{2}\rfloor} a^\frac{1}{r})}, & \operatorname{Im}(z)>0,\\
\displaystyle\frac{1}{r z^{1-\frac{1}{r}}(z^\frac{1}{r}- \Omega^{(-1)^{j-1}\lfloor \frac{j}{2}\rfloor} a^\frac{1}{r})}, & \operatorname{Im}(z)<0.
\end{array} \right.
\end{align}
We remind the reader that $z^{1-\frac{1}{r}}$ and $z^\frac{1}{r}$ are taken with the principle branch, as usual. Now, following \cite{Ku}, we construct some auxiliary measures $\nu^a_1, \ldots, \nu^a_{r-1}$ out of these, namely 
\begin{align} \label{ch4:eq:defAuxMeas}
d\nu^a_{j}(s) = \frac{\Psi^a_{j+}(s)-\Psi^a_{j-}(s)}{2\pi i} ds, \quad\quad s\in\Delta_j. 
\end{align}
It is a known fact from \cite{Ku} that the $d\nu^a_{j}$ are bonafide positive measures. 

By formula (3.7) of \cite{Ku} and the first formula in the proof of Proposition 3.2 of \cite{Ku} we then have for $j=1,\ldots,r-1$ that 
\begin{align} \label{ch4:explicitMeasures}
d\mu_j(s) =\int_0^q d\nu^a_j(s) d\mu_{V,\frac{1}{r}}^*(a).
\end{align}
We remind the reader that $(\mu_0,\mu_1,\ldots,\mu_{r})$ solves our vector equilibrium problem mentioned in the beginning of this section. We denote the Stieltjes transforms of the auxiliary measures by
\begin{align*}
F^a_j(z) = \int_{\Delta_j} \frac{d\nu^a_{j}(s)}{z-s}, \quad\quad j=1,\ldots,r-1.
\end{align*}
Then according to \cite{Ku} we have
\begin{align*}
F^a_1(z) &= \frac{1}{z-a} - \Psi^a_{1}(z)\\
F^a_{r-1}(z) &= \Psi^a_{r}(z)
\end{align*}
and, in particular, for $j=2,\ldots,r-1$
\begin{align} \label{ch4:eq:StieltjesIdj}
F^a_{j-1}(z) - F^a_j(z) = \Psi^a_j(z).
\end{align}
Using \eqref{ch4:explicitMeasures} and \eqref{ch4:eq:StieltjesIdj}, we see that for $j=2,\ldots,r-1$
\begin{align*}
\int_{\Delta_{j-1}} \frac{d\mu_{j-1}(s)}{z-s}  - \int_{\Delta_{j}} \frac{d\mu_{j}(s)}{z-s} 
&= \int_0^q \left(F^a_{j-1}(z) - F^a_{j}(z)\right) d\mu_{V,\frac{1}{r}}^*(a)
= \int_0^q \Psi^a_{j}(z) d\mu_{V,\frac{1}{r}}^*(a).
\end{align*}
Integrating this with respect to $z$, using the explicit formula in \eqref{ch4:eq:defPsijExplicit}, then yields for $\pm\operatorname{Im}(z)>0$
\begin{align*}
g_{j-1}(z) - g_{j}(z) &= \int_0^q \log(z^\frac{1}{r}- \Omega^{\pm (-1)^{j}\lfloor \frac{j}{2}\rfloor} a^\frac{1}{r}))
 d\mu_{V,\theta}^*(a)\\
 &= \frac{1}{r}\log(z) - \sum_{k=1}^\infty \frac{1}{k} \Omega^{\pm (-1)^{j}\lfloor \frac{j}{2}\rfloor k} z^{-\frac{k}{r}} \int_0^q a^\frac{k}{r} d\mu_{V,\theta}^*(a)\\
 &= \frac{1}{r}\log(z) - \sum_{k=1}^{r-1} \frac{1}{k} \Omega^{\pm (-1)^{j}\lfloor \frac{j}{2}\rfloor k} z^{-\frac{k}{r}} m_{\frac{k}{r}} + \mathcal O\left(\frac{1}{z}\right)
\end{align*}
as $z\to\infty$. A similar reasoning will prove the case for $j=r$.
\end{proof}

\begin{definition} \label{ch4:def:Cn}
We define the $r\times r$ upper-triangular matrix
\begin{align} \label{ch4:eq:defCn}
C_n = 
\begin{pmatrix}
1 & a_1 & a_2 & a_3 & \cdots & a_{r-1}\\
0 & 1 & a_1 & a_2 & \cdots & a_{r-2}\\
0 & 0 & 1 & a_1 & \cdots & a_{r-3}\\
\vdots & & & \ddots & \ddots\\
0 & 0 & 0 & 0 & 1 & a_1\\
0 & 0 & 0 & 0 & \cdots & 1
\end{pmatrix},
\end{align}
where, with $m_\frac{1}{r}, m_\frac{2}{r}, \ldots, m_\frac{r-1}{r}$ as in \eqref{ch4:eq:defma}, we take
\begin{align}
a_j = \sum_{l=1}^j  \frac{(-n)^l}{l!} \sum_{k_1+\ldots+k_l=j} m_\frac{k_1}{r} \cdots m_\frac{k_l}{r}.
\end{align}
\end{definition}

\begin{lemma} \label{ch4:prop:Cn}
With $C_n$ as in Definition \ref{ch4:def:Cn} we have
\begin{align} \label{ch4:eq:prop:Cn}
z^{-\frac{n}{r}} \left(\bigoplus_{j=1}^r z^{-\frac{j-1}{r}}\right) 
U^\pm D^\pm
\bigoplus_{j=1}^{r} e^{n(g_{j-1}(z)-g_j(z))}
= \left(C_n
+\mathcal O(1/z)\right)
\left(\bigoplus_{j=1}^r z^{-\frac{j-1}{r}}\right)
U^\pm D^\pm
\end{align}
as $z\to\infty$, for $\pm \operatorname{Im}(z)>0$.
\end{lemma}

\begin{proof}
So let $\pm\operatorname{Im}(z)>0$. Notice that we may omit the $D^\pm$ factors in what follows, due to their diagonal form. It follows from \eqref{ch4:eq:asympgjinfty}, and some straightforward algebra, that
\begin{align}
z^{-\frac{n}{r}} e^{n(g_1(z)-g_2(z))} = a_0 + a_1 z^{-\frac{1}{r}} + \ldots + a_{r-1} z^{-\frac{r-1}{r}}+ \mathcal O\left(\frac{1}{z}\right)
\end{align}
as $z\to\infty$, where $a_0=1$ and $a_1,\ldots,a_{r-1}$ are as in Definition \ref{ch4:def:Cn}. The components of $G$ further down the diagonal have a similar expansion but with (effectively) the properly chosen branch of the $z^{-\frac{1}{r}}$ term, namely
\begin{align}\label{ch4:eq:Gdiagm}
z^{-\frac{n}{r}} 
\bigoplus_{j=1}^{r} e^{n(g_{j-1}(z)-g_j(z))} 
= \sum_{m=0}^{r-1} a_m z^{-\frac{m}{r}}
\left(\bigoplus_{j=1}^r \Omega^{\pm (-1)^{j} \lfloor \frac{j}{2}\rfloor} \right)^m
+\mathcal O\left(\frac{1}{z}\right)
\end{align}
as $z\to\infty$. Again, this follows from \eqref{ch4:eq:asympgjinfty}. Let's look at $m=1$ first. A simple calculation shows that
\begin{align} \label{ch4:eq:U+-omegaU}
U^\pm \left(\bigoplus_{j=1}^r \Omega^{\pm (-1)^{j} \lfloor \frac{j}{2}\rfloor} \right) (U^\pm)^{-1}
= \begin{pmatrix}
0 & 1 & 0 & \hdots & 0\\
0 & 0 & 1 & \hdots & 0\\
\vdots & & & \ddots & \vdots\\
0 & 0 & 0 &  \hdots & 1\\
1 & 0 & 0 &  \hdots & 0
\end{pmatrix}.
\end{align}
For the other powers $m$ in \eqref{ch4:eq:Gdiagm} we would have to take a power of this cyclic permutation matrix. 
Now from \eqref{ch4:eq:U+-omegaU} we arrive at
\begin{multline*}
z^{-\frac{1}{r}}
\left(\bigoplus_{j=1}^r z^{-\frac{j-1}{r}}\right)
U^\pm \left(\bigoplus_{j=1}^r \Omega^{\pm (-1)^{j} \lfloor \frac{j}{2}\rfloor} \right) (U^\pm)^{-1}
\left(\bigoplus_{j=1}^r z^{-\frac{j-1}{r}}\right)^{-1}\\
= \begin{pmatrix}
0 & 1 & 0 & \hdots & 0\\
0 & 0 & 1 & \hdots & 0\\
\vdots & & & \ddots & \vdots\\
0 & 0 & 0 &  \hdots & 1\\
z^{-1} & 0 & 0 &  \hdots & 0
\end{pmatrix}
= \begin{pmatrix}
0 & 1 & 0 & \hdots & 0\\
0 & 0 & 1 & \hdots & 0\\
\vdots & & & \ddots & \vdots\\
0 & 0 & 0 &  \hdots & 1\\
0 & 0 & 0 &  \hdots & 0
\end{pmatrix}
+\mathcal O\left(\frac{1}{z}\right)
\end{multline*}
as $z\to\infty$. An analogous argument works for the other powers $m$ in \eqref{ch4:eq:Gdiagm} and we obtain that
\begin{multline*}
z^{-\frac{n}{r}}   \left(\bigoplus_{j=1}^r z^{-\frac{j-1}{r}}\right) U^\pm
\bigoplus_{j=1}^{r} e^{n(g_{j-1}(z)-g_j(z))}
(U^\pm)^{-1} \left(\bigoplus_{j=1}^r z^{-\frac{j-1}{r}}\right)^{-1}\\ 
=  \begin{pmatrix}
1 & a_1 & a_2 & a_3 & \cdots & a_{r-1}\\
0 & 1 & a_1 & a_2 & \cdots & a_{r-2}\\
0 & 0 & 1 & a_1 & \cdots & a_{r-3}\\
\vdots & & & \ddots & \ddots\\
0 & 0 & 0 & 0 & 1 & a_1\\
0 & 0 & 0 & 0 & \cdots & 1
\end{pmatrix}
+\mathcal O\left(\frac{1}{z}\right)
\end{multline*}
as $z\to\infty$ and, after rearranging the factors and reinserting the factors $D^\pm$, we obtain \eqref{ch4:eq:prop:Cn} with $C_n$ as in \eqref{ch4:eq:defCn} from Definition \ref{ch4:def:Cn}. 
\end{proof}

We will also need the asymptotics of the $g$-functions as $z\to 0$. 
\begin{proposition} \label{ch4:prop:gfunctionsbounded0}
The $g$-functions are bounded near the origin.
\end{proposition}
\begin{proof}
We claim that for all $j=0,1,\ldots,r$
\begin{align} \label{ch4:eq:behavmujs}
\frac{d\mu_j(s)}{ds} &= \mathcal O\left(s^{-\frac{r}{r+1}}\right), & s\to 0. 
\end{align}
For $j=0$ this is nothing else than what the one-cut $\frac{1}{r}$-regularity of $V$ prescribes. So let us consider $j=1,\ldots,r-1$.  From the previous proof we recall (see \eqref{ch4:eq:defPsijExplicit}, \eqref{ch4:eq:defAuxMeas} and \eqref{ch4:explicitMeasures}) that
\begin{align} \label{ch4:eq:intintint}
\frac{d\mu_j(s)}{ds} &= \int_0^q \frac{\Psi_{j+}^a(s)-\Psi_{j-}^a(s)}{2\pi i} d\mu_{V,\frac{1}{r}}^*(a).
\end{align}
It follows from \eqref{ch4:eq:behavmujs} for $j=0$ that there exists a $0<\delta<\min(1,q)$ and a $c>0$ such that for $s\in(0,\delta]$
\begin{align} \label{ch4:eq:dmuVbehav0}
\frac{d\mu_{V,\frac{1}{r}}^*(s)}{ds} \leq c s^{-\frac{r}{r+1}}.  
\end{align}
Now we rewrite \eqref{ch4:eq:intintint} as 
\begin{align} \label{ch4:eq:intintint2}
\frac{d\mu_j(s)}{ds}= \int_0^{\delta/s} \frac{\Psi_{j+}^{s a}(s)-\Psi_{j-}^{s a}(s)}{2\pi i} d\mu_{V,\frac{1}{r}}^*(s a)
+ \int_\delta^q \frac{\Psi_{j+}^a(s)-\Psi_{j-}^a(s)}{2\pi i} d\mu_{V,\frac{1}{r}}^*(a).
\end{align}
By some straightforward algebra we find for $a\in (0,\delta]$ and $s>0$ the estimate
\begin{align} \label{ch4:eq:Psiestimate}
\left|\frac{\Psi_{j+}^{s a}(s)-\Psi_{j-}^{s a}(s)}{2\pi i}\right| &= \frac{|\operatorname{Im}(\Omega^{(-1)^{j}\lfloor \frac{j}{2}\rfloor})| a^\frac{1}{r}}{\pi s} 
\left|1-\Omega^{(-1)^{j}\lfloor \frac{j}{2}\rfloor} a^\frac{1}{r}\right|^{-2} 
\leq \frac{a^{\frac{1}{r}}}{(1-\delta^\frac{1}{r})^2 \pi s}.
\end{align}
Then it follows from \eqref{ch4:eq:dmuVbehav0} and \eqref{ch4:eq:Psiestimate} that
\begin{align*}
\left|\int_0^{\delta/s} \frac{\Psi_{j+}^{s a}(s)-\Psi_{j-}^{s a}(s)}{2\pi i} d\mu_{V,\frac{1}{r}}^*(s a)\right|
&\leq \int_0^{\delta/s} \frac{a^{\frac{1}{r}}}{(1-\delta^\frac{1}{r})^2 \pi s} c s^{-\frac{r}{r+1}} a^{-\frac{r}{r+1}} s da\\
&= \frac{r^2+r}{2r+1}\frac{\delta^{\frac{1}{r}+\frac{1}{r+1}} c}{(1-\delta^\frac{1}{r})^2 \pi } s^{\frac{1}{r}+\frac{1}{r+1}-\frac{r}{r+1}}.
\end{align*}
Using the one-cut $\frac{1}{r}$-regularity of $V$, but now for the behavior around $q$, we can show that the remaining integral in the right-hand side of \eqref{ch4:eq:intintint2} is bounded. We conclude that
\begin{align} \label{ch4:eq:behavmujs2}
\frac{d\mu_j(s)}{ds} = \mathcal O\left(s^{\frac{1}{r}+\frac{1}{r+1}-\frac{r}{r+1}}\right) + \mathcal O(1)
\end{align}
as $s\to 0$, and this is even better than \eqref{ch4:eq:behavmujs}. Notice that we cannot ignore the $\mathcal O(1)$ term in the case $r=2$. The claim is proved. Plugging \eqref{ch4:eq:behavmujs} in the definition \eqref{ch4:eq:defgfunctions} of the $g$-functions, we find with standard arguments that the $g$-functions are bounded near the origin.
\end{proof}

\subsection{Normalization $X\mapsto T$}

For the next two transformations of our RHP it will turn out to be convenient to define the following function $\varphi$. 
\begin{align} \label{ch4:eq:defvarphi}
\varphi(z) = -g_0(z)+\frac{1}{2} g_{1}(z) + \frac{1}{2}\left(V(z)+\ell\right).
\end{align}
Due to our assumption that $V$ is real analytic on $[0,\infty)$ we know that there exists an open neighborhood $O_V$ of $[0,\infty)$ to which $V$ has an analytic continuation. Thus $\varphi$ is analytic on $O_V\setminus (-\infty,q]$. By \eqref{ch4:g1g2jump} we have for $x>q$ that
\begin{align} \label{ch4:eq:varphiqinfty}
\varphi(x) = -\frac{1}{2}\left(g_{0,+}(x)+g_{0,-}(x)-g_{1,+}(x)- V(x)-\ell\right) >0
\end{align}
and also by \eqref{ch4:g1g2jump} we have for $x\in(0,q)$ that
\begin{align} \label{ch4:eq:varphi0q}
\varphi_\pm(x) = \mp \pi i \int_x^q d\mu_0(s). 
\end{align}

\begin{definition} \label{ch4:def:T}
With $G$ as in \eqref{ch4:defG} and $C_n$ as in Definition \ref{ch4:def:Cn}, we define
\begin{align}
\label{ch4:defT}
T(z) = L^{-1} (1\oplus C_n^{-1}) X(z) G(z) L,
\end{align}
where
\begin{align} \label{ch4:eq:defL}
L = \operatorname{diag}(e^{n \frac{r\ell}{r+1}}, e^{-n \frac{\ell}{r+1}},\ldots, e^{-n \frac{\ell}{r+1}}).
\end{align}
\end{definition}

Then $T$ satisfies the following RHP. 

\begin{rhproblem} \label{ch4:RHPforT} \
\begin{description}
\item[RH-T1] $T$ is analytic on $\mathbb C\setminus \mathbb R$.
\item[RH-T2] $T$ has boundary values for $x\in (0,-\infty)\cup (0,q) \cup (q,\infty)$.
\begin{align} \nonumber
T_+(x) &= T_-(x)\\ \nonumber
&\hspace{-1.7cm}\times \left\{\begin{array}{ll} 
\begin{pmatrix}
e^{2n\varphi_+(x)} & x^{\beta}\\
0 & e^{2n\varphi_-(x)}
\end{pmatrix} \oplus
\bigoplus_{j=1}^{\frac{r}{2}-1} 
\begin{pmatrix}
0 & 1\\
-1 & 0
\end{pmatrix}
\oplus 1, & r \equiv 0 \mod 2,\\
\begin{pmatrix}
e^{2n\varphi_+(x)} & x^{\beta}\\
0 & e^{2n\varphi_-(x)}
\end{pmatrix} \oplus
\bigoplus_{j=1}^{\frac{r-1}{2}} 
\begin{pmatrix}
0 & 1\\
-1 & 0
\end{pmatrix}
, & r \equiv 1 \mod 2,
\end{array}
\right. \\
&\hspace{6cm} x\in(0,q),\\ \nonumber
T_+(x) &= T_-(x)\\
&\hspace{-1cm}\times \left\{\begin{array}{ll} 
\begin{pmatrix}
1 & x^{\beta} e^{-2n\varphi(x)}\\
0 & 1
\end{pmatrix} \oplus
\bigoplus_{j=1}^{\frac{r}{2}-1} 
\begin{pmatrix}
0 & 1\\
-1 & 0
\end{pmatrix}
\oplus 1, & r \equiv 0 \mod 2,\\
\begin{pmatrix}
1 & x^{\beta} e^{-2n\varphi(x)}\\
0 & 1
\end{pmatrix} \oplus
\bigoplus_{j=1}^{\frac{r-1}{2}} 
\begin{pmatrix}
0 & 1\\
-1 & 0
\end{pmatrix}
, & r \equiv 1 \mod 2,
\end{array}
\right. \\ \nonumber
&\hspace{6cm} x>q,\\ \nonumber
T_+(x) &= T_-(x) \hspace{0.1cm}\times\hspace{0.1cm}
\left\{\begin{array}{ll}
1 \oplus
\bigoplus_{j=1}^{\frac{r}{2}} 
\begin{pmatrix}
0 & 1\\
-1 & 0
\end{pmatrix}, & r \equiv 0 \mod 2,\\
1\oplus
\bigoplus_{j=1}^{\frac{r-1}{2}} 
\begin{pmatrix}
0 & 1\\
-1 & 0
\end{pmatrix}
\oplus 1, & r \equiv 1 \mod 2,\\
\end{array}
\right. \\
&\hspace{6cm} x<0.
\end{align}
\item[RH-T3] As $|z|\to\infty$
\begin{align} \label{ch4:RHT3}
T(z) = \left(\mathbb{I}+\mathcal{O}\left(\frac{1}{z}\right)\right) 
\left(1 \oplus z^{\frac{r-1}{2 r}} 
\bigoplus_{j=1}^r z^{-\frac{j-1}{r}}\right)
\times\left\{\begin{array}{ll}
1 \oplus U^+ D^+
, & \operatorname{Im}(z) >0,\\
1 \oplus U^- D^-
, & \operatorname{Im}(z) <0.
\end{array}
\right.
\end{align}
\item[RH-T4] As $z\to 0$
\begin{align}
\label{ch4:RHT4}
T(z) = \mathcal{O}\begin{pmatrix} 1 & z^{-\frac{r-1}{2r}} h_{\alpha+\frac{r-1}{r}}(z) & \ldots & z^{-\frac{r-1}{2r}} h_{\alpha+\frac{r-1}{r}}(z)\\ 
\vdots & & & \vdots\\
1 & z^{-\frac{r-1}{2r}} h_{\alpha+\frac{r-1}{r}}(z) & \ldots & z^{-\frac{r-1}{2r}} h_{\alpha+\frac{r-1}{r}}(z)\end{pmatrix}.
\end{align}
\end{description}
\end{rhproblem}

\begin{proof}
We prove RH-T2 for $r\equiv 1\mod 2$. 
For $x>0$ we see, using RH-X2, that
\begin{multline*}
T_-(x)^{-1} T_+(x)\\
= L^{-1} 
\bigoplus_{j=0}^r e^{n (g_{j,-}(x)-g_{j-1,-}(x))}
\left(\begin{pmatrix}
1 & w_{\beta}(x)\\
0 & 1
\end{pmatrix} \oplus
\bigoplus_{j=1}^{\frac{r-1}{2}} 
\begin{pmatrix}
0 & 1\\
-1 & 0
\end{pmatrix}\right)\\
\bigoplus_{j=0}^r e^{n (g_{j-1,+}(x)-g_{j,+}(x))} L\\
=
\begin{pmatrix}
e^{n(g_{0,-}(x)-g_{0,+}(x))} & x^\beta e^{n (-V(x) + g_{0,-}(x) + g_{0,+}(x) - g_{1,+}(x) - \ell)}\\
0 & e^{n(-g_{0,-}(x) + g_{0,+}(x)+g_{1,-}(x)-g_{1,+}(x))}
\end{pmatrix}\\ 
\oplus
\bigoplus_{j=1}^{\frac{r-1}{2}} 
\begin{pmatrix}
0\hspace{1.7cm} e^{n(g_{2j,-}(x) - g_{2j-1,-}(x) + g_{2j,+}(x) - g_{2j+1,+}(x))}\\
-e^{n(g_{2j+1,-}(x)-g_{2j,-}(x)+g_{2j-1,+}(x)-g_{2j,+}(x))}  \hspace{1.7cm}0
\end{pmatrix}
\end{multline*}
Now RH-T2, both for $x\in(0,q)$ and for $x>q$, is a consequence of \eqref{ch4:eq:varphi0q}, \eqref{ch4:eq:defvarphi}, \eqref{ch4:g0jump}, \eqref{ch4:g0jump2}, \eqref{ch4:g1jump}, \eqref{ch4:g1g2jumpb} and the complex conjugated version of the latter. 

Let us now look at the jump for $x<0$. Similarly as before, we have
\begin{multline*}
T_-(x)^{-1} T_+(x)
= e^{n(g_{0,-}(x)-g_{0,+}(x))}\\
\oplus
\bigoplus_{j=1}^{\frac{r-1}{2}} 
\begin{pmatrix}
0 \hspace{1.7cm} e^{n(g_{2j-1,-}(x) - g_{2j-2,-}(x) + g_{2j-1,+}(x) - g_{2j,+}(x))}\\
-e^{n(g_{2j,-}(x)-g_{2j-1,-}(x)+g_{2j-2,+}(x)-g_{2j-1,+}(x))} \hspace{1.7cm} 0
\end{pmatrix}\\
\oplus e^{n(g_{r-1,-}(x)-g_{r-1,+}(x))}
\end{multline*}
Now RH-T2 follows from \eqref{ch4:g0jump2}, \eqref{ch4:grjump2}, \eqref{ch4:g1g2jumpb} and the complex conjugated version of the latter. Here we are using that $n$ is divisible by $r$, and thus $e^{-\frac{2\pi i}{r} n}=1$. The case where $r\equiv 0\mod 2$ is analogous. 

Let us now prove RH-T3. Using \eqref{ch4:eq:asympg0infty}, Lemma \ref{ch4:prop:Cn} and RH-X3 we have as $z\to\infty$ that
\begin{align*}
T(z) &= L^{-1} \left(1\oplus C_n^{-1}\right) \left(\mathbb I + \mathcal O\left(\frac{1}{z}\right)\right) \\
&\hspace{1cm}\left((e^{-n\frac{r\ell}{r+1}}+\mathcal O(1/z))\oplus (C_n+\mathcal O(1/z)) z^\frac{r-1}{2r} 
\left(\bigoplus_{j=1}^r z^{-\frac{j-1}{r}}\right) U^\pm D^\pm e^{n\frac{\ell}{r+1}}\right)\\
&= L^{-1} \left(1\oplus C_n^{-1}\right) L \left(\mathbb I + \mathcal O\left(\frac{1}{z}\right)\right) 
\left(1\oplus C_n z^\frac{r-1}{2r} \left(\bigoplus_{j=1}^r z^{-\frac{j-1}{r}}\right) U^\pm D^\pm\right)\\
&=\left(\mathbb I + \mathcal O\left(\frac{1}{z}\right)\right)  \left(1\oplus C_n^{-1}\right) 
\left(1\oplus C_n z^\frac{r-1}{2r} \left(\bigoplus_{j=1}^r z^{-\frac{j-1}{r}}\right) U^\pm D^\pm\right)
\end{align*}
for $\pm\operatorname{Im}(z)>0$ and we obtain RH-T3. 

RH-T4 follows from the fact that the $g$-functions are bounded around $z=0$ (see Proposition \ref{ch4:prop:gfunctionsbounded0}). 
\end{proof}

When $n$ is not divisible by $r$ we can nevertheless arrive at the same RHP by using an additional transformation, see Appendix \ref{ch:appendixA}. From this point onwards the RHPs do not depend on the particular modulo class $r$ that $n$ is in. 

\subsection{Opening of the lens $T\mapsto S$}

We will open a lens from $0$ to $q$ with the $\varphi$-function as defined in \eqref{ch4:eq:defvarphi}, as usual. 
We denote the upper and lower lips of this lens by $\Delta_0^+$ and $\Delta_0^-$ respectively. The direction of the lips of the lens near the origin is perpendicular to the real line for now (see Figure \ref{ch4:FigS}), but later on, in Section \ref{ch4:sec:defLocalP}, we will slightly deform the lips. 

\begin{figure}[t]
\begin{center}
\resizebox{10.5cm}{6cm}{
\begin{tikzpicture}[>=latex]
	\draw[-] (-7,0)--(7,0);
	\draw[-] (-3,-4)--(-3,4);
	\draw[fill] (-3,0) circle (0.1cm);
	\draw[fill] (3,0) circle (0.1cm);
	
	\node[above] at (-3.25,0) {\large 0};	
	\node[above] at (3.1,0.05) {\large $q$};	
	\node[above] at (-3.5,1.8) {\large $\Delta_0^+$};	
	\node[below] at (-3.5,-1.8) {\large $\Delta_0^-$};	
	\node[above] at (-5,0.05) {\large $\Delta_1$};	
	\node[above] at (0,0.05) {\large $\Delta_0$};
	
	\draw[->, ultra thick] (-3,0) -- (-3,1.2);
	\draw[-, ultra thick] (-3,0) -- (-3,2);
	
	\draw[->, ultra thick] (-3,0) -- (-3,-1.2);
	\draw[-, ultra thick] (-3,0) -- (-3,-2);
	
	\draw[->, ultra thick] (3,0) -- (5,0);
	\draw[->, ultra thick] (-3,0) -- (0,0);
	\draw[-, ultra thick] (-3,0) -- (7,0);
	
	\draw[->, ultra thick] (0,0) -- (-5,0);
	\draw[-, ultra thick] (-7,0) -- (-3,0);
	
	\draw[-, ultra thick] (-3,2) to [out=0, in=135] (3,0);
	\draw[-, ultra thick] (-3,-2) to [out=0, in=225] (3,0);

	\draw[->, ultra thick] (-0.9,1.88) to (-0.7,1.85);
	\draw[->, ultra thick] (-0.9,-1.88) to (-0.7,-1.85);
	
\end{tikzpicture}
}
\caption{Contour $\Sigma_{S} = \mathbb{R} \cup \Delta_{0}^{\pm}$ for the RHP for $S$. The
	lens around $\Delta_0$ is contained in the domain $O_V$ where $V$ is analytic. \label{ch4:FigS}}
\end{center}
\end{figure}
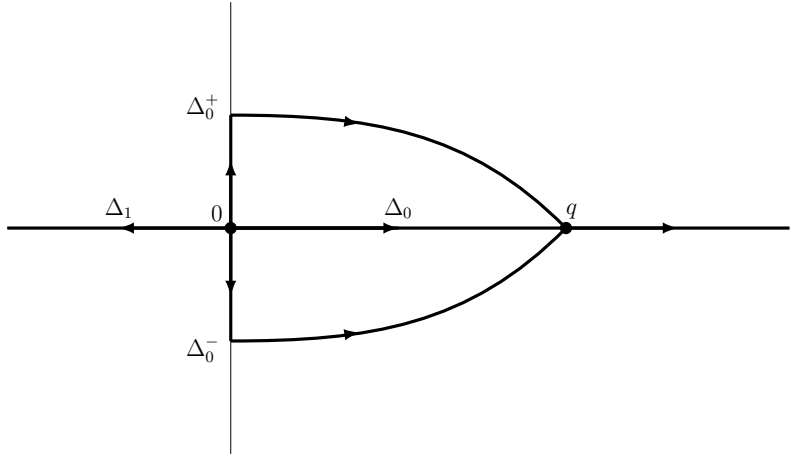

\begin{definition} \label{ch4:def:S}
$S$ is defined by
\begin{align*}
S(z) &= T(z) \left(\begin{pmatrix} 1 & 0\\ -z^{-\beta} e^{2 n \varphi(z)} & 1\end{pmatrix} \oplus \mathbb I_{(r-1)\times (r-1)}\right),
&z \text{ in the upper part of the lens}\\
S(z) &= T(z) \left(\begin{pmatrix} 1 & 0\\ z^{-\beta} e^{2 n \varphi(z)} & 1\end{pmatrix} \oplus \mathbb I_{(r-1)\times (r-1)}\right),
&z \text{ in the lower part of the lens}\\
S(z) &= T(z), &\text{elsewhere.}
\end{align*}
\end{definition}

Then $S$ satisfies the following RHP. 

\begin{rhproblem} \label{ch4:RHPforS} \
\begin{description}
\item[RH-S1] $S$ is analytic on $\mathbb C\setminus \Sigma_S$.
\item[RH-S2] $S$ has boundary values for $x\in \Sigma_S$. 
\begin{align} \nonumber
S_+(x) &= S_-(x)\\ \nonumber
&\hspace{-0.5cm}\times \left\{\begin{array}{ll} 
\begin{pmatrix}
0 & x^{\beta}\\
-x^{-\beta} & 0
\end{pmatrix} \oplus
\bigoplus_{j=1}^{\frac{r}{2}-1} 
\begin{pmatrix}
0 & 1\\
-1 & 0
\end{pmatrix}
\oplus 1, & r \equiv 0 \mod 2,\\
\begin{pmatrix}
0 & x^{\beta}\\
-x^{-\beta} & 0
\end{pmatrix} \oplus
\bigoplus_{j=1}^{\frac{r-1}{2}} 
\begin{pmatrix}
0 & 1\\
-1 & 0
\end{pmatrix}
, & r \equiv 1 \mod 2,
\end{array}
\right. \\
&\hspace{6cm} x \in (0,q),\\  \nonumber
S_+(x) &= S_-(x)\\  \nonumber
&\hspace{-1.1cm}\times \left\{\begin{array}{ll} 
\begin{pmatrix}
1 & x^{\beta} e^{-2n\varphi(x)}\\
0 & 1
\end{pmatrix} \oplus
\bigoplus_{j=1}^{\frac{r}{2}-1} 
\begin{pmatrix}
0 & 1\\
-1 & 0
\end{pmatrix}
\oplus 1, & r \equiv 0 \mod 2,\\
\begin{pmatrix}
1 & x^{\beta} e^{-2n\varphi(x)}\\
0 & 1
\end{pmatrix} \oplus
\bigoplus_{j=1}^{\frac{r-1}{2}} 
\begin{pmatrix}
0 & 1\\
-1 & 0
\end{pmatrix}
, & r \equiv 1 \mod 2,
\end{array}
\right. \\
&\hspace{6cm} x>q,\\ \nonumber
S_+(x) &= S_-(x) \hspace{0.1cm}\times
\left\{\begin{array}{ll}
1 \oplus
\bigoplus_{j=1}^{\frac{r}{2}} 
\begin{pmatrix}
0 & 1\\
-1 & 0
\end{pmatrix}, & r \equiv 0 \mod 2,\\
1\oplus
\bigoplus_{j=1}^{\frac{r-1}{2}} 
\begin{pmatrix}
0 & 1\\
-1 & 0
\end{pmatrix}
\oplus 1, & r \equiv 1 \mod 2,\\
\end{array}
\right. \\
&\hspace{6cm} x < 0,\\  \nonumber
S_+(z) &= S_-(z)  
\begin{pmatrix}
1 & 0\\
z^{-\beta} e^{2 n \varphi(z)} & 1
\end{pmatrix} 
\oplus \mathbb I_{(r-1)\times (r-1)}\\
&\hspace{6cm} z \in \Delta_0^\pm.
\end{align}
\item[RH-S3] As $|z|\to\infty$
\begin{align} \label{ch4:RHS3}
S(z) = \left(\mathbb{I}+\mathcal{O}\left(\frac{1}{z}\right)\right) 
\left(1 \oplus z^{\frac{r-1}{2 r}} 
\bigoplus_{j=1}^r z^{-\frac{j-1}{r}}\right)
\times\left\{\begin{array}{ll}
1 \oplus U^+ D^+
, & \operatorname{Im}(z) >0,\\
1 \oplus U^- D^-
, & \operatorname{Im}(z) <0.
\end{array}
\right.
\end{align}
\item[RH-S4] As $z\to 0$
\begin{align}
\nonumber
S(z) &= \mathcal{O}\begin{pmatrix} z^{-\frac{r-1}{2r}} h_{\alpha+\frac{r-1}{r}}(z) & \ldots & z^{-\frac{r-1}{2r}} h_{\alpha+\frac{r-1}{r}}(z)\\ 
\vdots & & \vdots\\
z^{-\frac{r-1}{2r}} h_{\alpha+\frac{r-1}{r}}(z) & \ldots & z^{-\frac{r-1}{2r}} h_{\alpha+\frac{r-1}{r}}(z)\end{pmatrix}\\ \label{ch4:RHS4a}
&\hspace{4cm} \text{for $z$ to the right of }\Delta_0^\pm\\
\nonumber
S(z) &= \mathcal{O}\begin{pmatrix} 1 & z^{-\frac{r-1}{2r}} h_{\alpha+\frac{r-1}{r}}(z) & \ldots & z^{-\frac{r-1}{2r}} h_{\alpha+\frac{r-1}{r}}(z)\\ 
\vdots & & & \vdots\\
1 & z^{-\frac{r-1}{2r}} h_{\alpha+\frac{r-1}{r}}(z) & \ldots & z^{-\frac{r-1}{2r}} h_{\alpha+\frac{r-1}{r}}(z)\end{pmatrix}\\ \label{ch4:RHS4b}
&\hspace{4cm} \text{for $z$ to the left of }\Delta_0^\pm
\end{align}
\end{description}
\end{rhproblem}
We omit a proof, since the procedure is standard.

\section{Global parametrix}

The jump matrix on $\Delta_0^\pm$ and $(q,\infty)$ will tend to the unit matrix as $n\to\infty$. For $z\in\Delta_0^\pm$ we have that $\operatorname{Re} \varphi(z)<0$ if the distance of the lips of the lens to $(0,q)$ is small enough (we have the freedom to make this distance as small as we want),  one can argue this with the Cauchy-Riemann equations and \eqref{ch4:eq:varphi0q}, in the usual way. Thus there are no jumps on $\Delta_0^\pm$ in the limit that $n\to\infty$. On $(q,\infty)$ one may use \eqref{ch4:eq:varphiqinfty} to see that the upper-left $2\times 2$ block tends to the unit matrix as $n\to\infty$. The global parametrix problem is the problem where we replace the jumps on these contours by their large $n$ limit. It should be a good approximation of $S$ away from the end points $0$ and $q$ for large $n$.

\subsection{The global parametrix problem}

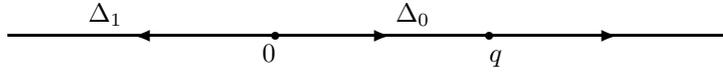
\begin{figure}[h]
\centering
\begin{picture}(300,50)(10,30)
\thicklines
\put(20,50){\line(1,0){270}}
\put(120,50){\circle*{3}}
\put(200,50){\circle*{3}}
\put(70,50){\vector(-1,0){3}}
\put(160,50){\vector(1,0){3}}
\put(245,50){\vector(1,0){3}}
\put(115,40){$0$}
\put(200,40){$q$}
\put(50,55){$\Delta_1$}
\put(165,55){$\Delta_0$}
\end{picture}
\caption{Contour $\Sigma_{N}$ for the RHP for $N$. \label{ch4:FigN}}
\end{figure}

The global parametrix problem takes the following form. 

\begin{rhproblem} \label{ch4:RHPforN} \
\begin{description}
\item[RH-N1] $N$ is analytic on $\mathbb C\setminus \mathbb R$.
\item[RH-N2] $N$ has boundary values for $x\in \Sigma_N$ (see Figure \ref{ch4:FigN}), and we have the jumps
\begin{align} \nonumber
N_+(x) &= N_-(x)\\ \nonumber
&\hspace{-0.5cm}\times \left\{\begin{array}{ll} 
\begin{pmatrix}
0 & x^{\beta}\\
-x^{-\beta} & 0
\end{pmatrix} \oplus
\bigoplus_{j=1}^{\frac{r}{2}-1} 
\begin{pmatrix}
0 & 1\\
-1 & 0
\end{pmatrix}
\oplus 1, & r \equiv 0 \mod 2,\\
\begin{pmatrix}
0 & x^{\beta}\\
-x^{-\beta} & 0
\end{pmatrix} \oplus
\bigoplus_{j=1}^{\frac{r-1}{2}} 
\begin{pmatrix}
0 & 1\\
-1 & 0
\end{pmatrix}
, & r \equiv 1 \mod 2,
\end{array}
\right. \\ 
&\hspace{6cm} x \in (0,q),\\  \nonumber
N_+(x) &= N_-(x)\\ \nonumber
&\hspace{0.4cm}\times \left\{\begin{array}{ll} 
\begin{pmatrix}
1 & 0\\
0 & 1
\end{pmatrix} \oplus
\bigoplus_{j=1}^{\frac{r}{2}-1} 
\begin{pmatrix}
0 & 1\\
-1 & 0
\end{pmatrix}
\oplus 1, & r \equiv 0 \mod 2,\\
\begin{pmatrix}
1 & 0\\
0 & 1
\end{pmatrix} \oplus
\bigoplus_{j=1}^{\frac{r-1}{2}} 
\begin{pmatrix}
0 & 1\\
-1 & 0
\end{pmatrix}
, & r \equiv 1 \mod 2,
\end{array}
\right. \\ 
&\hspace{6cm} x > q,\\ \nonumber
N_+(x) &= N_-(x) \\ \nonumber
&\hspace{1.5cm}\times \left\{\begin{array}{ll}
1 \oplus
\bigoplus_{j=1}^{\frac{r}{2}} 
\begin{pmatrix}
0 & 1\\
-1 & 0
\end{pmatrix}, & r \equiv 0 \mod 2,\\
1\oplus
\bigoplus_{j=1}^{\frac{r-1}{2}} 
\begin{pmatrix}
0 & 1\\
-1 & 0
\end{pmatrix}
\oplus 1, & r \equiv 1 \mod 2,\\
\end{array}
\right. \\
&\hspace{6cm} x < 0.
\end{align}
\item[RH-N3] As $|z|\to\infty$
\begin{align} \label{ch4:RHN3}
N(z) = \left(\mathbb{I}+\mathcal{O}\left(\frac{1}{z}\right)\right) 
\left(1 \oplus z^{\frac{r-1}{2 r}} 
\bigoplus_{j=1}^r z^{-\frac{j-1}{r}}\right)
\times \left\{\begin{array}{ll}
1 \oplus U^+ D^+
, & \operatorname{Im}(z) >0,\\
1 \oplus U^- D^-
, & \operatorname{Im}(z) <0,
\end{array}
\right.
\end{align}
\end{description}
\end{rhproblem}

We have some freedom in our choice of the behavior of $N$ as $z\to 0$ and $z\to q$.\\

In \cite{KuMo} we were lucky in that we could use a minor modification of the global parametrix from \cite{KuMFWi}. For $r>2$, we really have to solve the global parametrix problem. \\

\begin{remark} In what follows, we opt to find a solution to RH-N with an appropriate algebraic equation and corresponding Riemann surface (e.g., see \cite{KuMFWi}). There is an alternative method though. After making the jumps constant with Szeg\H{o} functions, one can use the method with differentials introduced by Kuijlaars and Mo in \cite{KuMo} (see \cite{KuLo} for the larger size case, and in particular on how to obtain the correct asymptotic behavior as $z\to\infty$). The latter has the advantage that technical calculations can mostly be avoided. There is a trade-off however, in that the formulae seem to get more explicit in the first method. In the end however, all that really matters to us, is how $N$ behaves near $0$ and $q$ (which we describe in RH-N4 later). The choice for the first method is one of personal preference. 
\end{remark}

\subsection{An $r+1$ sheeted Riemann surface}

First we will solve RH-N for $\beta=0$. We need an $r+1$ sheeted Riemann surface. To find out what Riemann surface will help us, we let the associated vector equilibrium problem for linear external fields guide us. Let $\mu_0',\mu_1',\ldots,\mu_{r-1}'$ be the equilibrium measures of \eqref{ch4:eq:eqProblem} corresponding to a lineair external field $x$ (see \cite[Proposition 5.1]{Ku}). Using  \cite[Theorem 1.8]{ClRo} for $\theta=r$ and external field $x^r$ we see after a little calculation that $\mu_0'$ has support $[0,\frac{2r}{r+1}]$. For this we used the duality between $\theta=r$ and $\theta=\frac{1}{r}$ of the MBE. 
We consider the Stieltjes transforms
\begin{align*}
F_j(z) &= \displaystyle \int_{\Delta_j'} \frac{d\mu_j'(s)}{z-s}, & j=1,2,\ldots,r,
\end{align*}
where $\Delta_0'=[0,\frac{2r}{r+1}]$ and $\Delta_j'=\Delta_j$ as in \eqref{ch4:eq:defDeltaj} for $j=1,\ldots,r$. 
From these we construct a function on a Riemann surface as follows
\begin{align*}
\zeta(z) = \left\{\begin{array}{ll}
\zeta_0(z) = 1 - F_1(z), & z\in \mathfrak R_0\\
\zeta_j(z) = F_j(z) - F_{j+1}(z), & z\in \mathfrak R_j, j=1,\ldots,r-1,\\
\zeta_r(z) = F_r(z), & z\in \mathfrak R_r.
\end{array}\right.
\end{align*}
It is known from \cite[Proposition 5.1]{Ku} that $\zeta$ defines a meromorphic function from the $r+1$ sheeted Riemann surface with cut $\Delta_j'$ between sheet $\mathfrak R_j$ and sheet $\mathfrak R_{j+1}$, where $j=0,1,\ldots,r-1$, to the extended complex plane. We can get the asymptotic behavior as $z\to\infty$ immediately from \text{Proposition \ref{ch4:asympgfunctions}}, by taking the derivative (linear external fields are one-cut $\frac{1}{r}$-regular). Namely, we have as $z\to \infty$ 
\begin{align} \label{ch4:eq:zeta0AsympInfty}
\zeta_0(z) &= 1-\frac{1}{z} + O\left(\frac{1}{z^2}\right),
\end{align}
and for $j=1,\ldots,r$ and $\pm \operatorname{Im}(z)>0$, we have
\begin{align} \label{ch4:eq:zetajAsympInfty}
\zeta_{j}(z) &= \frac{1}{r z} 
\left(1+\sum_{k=1}^{r-1} k m'_{\frac{k}{r}} \Omega^{\pm (-1)^{j}\lfloor \frac{j}{2}\rfloor k} z^{-\frac{k}{r}} + O\left(\frac{1}{z}\right)\right),
\end{align}
where $m_\frac{1}{r}',\ldots,m_\frac{r-1}{r}'$ correspond to the external field $x$, the accent is added to avoid confusion with $m_\frac{1}{r},\ldots,m_\frac{r-1}{r}$ from Proposition \ref{ch4:asympgfunctions} corresponding to our general external field $V$.\\

\begin{proposition} \label{ch4:prop:zetaAsympInfty}
$\zeta$ satisfies the algebraic equation 
\begin{align} \label{ch4:eq:AlgeqforZeta}
\zeta^{r+1} = \left(\zeta-\frac{1}{r z}\right)^r.
\end{align}
\end{proposition}

\begin{proof}
We can read off from the asymptotic behaviors \eqref{ch4:eq:zeta0AsympInfty} and \eqref{ch4:eq:zetajAsympInfty} that $\zeta$ has a zero of order $r$ at infinity. The equilibrium measure $\mu_0'$ behaves as $d_1 s^{-\frac{r}{r+1}}\left(1+\mathcal O\left(x^\frac{1}{r+1}\right)\right)$ as $s\to 0^+$ for some constant $d_1>0$. To see this, one combines Theorem 1.8 and Remark 1.9 from \cite{ClRo} with the duality between $\theta=r$ and $\theta=\frac{1}{r}$ of the MBE. 
It follows from the behavior of the equilibrium measure that $\zeta_0(z) \sim - (r+1) d_1 z^{-\frac{r}{r+1}}$ as $z\to 0$. This implies that $\zeta$ has a pole of order $r$ at $z=0$, since $z=0$ is a branch point of order $r$ of the Riemann surface. We conclude that there cannot be any other poles. 
From the asymptotic behaviors \eqref{ch4:eq:zeta0AsympInfty} and \eqref{ch4:eq:zetajAsympInfty} we also read off that as $z\to\infty$
\begin{align} \label{ch4:zzzzz}
\zeta_0(z)+\zeta_1(z)+\ldots+\zeta_r(z) &= 1+\mathcal O\left(\frac{1}{z^{2}}\right)\\
\zeta_0(z)\zeta_1(z) + \zeta_0(z)\zeta_2(z) + \ldots + \zeta_{r-1}(z)\zeta_r(z) &= \frac{1}{z} + \mathcal O\left(\frac{1}{z^{2}}\right)\\
\zeta_0(z)\zeta_1(z)\zeta_2(z)+\zeta_0(z)\zeta_1(z)\zeta_3(z) + \ldots + \zeta_{r-2}(z)\zeta_{r-1}(z)\zeta_r(z)
&= \frac{1}{r} \binom{r}{2} \frac{1}{z^2} + \mathcal O\left(\frac{1}{z^{3}}\right)\\ \nonumber
&\vdots\\ \label{ch4:zzzzz2}
\zeta_0(z)\zeta_1(z)\cdots \zeta_r(z) &= \frac{1}{r^{r}} \binom{r}{r} \frac{1}{z^r} + \mathcal O\left(\frac{1}{z^{r+1}}\right).
\end{align}
No calculation is needed to argue that there are only integer powers of $z$. Each formula in \eqref{ch4:zzzzz}-\eqref{ch4:zzzzz2} represents an elementary symmetric polynomial, and is thus invariant under any permutation of $\zeta_0,\ldots,\zeta_r$. Then the expressions in \eqref{ch4:zzzzz}-\eqref{ch4:zzzzz2} do not have jumps, and the asymptotic behaviors can only contain integer powers of $z$. Additionally, it implies that \eqref{ch4:zzzzz}-\eqref{ch4:zzzzz2} represent meromorphic functions in the full complex plane with only a possible pole at $z=0$. It follows from \eqref{ch4:eq:behavmujs}, with external field $x$, that we have for $j=1,\ldots,r$ that
\begin{align*}
\zeta_j(z) = \mathcal O\left(z^{-\frac{r-1}{r}}\right)
\end{align*}
as $z\to 0$. Then it actually follows that the $\mathcal O$ terms in \eqref{ch4:zzzzz}-\eqref{ch4:zzzzz2} vanish. We conclude that
\begin{align*}
\prod_{j=0}^r (\zeta - \zeta_j(z)) &= \zeta^{r+1}+\sum_{k=0}^{r} \zeta^{j} \frac{(-1)^{r+1-j}}{r^{r-j}} \binom{r}{j} \frac{1}{z^{r-j}}
= \zeta^{r+1} - \left(\zeta-\frac{1}{r z}\right)^r
\end{align*}
and we arrive at \eqref{ch4:eq:AlgeqforZeta}.
\end{proof}

We mention that \eqref{ch4:eq:AlgeqforZeta} is in agreement with (2.15) from \cite{KuMo}  when $r=2$. 

To solve the global parametrix problem, it is convenient to modify $\zeta$. Namely, we define
\begin{align} \label{ch4:eq:defXi}
\xi_j(z) &= 1 - \frac{1}{r z \zeta_j\left(\displaystyle\frac{c_q}{r} z\right)}, & c_q = \frac{r+1}{2 q}, \hspace{1cm} j=0,1,\ldots,r.
\end{align}
Then $\xi$ is a meromorphic function on the $(r+1)$-sheeted Riemann surface with cuts $\Delta_0,\ldots,\Delta_r$ as in \eqref{ch4:eq:defDeltaj}. Proposition \ref{ch4:prop:zetaAsympInfty} immediately implies the following algebraic equation for $\xi$. 

\begin{corollary}
$\xi$ satisfies
\begin{align} \label{ch4:eq:xiAlg}
c_q z = \frac{1}{(1-\xi) \xi^r}.
\end{align}
\end{corollary}

It is a direct consequence of the asymptotic behaviors of $\zeta_0,\ldots,\zeta_r$ in \eqref{ch4:eq:zeta0AsympInfty} and \eqref{ch4:eq:zetajAsympInfty} that as $z\to\infty$
\begin{align} \label{ch4:eq:behavInftyXi1}
\xi_0(z) &= 1 - \frac{r}{c_q z}+\mathcal O\left(\frac{1}{z^2}\right),\\ \label{ch4:eq:behavInftyXi2}
\xi_j(z) &= 
m'_{\frac{1}{r}} \displaystyle\left(\frac{r}{c_q}\right)^\frac{1}{r} \Omega^{\pm (-1)^{j}\lfloor \frac{j}{2}\rfloor} z^{-\frac{1}{r}} + \mathcal O\left(z^{-\frac{2}{r}}\right)
\end{align}
for $\pm\operatorname{Im}(z)>0$. In principle, we could write down the expansion of $\xi_j(z)$ completely in powers of $z^{-1/r}$ but we will not need it. The equation \eqref{ch4:eq:behavInftyXi2} illuminates why we chose the particular modification $\zeta\to \xi$ as in \eqref{ch4:eq:defXi}, we want $\xi_1,\ldots,\xi_r$ to correspond intuitively to (a multiple of) $z^{-\frac{1}{r}}$ for large $z$. This will be key to getting the asymptotics of the global parametrix as $z\to\infty$ right, as shall be clear shortly. 

From \eqref{ch4:eq:xiAlg} we may also deduce the asymptotic behavior as $z\to 0$. Namely, for $j=0,1,\ldots,r$, we have as $z\to 0$ that
\begin{align} \label{ch4:eq:behav0Xi}
\xi_j(z) = \mathcal O\left(z^{-\frac{1}{r+1}}\right).
\end{align}
We shall be more precise about the behavior around the origin in a moment. 

\subsection{Properties of $\xi$}

We prove some properties for $\xi$ that will be practical later on. 
\begin{lemma} \label{ch4:prop:propXi}
For $j=0,1,\ldots,r$ we have
\begin{itemize}
\item[(i)] $\xi_{j,\pm}(z) = \xi_{j+1,\mp}(z)$ for all $z\in \Delta_j$ (and $j<r$). 
\item[(ii)] $\xi_j(\overline{z}) = \overline{\xi_j(z)}$ for all $z\in\mathbb C\setminus \Delta_j$. 
\item[(iii)] $\operatorname{sgn}\operatorname{Im}(\xi_j(z)) = (-1)^j\operatorname{sgn}\operatorname{Im}(z)$ for all $z\in \mathbb C\setminus \mathbb R$.
\end{itemize}
Furthermore, we have
\begin{align*}
\xi_0((q,\infty))=\left(\frac{r}{r+1},1\right)\quad \text{and}\quad \xi_0((-\infty,0))=(1,\infty).
\end{align*} 
\end{lemma}

\begin{proof}
$\xi$ inherits its jumps from $\zeta$ (though with a rescaled cut on the first and second sheet), hence (i) holds trivially. 

Let us prove (ii). By complex conjugating both \eqref{ch4:eq:xiAlg} and $z$ we find that
\begin{align*} 
c_q z = \frac{1}{(1-\overline{\xi(\overline{z})}) \overline{\xi(\overline{z})}^r}. 
\end{align*}
Then on any small neighborhood in, say, the upper half-plane we find a number $\sigma(j)\in\{0,1,\ldots,r\}$ such that
\begin{align*}
\overline{\xi_j(\overline{z})} = \xi_{\sigma(j)}(z). 
\end{align*}
By analytic continuation this must hold on the entire upper half-plane. Of course, a similar argument works in the lower half-plane. Then it follows from either \eqref{ch4:eq:behavInftyXi1} or \eqref{ch4:eq:behavInftyXi2} that this can only be true if $\sigma(j)=j$ for all $j=0,1,\ldots,r$. The equality extends to $\mathbb C\setminus \Delta_j$ by continuity. 

Let us prove (iii). First we prove that the sign of $\operatorname{Im}(\xi_j(z))$ is fixed in the upper half-plane and the lower half-plane. Suppose that $z_1$ and $z_2$ are two points in the upper half-plane such that $\operatorname{Im}(\xi_j(z_1))\neq\operatorname{Im}(\xi_j(z_2))$. Then, by continuity, there must exist a $z_3$ in the segment $[z_1,z_2]$, and in the upper half-plane in particular, such that $\operatorname{Im}(\xi_j(z_3))=0$. By (ii) this means that $\xi_j(z_3) = \xi_j(\overline{z_3})$. This is a contradiction since it would mean that $\xi$ is not an isomorphism (it has to be, since $\zeta$ is). We conclude that the sign of $\operatorname{Im}(\xi_j(z))$ is fixed in the upper half-plane. Of course a similar argument applies to the lower half-plane. 
Let us now view the case $j=0$. By \eqref{ch4:eq:behavInftyXi1} it should hold that $\operatorname{Im}(\xi_0(z))$ and $\operatorname{Im}(z)$ have the same sign for large $z$. Since the sign is fixed in the upper half-plane and the lower half-plane, as we just proved, this actually holds for all $z\in \mathbb C\setminus\mathbb R$. Thus we proved (iii) for $j=0$. The cases $j>0$ now simply follow from the cuts in (i). 

For the remaining part of the proposition we first prove that $\xi_0(q)=\frac{r}{r+1}$. It follows from (ii) and the bijectivity of $\xi$ that $\xi_0$ and $\xi_r$ are the only functions amongst $\xi_0,\ldots,\xi_r$ that can attain real values. Let us look at the function in the right-hand side of \eqref{ch4:eq:xiAlg}. It has a local minimum at $\xi = \frac{r}{r+1}$ and this is its only extremum. Since our map $\xi$ is bijective this local minimum must correspond to an endpoint of either $\Delta_0$ or $\Delta_r$. It follows from the asymptotic behaviors \eqref{ch4:eq:behavInftyXi1}, \eqref{ch4:eq:behavInftyXi2} and \eqref{ch4:eq:behav0Xi} that it cannot correspond to either $\infty$ or $0$. We must conclude that $\xi_0(q) = \frac{r}{r+1}$.
Now using this, \eqref{ch4:eq:behavInftyXi1} and the fact that $\xi_0((q,\infty))\subset \mathbb R$ we infer that $\xi_0((q,\infty)) = (\frac{r}{r+1},1)$. Then we cannot have $\xi_0((-\infty,0))=(-\infty,1)$ and we must conclude that $\xi_0((-\infty,0))=(1,\infty)$. For this, we also used \eqref{ch4:eq:behav0Xi}. 
\end{proof}

\begin{corollary} \label{ch4:cor:behavxi0}
Let $j=0,1,\ldots,r$. As $z\to 0$ we have
\begin{align} \label{ch4:eq:behav0Xi2}
\xi_j(z)
= \left\{\begin{array}{ll}  \displaystyle\left(\frac{r}{c_q}\right)^\frac{1}{r+1} \omega^{(-1)^{j}(\lfloor \frac{j+1}{2}\rfloor + \frac{1}{2})}z^{-\frac{1}{r+1}} + \mathcal O\left(z^{-\frac{2}{r+1}}\right), & \operatorname{Im}(z)>0,\\
\displaystyle\left(\frac{r}{c_q}\right)^\frac{1}{r+1} \omega^{(-1)^{j-1}(\lfloor \frac{j+1}{2}\rfloor + \frac{1}{2})} z^{-\frac{1}{r+1}} + \mathcal O\left(z^{-\frac{2}{r+1}}\right), & \operatorname{Im}(z)<0.
\end{array} \right.
\end{align}
\end{corollary}

We remind the reader that $\omega$ is as in \eqref{ch4:defab}. 

\begin{proof}
Since $\xi_0((-\infty,0))=(1,\infty)$ we must have, using \eqref{ch4:eq:xiAlg}, that
$$\xi_0(z) = \left(\frac{r}{c_q}\right)^\frac{1}{r+1}(-x)^{-\frac{1}{r+1}}$$
as $x\to 0^-$. Thus we have
$\xi_0(z) = \omega^{\pm \frac{1}{2}} z^{-\frac{1}{r+1}}$
as $z\to 0$ for $\pm\operatorname{Im}(z)>0$. The behaviors for $\xi_1,\ldots,\xi_r$ follow from \ref{ch4:prop:propXi}(i). 
\end{proof}

\subsection{Solution of the global parametrix for $\beta=0$}

We first find a solution to the global parametrix problem for $\beta=0$. 

\begin{theorem} \label{ch4:thm:globalParam0}
For $\beta=0$ the global parametrix problem is solved by
\begin{align} \label{ch4:eq:defN0}
N_0(z) = 
\begin{pmatrix}
p_0(\xi_0(z)) F(\xi_0(z)) & p_0(\xi_1(z)) F(\xi_1(z)) & \cdots & p_0(\xi_r(z)) F(\xi_r(z))\\
p_1(\xi_0(z)) F(\xi_0(z)) & p_1(\xi_1(z)) F(\xi_1(z)) & \cdots & p_1(\xi_r(z)) F(\xi_r(z))\\
\vdots & & & \vdots\\
p_r(\xi_0(z)) F(\xi_0(z)) & p_r(\xi_1(z)) F(\xi_1(z)) & \cdots & p_r(\xi_r(z)) F(\xi_r(z))
\end{pmatrix},
\end{align}
where $p_0,\ldots,p_r$ are unique polynomials of (at most) degree $r$ and where
\begin{align} \label{ch4:eq:defF}
F(\xi) = \frac{1}{\sqrt{(r+1)\xi^r-r \xi^{r-1}}}.
\end{align}
In the definition of $F$ the square root is taken to have the $r$ cuts $\xi_{0,+}(\Delta_0), \xi_{1,+}(\Delta_1), \ldots, \xi_{r-1,+}(\Delta_{r-1})$ and it is positive for large positive values of $\xi$. 
\end{theorem}

\begin{proof} We start by proving that RH-N3 is satisfied. It turns out to be convenient to write the asymptotics in terms of $\xi_j(z)$ rather than $z$, for $j=1,\ldots,r$. Namely, it follows from the algebraic equation \eqref{ch4:eq:xiAlg} and the asymptotics \eqref{ch4:eq:behavInftyXi2} that 
\begin{align} \label{ch4:eq:xitoz1}
c_q^\frac{k}{r} \xi_j^k (1-\xi_j)^\frac{k}{r} = 
\left\{\begin{array}{rl} 
\Omega^{(-1)^{j}\lfloor \frac{j}{2}\rfloor k}z^{-\frac{k}{r}}, & \operatorname{Im}(z)>0,\\
\Omega^{(-1)^{j-1}\lfloor \frac{j}{2}\rfloor k}z^{-\frac{k}{r}}, & \operatorname{Im}(z)<0,
\end{array}\right.
\end{align} 
for large enough $z$. With large enough $z$ we mean that we should have $|\xi_j(z)|<1$(see \eqref{ch4:eq:behavInftyXi2}) so that $(1-\xi_j)^\frac{k}{r}$ is well-defined. Then by \eqref{ch4:eq:xitoz1} we have for $\pm\operatorname{Im}(z)>0$ that
\begin{align} \nonumber
\Omega^{\pm (-1)^{j}\lfloor \frac{j}{2}\rfloor k} z^{-\frac{k}{r}} &= c_q^\frac{k}{r} \xi^k \sum_{m=0}^\infty \binom{\frac{k}{r}}{m} (-1)^m \xi^m\\ \label{ch4:eq:zktoxi}
&= c_q^\frac{k}{r} \sum_{m=k}^{r-1} \binom{\frac{k}{r}}{m-k} (-1)^{m-k} \xi^m + \mathcal O(\xi^r).
\end{align}
as $z\to \infty$. We will use this expansion in a moment. 

We know that $\xi_0(q)=\frac{r}{r+1}$ (see Lemma \ref{ch4:prop:propXi}). We may decompose $F$ as 
\begin{align} \label{ch4:eq:decompF}
F(\xi) = \frac{1}{\sqrt{r+1}} \frac{1}{\sqrt{\xi-\xi_0(q)}}  (-1)^{\sigma(\xi)} \xi^\frac{1-r}{2},
\end{align}
where the square root function has the cut $\xi_{0,+}(\Delta_0)$, is positive for large positive values of $\xi$, and, as usual, $\xi^\frac{1-r}{2}$ is taken with the principle branch (when $r\equiv 0\mod 2$). $\sigma(\xi)$ counts the number of cuts among $\xi_{1,+}(\Delta_1), \ldots, \xi_{r-1,+}(\Delta_{r-1})$ that have been touched if we go with a circular arc from $|\xi|$ to $\xi$, let's say that this arc is in the upper half-plane if $\operatorname{Im}(\xi)>0$ and in the lower half-plane if $\operatorname{Im}(\xi)<0$. We then have $\sigma(\xi)=0$ for $\xi$ in the lower half-plane, because $\xi_{1,+}(\Delta_1), \ldots, \xi_{r-1,+}(\Delta_{r-1})$ are all in the upper half-plane by Lemma \ref{ch4:prop:propXi}(c). We find that for $j=1,\ldots,r$
\begin{align} \label{ch4:eq:1-1jsigma}
\sigma(\xi_j(z)) = \left\{\begin{array}{rl}
\frac{1+(-1)^j}{2} (j-1), & \operatorname{Im}(z)>0,\\
\frac{1-(-1)^j}{2} (j-1), & \operatorname{Im}(z)<0.
\end{array}\right.
\end{align}
Now using \eqref{ch4:eq:behavInftyXi2} and \eqref{ch4:eq:1-1jsigma} we find for $j=1,\ldots,r$ that
\begin{align*}
F(\xi_j(z)) = i\frac{r^\frac{r-1}{2r} c_q^{-\frac{r-1}{2r}}}{\sqrt r (m_\frac{1}{r}')^\frac{r-1}{2}}\frac{1}{\sqrt{1-\frac{r+1}{r}\xi_j(z)}}  (-1)^{\sigma(j)+\left\lfloor\frac{j}{2}\right\rfloor} \Omega^{\pm (-1)^{j} \frac{1}{2}\left\lfloor\frac{j}{2}\right\rfloor} z^{-\frac{r-1}{2 r}} 
\left(1 + \mathcal O\left(z^{-\frac{1}{r}}\right)\right)
\end{align*}
for $\pm\operatorname{Im}(z)>0$ as $z\to\infty$. The expression between brackets on the far right can be written as
\begin{align*}
f\left(\Omega^{\pm (-1)^{j}\left\lfloor\frac{j}{2}\right\rfloor} z^{-\frac{1}{r}}\right),
\end{align*}
for some function $f$ with $f(0)=1$ that is analytic around $0$, and does not depend on $j$, which is clear from \eqref{ch4:eq:zetajAsympInfty} and \eqref{ch4:eq:defXi}. Then there exists a function $h$ with $h(0)\neq 0$ that is analytic around $0$, and does not depend on $j$, such that  for $j=1,\ldots,r$
\begin{align} \label{ch4:eq:Fxiasymph}
F(\xi_j(z)) = (-1)^{\sigma(j)+\left\lfloor\frac{j}{2}\right\rfloor} \Omega^{\pm (-1)^{j} \frac{1}{2}\left\lfloor\frac{j}{2}\right\rfloor} z^{-\frac{r-1}{2 r}}  h(\xi_j(z))
\end{align}
for $\pm\operatorname{Im}(z)>0$ as $z\to\infty$. One can verify that the power of $-1$ in \eqref{ch4:eq:Fxiasymph} follows exactly the patern of the matrix $D^\pm$ as defined in $\eqref{ch4:eq:defDpm}$. We conclude that
\begin{align} \label{ch4:eq:asympFxi}
F(\xi_j(z)) = D^\pm_{jj} z^{\frac{r-1}{2 r}} h(\xi_j(z))
\end{align}
as $z\to\infty$, for $j=1,\ldots,r$. 

Now, considering only the lower-right $r\times r$ block, RH-N3 demands that
\begin{multline*}
\begin{pmatrix}
p_1(\xi_1(z)) & p_1(\xi_2(z)) & \hdots & p_1(\xi_r(z))\\
p_2(\xi_1(z)) & p_2(\xi_2(z)) & \hdots & p_2(\xi_r(z))\\
\vdots & & & \vdots\\
p_r(\xi_1(z)) & p_r(\xi_2(z)) & \hdots & p_r(\xi_r(z))
\end{pmatrix}
\bigoplus_{j=1}^r F(\xi_j(z))\\
= \left(\mathbb I +\mathcal O\left(\frac{1}{z}\right)\right) z^{\frac{r-1}{2 r}}
\left(\bigoplus_{j=1}^r z^{-\frac{j-1}{r}}\right) U^\pm D^\pm.
\end{multline*}
as $z\to\infty$ for $\pm\operatorname{Im}(z)>0$. Using \eqref{ch4:eq:asympFxi} we see that this implies that we should have
\begin{multline*}
\left(\mathbb I +\mathcal O\left(\frac{1}{z}\right)\right)
\begin{pmatrix}
p_1(\xi_1(z)) & p_1(\xi_2(z)) & \hdots & p_1(\xi_r(z))\\
p_2(\xi_1(z)) & p_2(\xi_2(z)) & \hdots & p_2(\xi_r(z))\\
\vdots & & & \vdots\\
p_r(\xi_1(z)) & p_r(\xi_2(z)) & \hdots & p_r(\xi_r(z))
\end{pmatrix}\\
= i \left(\bigoplus_{j=1}^r z^{-\frac{j-1}{r}}\right) U^\pm
\left(\bigoplus_{j=1}^r h(\xi_j(z))\right)^{-1}.
\end{multline*}
as $z\to\infty$ for $\pm\operatorname{Im}(z)>0$. Looking in the $k$-th  column and the $j$-th row this means that $p_k$ should satisfy
\begin{align} \label{ch4:eq:asympforpk}
p_k(\xi_j(z)) + \mathcal O\left(\frac{1}{z}\right) = i z^{-\frac{k}{r}} \Omega^{\pm (-1)^{j} \lfloor \frac{j}{2}\rfloor k} h(\xi_j(z))^{-1} 
\end{align}
as $z\to\infty$. Now by \eqref{ch4:eq:zktoxi} and the fact that $h(0)\neq 0$, we infer that there exist coefficients $a_0^{[k]}, a_1^{[k]}, \ldots$, not depending on $j$, such that
\begin{align*}
 i z^{-\frac{k}{r}} \Omega^{\pm (-1)^{j} \lfloor \frac{j}{2}\rfloor k} h(\xi_j(z))^{-1} 
 = a_0^{[k]} + a_1^{[k]} \xi_j(z) + a_2^{[k]} \xi_j(z)^2 + \ldots + a_{r-1}^{[k]} \xi_{j}(z)^{r-1} + \mathcal O\left(\xi_j(z)^r\right)
\end{align*}
as $z\to\infty$. In principle these coefficients can be determined explicitly with the help of a taylor series, but we shall not need an explicit description. Looking at \eqref{ch4:eq:asympforpk}, we infer that we obtain RH-N3 in the lower-right $r\times r$ block if we define
\begin{align} \label{ch4:eq:defpkbk}
p_k(\xi) = a_0^{[k]} + a_1^{[k]} \xi + a_2^{[k]} \xi^2 + \ldots + a_{r-1}^{[k]} \xi^{r-1} + b_k \xi^r,
\end{align}
where we still have some freedom in choosing $b_k$.

Let us now focus on getting the correct asymptotics RH-N3 in the first colum. 
By \eqref{ch4:eq:decompF} and \eqref{ch4:eq:behavInftyXi1} we have as $z\to\infty$ that
\begin{align} \label{ch4:eq:Fxi0asymp}
F(\xi_0(z)) = 1+ \mathcal O\left(\frac{1}{z}\right).
\end{align}
We should have for all $k=1,\ldots,r$ that
\begin{align*}
p_k(\xi_0(z)) F(\xi_0(z)) = \mathcal O\left(\frac{1}{z}\right)
\end{align*}
as $z\to\infty$. Using the asymptotics of $\xi_0$ from \eqref{ch4:eq:behavInftyXi1} and \eqref{ch4:eq:Fxi0asymp} this means that we should have 
$p_k(1) = 0$. This can easily be achieved, simply by choosing $b_k = -a_0^{[k]} - a_1^{[k]} - a_2^{[k]} - \ldots - a_{r-1}^{[k]}$ in \eqref{ch4:eq:defpkbk}. This fixes the definition of $p_1,\ldots,p_r$ and we have RH-N3 for all rows except the first row. 
We should have 
\begin{align*}
p_0(\xi_0(z))F(\xi_0(z)) &= 1+\mathcal O\left(\frac{1}{z}\right),\\
p_0(\xi_j(z)) F(\xi_j(z)) &=\mathcal O\left(\frac{1}{z}\right),
\end{align*}
as $z\to\infty$, for all $j=1,\ldots,r$. In view of \eqref{ch4:eq:Fxi0asymp}, and the asymptotics \eqref{ch4:eq:behavInftyXi1} and \eqref{ch4:eq:behavInftyXi2}, this is achieved if we define
\begin{align*}
p_0(\xi) = \xi^r. 
\end{align*}
We conclude that RH-N3 is satisfied with our particular choice of polynomials $p_0, p_1, \ldots, p_r$. It is clear that, when we impose that $p_0, p_1, \ldots, p_r$ have degree at most $r$, we have no other choice for their coefficients, and the uniqueness follows. 

It remains to prove that RH-N2 is satisfied. These are more or less immediate due to the cuts of the Riemann surface associated with $\xi$ (see Lemma \ref{ch4:prop:propXi}(i)), except that we should argue that the minus signs are in the correct place, i.e., in the lower-left component of each $2\times 2$ blocks. For $x>0$ a particular $2\times 2$ block in the jump of $N_0$ takes the form
\begin{align}
\begin{pmatrix} \label{ch4:eq:2by2blockwithFF}
0 & F(\xi_{2j+1})_+(x)/F(\xi_{2j})_-(x)\\
F(\xi_{2j})_+(x)/F(\xi_{2j+1})_-(x) & 0
\end{pmatrix}
\end{align}
where $j=0,1,\ldots,\lfloor \frac{r-1}{2} \rfloor$. We know that $\xi_{2j+1,+}((0,\infty))$ does not intersect with any of the cuts of $F$, hence we will not get a minus sign. However, we know that $\xi_{2j,+}([0,\infty))= \xi_{2j,+}(\Delta_{2j})$ is one of the cuts, hence we get a factor $-1$ in the lower left component of \eqref{ch4:eq:2by2blockwithFF}. In the case that $r$ is even, the last block is a $1\times 1$ block. Indeed we have
\begin{align*}
F(\xi_r)_+(x)/F(\xi_r)_-(x) = 1
\end{align*}
because $\xi_{r}((0,\infty))$ does not intersect with any of the cuts of $F$.

An analogous argument works for the jump with $x<0$. 
\end{proof}

\subsection{Definition of the global parametrix for general $\beta$}

With the global parametrix $N_0$ for $\beta=0$, and the functions $\xi_0,\ldots,\xi_r$ at our disposal, we are ready to define the global parametrix for general $\beta$ (or $\alpha$ equivalently). 

\begin{definition} \label{ch4:def:N}
We define the global parametrix by
\begin{align} \label{ch4:eq:defN}
N(z) = C_\beta N_0(z) (z^{-\beta} \oplus \mathbb I_{r\times r})
\bigoplus_{j=0}^r e^{-\beta \log(1-\xi_j(z))},
\end{align}
where $C_\beta$ is the matrix given by
\begin{multline} \label{ch4:eq:defGammabeta}
r \operatorname{diag}\left(1,c_q, c_q^2, \ldots, c_q^r\right)
\begin{pmatrix}
r^{\beta-1} c_q^{-\beta} & 0 & 0 & 0 & \hdots & 0\\
0 & 1 & \binom{\beta+\frac{1}{r}}{1} & \binom{\beta+\frac{2}{r}}{2} & \hdots & \binom{\beta+\frac{r-1}{r}}{r-1}\\
0 & 0 & 1 & \binom{\beta+\frac{1}{r}}{1} & \hdots & \binom{\beta+\frac{r-2}{r}}{r-2}\\
0 & 0 & 0 & 1 & \hdots & \binom{\beta+\frac{r-3}{r}}{r-3}\\
\vdots & & & & \ddots & \vdots\\
0 & 0 & 0 & 0 & \hdots & 1
\end{pmatrix}
\operatorname{diag}\left(1,c_q, c_q^2, \ldots, c_q^r\right)^{-1}.
\end{multline}
\end{definition}

\begin{theorem}
The global parametrix problem, for general $\alpha>-1$, is solved by $N$ as in Definition \ref{ch4:def:N}. 
\end{theorem}

\begin{proof}
We know from Lemma \ref{ch4:prop:propXi} that $\xi_0((-\infty,0))=(1,\infty)$ and $\xi_0((q,\infty))=(\frac{r}{r+1},1)$.  Then the values $(\frac{r}{r+1},\infty)$ are not attained by the other $\xi_j$. Hence $1-\xi_j(\mathbb C\setminus \Delta_j)$ does not intersect with $(-\infty,0)$ when $j=1,\ldots,r$. Thus of all the components of the diagonal matrix in \eqref{ch4:eq:defN} only $z^{-\beta} e^{-\beta \log(1-\xi_0(z))}$ could possibly be a case where the cut of the logarithm is intersected. Using Lemma \ref{ch4:prop:propXi}(iii) we infer that for $x<0$ we get the jump
\begin{align*}
\left(x^{-\beta} e^{-\beta \log(1-\xi_0(x))}\right)_\pm &= e^{\pm \pi i \beta} |x|^{-\beta} e^{\mp \pi i\beta} |1-\xi_0(x)|^{-\beta}
= |x|^{-\beta} |1-\xi_0(x)|^{-\beta}.
\end{align*}
We conclude that $z^{-\beta} e^{-\beta \log(1-\xi_0(z))}$ does not have a jump on $(-\infty,0)$. Combining this with Lemma \ref{ch4:prop:propXi}(i), we infer that $N$ satisfies RH-N2. Notice that the factor $z^{-\beta}$ in \eqref{ch4:eq:defN} does indeed yield the correct powers of $z$ in the upper-left $2\times 2$ block of the jump for $x\in (0,q)$. 

It remains to show that RH-N3 is satisfied if we choose $C_\beta$ correctly. It follows from \eqref{ch4:eq:behavInftyXi1} that
\begin{align} \label{ch4:eq:rbetaelogxi0}
r^\beta c_q^{-\beta} e^{-\beta \log(1-\xi_0(z))} = 1 + \mathcal O\left(\frac{1}{z}\right)
\end{align}
as $z\to\infty$. For $j=1,\ldots,r$ we may use \eqref{ch4:eq:behavInftyXi2} to conclude that
\begin{align} \label{ch4:eq:rbetaelogxij}
e^{-\beta \log(1-\xi_j(z))} = (1-\xi_j(z))^{-\beta}
\end{align}
for $z$ large enough. We are going to find $C_\beta$ in the form $C_\beta = 1 \oplus \Gamma_\beta$, where $\Gamma_\beta$ is an $r\times r$ matrix that only depends on $\beta$. Then, in view of \eqref{ch4:eq:rbetaelogxi0} and \eqref{ch4:eq:rbetaelogxij}, we obtain RH-N3 if for $\pm\operatorname{Im}(z)>0$ we have as $z\to\infty$ that
\begin{align*}
\Gamma_\beta z^\frac{r-1}{2r} \left(\bigoplus_{j=1}^r z^{-\frac{j-1}{r}}\right) U^\pm D^\pm
\left(\bigoplus_{j=1}^r 1-\xi_j(z)\right)^{-\beta}
= \left(\mathbb I + \mathcal O\left(\frac{1}{z}\right)\right)
z^\frac{r-1}{2r} \left(\bigoplus_{j=1}^r z^{-\frac{j-1}{r}}\right) U^\pm D^\pm.
\end{align*}
Thus we should have that
\begin{align} \label{ch4:eq:Gammabetaimplicit}
\Gamma_\beta +\mathcal O\left(\frac{1}{z}\right) 
= \left(\bigoplus_{j=1}^r z^{-\frac{j-1}{r}}\right) U^\pm
\left(\bigoplus_{j=1}^r 1-\xi_j(z)\right)^{\beta}
(U^\pm)^{-1} \bigoplus_{j=1}^r z^{\frac{j-1}{r}}.
\end{align}
for $\pm\operatorname{Im}(z)>0$ as $z\to\infty$. Using \eqref{ch4:eq:xitoz1} we can write the component in the $k$-th column and $l$-th row of the right-hand side of \eqref{ch4:eq:Gammabetaimplicit} as
\begin{align*}
c_q^\frac{k-l}{r} \sum_{m=1}^r \xi_m(z)^{k-l} (1-\xi_m(z))^{\frac{k-l}{r}+\beta}. 
\end{align*}
Then we must have
\begin{align} \label{ch4:eq:GammabetaxiO}
\Gamma_{\beta,kl} &= c_q^\frac{k-l}{r} \sum_{m=1}^r \xi_m(z)^{k-l} (1-\xi_m(z))^{\frac{k-l}{r}+\beta} + \mathcal O\left(\frac{1}{z}\right)
\end{align}
as $z\to\infty$. The expression on the right-hand side of \eqref{ch4:eq:GammabetaxiO} (without the $\mathcal O$ term) is invariant under permutations of $\xi_1,\ldots,\xi_r$. Then it cannot have a jump for large enough $z$ ($\xi_0$ does not enter the equation when $|z|>q$) hence it must have a Laurent series expansion (of integer powers of $z$) around $z=0$. All these powers have to be non-positive due to \eqref{ch4:eq:behavInftyXi2}. Indeed, we always have $\xi_m(z)^{k-l}=\mathcal O\left(z^{-\frac{r}{r-1}}\right)$ thus there is no positive (integer) power of $z$ in the Laurent series. Hence we deduce from \eqref{ch4:eq:GammabetaxiO} that we obtain RH-N3 if we take
\begin{align*} 
C_{\beta,kl} &= c_q^\frac{k-l}{r} \lim_{R\to \infty} \oint_{|z|=R} \sum_{m=1}^r \xi_m(z)^{k-l} (1-\xi_m(z))^{\frac{k-l}{r}+\beta}  \frac{dz}{z}\\
&= \left\{\begin{array}{ll} r \binom{\frac{k-l}{r}+\beta}{l-k} c_q^{\frac{k-l}{r}}, & k\leq l\\
0, & k>l.\end{array}\right.
\end{align*}
for $k,l=1,\ldots,r$, and this is in agreement with \eqref{ch4:eq:defGammabeta}. In the last step we have again used the argument of invariance under permutations of $\xi_1,\ldots,\xi_r$ to argue that certain expressions are analytic and thus vanish when integrated over. 
\end{proof}

\subsection{Behavior of the global parametrix near the hard and soft edge}

\begin{proposition}
The global parametrix $N$ has the following behavior near the branch points. 
\begin{align} \label{ch4:eq:behavNasz0}
N(z) & \operatorname{diag}\left(z^\frac{r\beta}{r+1}, z^{-\frac{\beta}{r+1}},\ldots,z^{-\frac{\beta}{r+1}}\right) = \mathcal O\left(z^{-\frac{r}{2(r+1)}}\right), & \text{as } z\to 0.\\ \label{ch4:eq:behavNaszq}
N(z) &= \mathcal O\begin{pmatrix}(z-q)^{-\frac{1}{4}} & (z-q)^{-\frac{1}{4}} & 1 & \hdots & 1\\
\vdots & & & & \vdots\\
(z-q)^{-\frac{1}{4}} & (z-q)^{-\frac{1}{4}} & 1 & \hdots & 1
\end{pmatrix}, & \text{as } z\to q.
\end{align}
\end{proposition}

\begin{proof}
It follows from \eqref{ch4:eq:behav0Xi} that for all $k,j=0,1,\ldots,r$
\begin{align} \label{ch4:eq:pkxijbehavat0}
p_k(\xi_j(z))=\mathcal O\left(z^{-\frac{r}{r+1}}\right)
\end{align}
as $z\to 0$, where $p_0,\ldots,p_r$ are the polynomials of degree $r$ in the definition of $N_0$ in Theorem \ref{ch4:thm:globalParam0}. 
On the other hand, we have for $j=0,1,\ldots,r$ that
\begin{align} \label{ch4:eq:Fxijbehavat0}
F(\xi_j(z)) = \mathcal O\left(z^{\frac{r}{2(r+1)}}\right)
\end{align}
as $z\to 0$, where $F$ is as in Theorem \ref{ch4:thm:globalParam0}. Then, plugging \eqref{ch4:eq:pkxijbehavat0} and \eqref{ch4:eq:Fxijbehavat0} in \eqref{ch4:eq:defN}, we have as $z\to 0$
\begin{align} \label{ch4:eq:CbetaN0behav0}
C_\beta N_0(z) =\mathcal O\left(z^{-\frac{r}{2(r+1)}}\right).
\end{align}
Using Corollary \ref{ch4:cor:behavxi0} we infer that for $j=0,1,\ldots,r$ as $z\to 0$
\begin{align*} 
e^{-\beta \log(1-\xi_j(z))} = \mathcal O\left(z^{\frac{\beta}{r+1}}\right).
\end{align*}
From this equation it follows that
\begin{align*}
(z^{-\beta} \oplus \mathbb I_{r\times r})
\bigoplus_{j=0}^r e^{-\beta \log(1-\xi_j(z))} \operatorname{diag}\left(z^\frac{r\beta}{r+1}, z^{-\frac{\beta}{r+1}},\ldots,z^{-\frac{\beta}{r+1}}\right)
= \mathcal O\left(1\right)
\end{align*}
as $z\to 0$, and if we combine this with \eqref{ch4:eq:CbetaN0behav0} and \eqref{ch4:eq:defN} then we arrive at \eqref{ch4:eq:behavNasz0}. 

To prove \eqref{ch4:eq:behavNaszq} we first notice that $F$, as defined in Theorem \ref{ch4:thm:globalParam0}, satisfies
\begin{align} \label{ch4:eq:behavFxiq}
F(\xi) = \mathcal O\left(\left(\xi-\frac{r}{r+1}\right)^{-\frac{1}{2}}\right)
\end{align}
as $z\to \frac{r}{r+1}$. We know from Lemma \ref{ch4:prop:propXi} that $\xi_0(q) = \frac{r}{r+1}$ and consequently also $\xi_1(q) = \frac{r}{r+1}$. The point $q$ corresponding to the first two sheets of the Riemann surface associated with \eqref{ch4:eq:xiAlg} is a regular point of the Riemann surface, and only $\xi_0$ and $\xi_1$ can approach $\frac{r}{r+1}$. Then we must conclude that there is a square root branch at $q$, i.e., we have
\begin{align*}
\xi_j(z)-\frac{r}{r+1} = \mathcal O\left(\sqrt{z-q}\right)
\end{align*}
as $z\to q$, for $j=0$ and $j=0$. Plugging this in \eqref{ch4:eq:behavFxiq} we find that
\begin{align*} 
F(\xi_j(z)) = \mathcal O\left(\left(z-q\right)^{-\frac{1}{4}}\right)
\end{align*}
as $z\to q$, for $j=0$ and $j=1$. For $j=2,\ldots,r-1$ the functions $\xi_j$ are bounded around $q$ and their limiting values are not $\frac{r}{r+1}$ or $0$, hence $F(\xi_j(z))$ is bounded around $z=q$. It's a simple task to verify that all the other expressions in the definition of $N$ are bounded, and \eqref{ch4:eq:behavNaszq} follows. 
\end{proof}

\section{Local parametrices}

Close to the end points $0$ and $q$ the global parametrix cannot be a good approximation. This means that we have to consider local parametrix problems around $z=0$ and $z=q$. The local parametrix problem around $z=q$ is standard and we omit the details, but the local parametrix problem around $z=0$ is new (when $r>2$) and we work it out in detail. It shows similarities with the bare Meijer-G parametrix from \cite{BeBo}, although I do not believe that there is a (simple) way to map the local parametrix problems to each other. 

\subsection{The local parametrix problem around the hard edge $z=0$} \label{ch4:sec:localParamSetUp}

The assumption that $V$ is real analytic on $[0,\infty)$ implies that there is an open neighborhood $O_V$ of $[0,\infty)$ on which $V$ can be analytically continued. We now consider a disk $D(0,r_0)\subset O_V$ around the origin of radius $r_0$. Here $r_0$ is a positive number that we shall eventually fix (see Section \ref{ch4:sec:conformalf}). It is assumed that $r_0$ is sufficiently small, such that the lips of the lens inside $D(0,r_0)$ are on the imaginary axis (see Figure \ref{ch4:FigS} also). We shall orient the boundary circle of any disk positively. The \textit{initial local parametrix problem} (we explain this terminology in a moment) is as follows. 

\begin{rhproblem} \label{ch4:RHPforMathringP} \
\begin{description}
\item[RH-$\mathring{\text{P}}$1] $\mathring P$ is analytic on $D(0,r_0) \setminus \Sigma_S$.
\item[RH-$\mathring{\text{P}}$2] $\mathring P$ has the same jumps as $S$ has on $D(0,r_0) \setminus \Sigma_S$.
\item[RH-$\mathring{\text{P}}$3] $\mathring P$ has the same asymptotics as $S$ has near the origin. 
\end{description}
\end{rhproblem}

The matching condition is usually given as RH-$\mathring{\text{P4}}$. In larger size RHPs obtaining the matching is often a major technical issue, and ours is no exception. We will therefore use a double matching (see \cite{Mo}) instead of an ordinary matching. Then there is also a jump on a shrinking circle inside $D(0,r_0)$, in our case this circle turns out to be $\partial D(0,r_n)$ where
\begin{align*}
r_n &=n^{-\frac{r+1}{2}}, & n=1,2,\ldots
\end{align*}
On the other hand, in the annulus $r_n<|z|<r_0$, denoted $A(0;r_n,r_0)$, the local parametrix will not have a jump on the lips of the lens anymore. Hence the actual local parametrix $P$, i.e., the one that we will use in the final transformation, satisfies an altered version of RH-$\mathring{\text{P}}$ (and RH-$\mathring{\text{P}}$2 in particular). Indeed, this is why we called the local parametrix problem for $\mathring P$ the initial local parametrix problem. In general, we will use the same notations and terminology as in \cite{Mo} as much as possible. We will first find a solution $\mathring P$ to RH-$\mathring{\text{P}}$ and then, using the double matching approach, we will construct $P$ as
\begin{align*}
P(z) = \left\{\begin{array}{ll}
E_n^0(z) \mathring P(z), & z\in D(0,r_n),\\
E_n^\infty(z) N(z), & z\in A(0;r_n,r_0),
\end{array}\right.
\end{align*}
where $E_n^0$ and $E_n^\infty$ are analytic prefactors that we shall obtain from Theorem 1.2 in \cite{Mo}. Then it will turn out that $P$ satisfies a double matching of the form
\begin{align*}
P_+(z) N(z)^{-1} &= \mathbb I + \mathcal O\left(\frac{1}{n}\right), & \text{uniformly for }z\in\partial D(0,r_0),\\
P_+(z) P_-(z)^{-1} &= \mathbb I + \mathcal O\left(\frac{1}{n^{r+2}}\right), & \text{uniformly for }z\in\partial D(0,r_n),
\end{align*}
as $n\to\infty$. For convenience to the reader, we repeat Theorem 1.2 of \cite{Mo} in Section \ref{ch4:sec:matching}, as \text{Theorem \ref{lem:matching}}. Much of our approach concerning the local parametrix is parallel to the approach in \cite{KuMo}, where the case $r=2$ was treated. 

\subsection{Reduction to constant jumps}

The first step to solving a local parametrix problem is generally to transform it to a problem with constant jumps. To that end we define $\varphi$-functions as follows.

\subsubsection{Definition of the $\varphi$-functions}

\begin{definition}
For $z\in O_V \setminus \mathbb R$ with $\pm\operatorname{Im}(z)>0$ we define
\begin{align} \label{ch4:eq:defvarphi0}
\varphi_0(z) &= - g_0(z) + \frac{1}{2} g_1(z) + \frac{1}{2}(V(z)+\ell) \pm \pi i,\\ \nonumber
\varphi_j(z) &= \frac{1}{2} g_{j-1}(z) - g_j(z) + \frac{1}{2} g_{j+1}(z) \pm (-1)^j \frac{r-j}{r} \pi i,\\ \label{ch4:eq:defvarphij}
&\hspace{4cm} j=1,\ldots,r-2,\\  \label{ch4:eq:defvarphir-1}
\varphi_{r-1}(z) &= \frac{1}{2} g_{r-2}(z) - g_{r-1}(z) \pm (-1)^{r-1} \frac{\pi i}{r}.
\end{align}
Here $g_0, g_1,\ldots, g_{r-1}$ are the $g$-functions as in \eqref{ch4:eq:defgfunctions}. For convenience, we also define $\varphi_{-1}=\varphi_r=0$. 
\end{definition}

Notice that we have $\varphi_0(z) = \varphi(z) \pm \pi i$ according to \eqref{ch4:eq:defvarphi}. The explicit form of the $\varphi$-functions is dictated by the variational equations \eqref{ch4:varCon1} and \eqref{ch4:varCon2}. Notice that the definition of $\varphi_2, \ldots, \varphi_{r-1}$ makes sense on $\mathbb C\setminus \mathbb R$, for our purposes it will suffice to let them have domain $O_V\setminus \mathbb R$ though. 

\begin{lemma} \label{ch4:prop:phiRelations}
For all $j=0,1,\ldots,r-1$ we have for $x\in \Delta_j\cap O_V$ that
\begin{align} \label{ch4:eq:varphij+=varphij-}
\varphi_{j,+}(x) &= -\varphi_{j,-}(x). 
\end{align} 
Furthermore, for all $j=0,\ldots,r-1$ we have for $x\in (\mathbb R\cap O_V)\setminus \Delta_j$ that
\begin{align} \label{ch4:eq:varphijcrelation}
\varphi_{j+}(x) &= \varphi_{j-1,-}(x)+\varphi_{j-}(x)+\varphi_{j+1,-}(x).
\end{align}
\end{lemma}

\begin{proof}
First we prove the relation \eqref{ch4:eq:varphij+=varphij-}. For $x\in (0,q) \cap O_V$ and $j=0$ we have 
\begin{align} \nonumber
\varphi_{0,\pm}(x) &= - \int_0^q \log|x-s| d\mu_0(s) \mp \pi i \int_x^q d\mu_{0}(s) + \frac{1}{2} \int_{-\infty}^0 \log|x-s| d\mu_1(s)
+ \frac{1}{2}(V(x)+\ell) \pm \pi i\\ \label{ch4:eq:varphieven0}
&= \pm\mu_0([0,x]),
\end{align}
where we have used the variational conditions \eqref{ch4:varCon1}, and the fact that $\mu_0$ has total mass $1$. Similarly, we have by \eqref{ch4:varCon2} for even $j>0$ and $x>0$  that
\begin{align} \nonumber
\varphi_{j,\pm}(x) &= \frac{1}{2} \int_{-\infty}^0 \log|x-s| d\mu_{j-1}(s) - \int_0^\infty \log|x-s| d\mu_j(s) \\ \nonumber
&\hspace{0.5cm}\mp \pi i \int_x^\infty d\mu_j(s)+ \frac{1}{2} \int_{-\infty}^0 \log|x-s| d\mu_{j+1}(s) \pm \frac{r-j}{r} \pi i\\ \nonumber
&= \mp \pi i \mu_j([x,\infty]) \pm \frac{r-j}{r} \pi i\\ \label{ch4:eq:varphievenj}
&= \pm \pi i \mu_j([0,x]).
\end{align}
Here we have used that $\mu_j$ has total mass $\frac{r-j}{r}$, see \eqref{ch4:eq:totalMass}. For odd $j$ we have for $x<0$ that
\begin{align} \nonumber
\varphi_{j,\pm}(x) &= \frac{1}{2} \int_{0}^\infty \log|x-s| d\mu_{j-1}(s) \mp \frac{\pi i}{2} \int_0^\infty d\mu_{j-1}(s)\\ \nonumber
&\quad - \int_0^\infty \log|x-s| d\mu_j(s) \pm \pi i \int_x^0 d\mu_j(s)\\ \nonumber
&\quad + \frac{1}{2} \int_{0}^\infty \log|x-s| d\mu_{j+1}(s) \mp \int_0^\infty d\mu_{j+1}(s) \pm \frac{r-j}{r} \pi i\\ \nonumber
&= \mp \frac{\pi i}{2} \left(\frac{r-j+1}{r}+\frac{r-j-1}{r}\right) \pm \mu_j([x,0]) \pm \frac{r-j}{r} \pi i\\ \label{ch4:eq:varphioddj}
&= \pm \mu_j([x,0]),
\end{align}
where we used that $\mu_{j-1}$ and $\mu_{j+1}$ have total mass $\frac{r-j+1}{r}$ and $\frac{r-j-1}{r}$ respectively.
We conclude that $\varphi_{j+}(x) = -\varphi_{j-}(x)$ for $x\in \Delta_j \cap O_V$ for all $j=0,1,\ldots,r-1$. 

Now we prove \eqref{ch4:eq:varphijcrelation}. For $x<0$ we have
\begin{align*}
\varphi_{0,+}(x) - \varphi_{0,-}(x) &= - (g_{0+}(x)-g_{0-}(x)) + \frac{1}{2} (g_{1+}(x)-g_{1-}(x)) - 2\pi i\\
&= 2\pi i \int_0^q d\mu_0(s) - \pi i \int_x^0 d\mu_1(s) - 2\pi i\\
&= -\pi i \mu_{1}([x,0])\\
&= \varphi_{1,-}(x),
\end{align*}
where we have used that $V$ is analytic and that $\mu_0$ has total mass $1$, and \eqref{ch4:eq:varphioddj} with $j=1$ in the last line. Similarly we have for even $j>0$ and $x<0$ that
\begin{align*}
&\varphi_{j,+}(x) - \varphi_{j,-}(x)\\
&= \frac{1}{2} (g_{j-1,+}(x) - g_{j-1,-}(x))
- (g_{j,+}(x) - g_{j,-}(x)) 
+ \frac{1}{2} (g_{j+1,+}(x) - g_{j+1,-}(x)) - 2\pi i \frac{r-j}{r}\\
&= -\pi i\int_x^0 d\mu_{j-1}(s) + 2\pi i\int_0^\infty d\mu_j(s) - \pi i\int_x^0 d\mu_{j+1}(s) - 2\pi i \frac{r-j}{r}\\
&= -\pi i\mu_{j-1}([x,0]) + 2\pi i \frac{r-j}{r} - \pi i\mu_{j+1}([x,0]) - 2\pi i \frac{r-j}{r}\\
&= \varphi_{j-1,-}(x) + \varphi_{j+1,-}(x),
\end{align*}
where we used that $\mu_j$ has total mass $\frac{r-j}{r}$, and we used \eqref{ch4:eq:varphioddj} in the last line. For odd $j$ and $x>0$ we find
\begin{align*}
\varphi_{j,+}(x) - \varphi_{j,-}(x) &= \pi i\int_x^\infty d\mu_{j-1}(s) + \pi i\int_x^\infty d\mu_{j+1}(s) - 2\pi i \frac{r-j}{r}\\
&= \pi i\mu_{j-1}([x,\infty)) + \pi i\mu_{j+1}([x,\infty)) - 2\pi i \frac{r-j}{r}\\
&= - \pi i\mu_{j-1}([0,x]) - \pi i\mu_{j+1}([0,x])\\
&= \varphi_{j-1,-}(x) + \varphi_{j+1,-}(x),
\end{align*}
where we used \eqref{ch4:eq:varphievenj}, and \eqref{ch4:eq:varphieven0} for $j=1$, in the last line. We conclude that $$\varphi_{j,+}(x)=\varphi_{j-1,-}(x) + \varphi_{j,-}(x)+ \varphi_{j+1,-}(x)$$ for $x\in (\mathbb R\cap O_V)\setminus \Delta_j$ for all $j=0,1,\ldots,r-1$. 
\end{proof}

\subsubsection{Analytic functions constructed out of the $\varphi$-functions}

From the $\varphi$-functions we construct functions $f_1,\ldots,f_m$ that we will eventually use to reduce the jumps of the local parametrix problem to constant jumps. These functions are the generalization of the analytic functions $f_1$ and $f_2$ from Section 5.3 in \cite{KuMo}. 

\begin{definition} \label{ch4:def:fm}
Let $m=1,2,\ldots,r$. We define for $z\in D(0,q) \cap O_V$
\begin{align} \label{ch4:eq:deffm}
f_m(z) = -z^{-\frac{m}{r+1}} 
\left\{\begin{array}{ll}
\displaystyle\sum_{j=0}^{r-1} \left(\sum_{k=0}^j \omega^{m ((-1)^{k} \lfloor \frac{k+1}{2}\rfloor+\frac{1}{2})} \right) \varphi_j(z), & \operatorname{Im}(z)>0,\\
\displaystyle\sum_{j=0}^{r-1} \left(\sum_{k=0}^j \omega^{m ((-1)^{k-1} \lfloor \frac{k+1}{2}\rfloor - \frac{1}{2})} \right) \varphi_j(z), & \operatorname{Im}(z)<0.
\end{array} \right.
\end{align}
\end{definition}

\begin{proposition} \label{ch4:prop:fmanalytic}
$f_m$ defines an analytic function for every $m=1,2,\ldots,r$.
\end{proposition}

\begin{proof}
We only prove it for the case $r\equiv 0\mod{2}$, the case $r\equiv 1\mod{2}$ is analogous.

For $x\in (0,q) \cap O_V$ we have by \eqref{ch4:eq:varphij+=varphij-} and \eqref{ch4:eq:varphijcrelation} that
\begin{align*} 
-\omega^{-\frac{m}{2}} x^\frac{m}{r+1} f_{m+}(x)
&= \sum_{j=0}^{\frac{r}{2}-1} \left(\sum_{k=0}^{2j} \omega^{m (-1)^{k} \lfloor \frac{k+1}{2}\rfloor} \right) \varphi_{2j+}(x) 
+\sum_{j=0}^{\frac{r}{2}-1} \left(\sum_{k=0}^{2j+1} \omega^{m (-1)^{k} \lfloor \frac{k+1}{2}\rfloor} \right) \varphi_{2j+1,+}(x)\\
&= -  \sum_{j=0}^{\frac{r}{2}-1} \left(\sum_{k=0}^{2j} \omega^{m (-1)^{k} \lfloor \frac{k+1}{2}\rfloor} \right) \varphi_{2j-}(x)\\
& \quad +  \sum_{j=0}^{\frac{r}{2}-1} \left(\sum_{k=0}^{2j+1} \omega^{m (-1)^{k} \lfloor \frac{k+1}{2}\rfloor} \right) (\varphi_{2j,-}(x)+\varphi_{2j+1,-}(x)+\varphi_{2j+2,-}(x))
\end{align*}
Notice that we have multiplied $f_m$ with an appropriate factor, $-\omega^{-\frac{m}{2}} x^\frac{m}{r+1}$. This is simply a pragmatic choice that makes our equations look nicer. When we shift the summation index for $\varphi_{2j+2,-}(x)$ we can write this as
\begin{align*} \nonumber
& (-1 + 1 + \omega^{-m}) \varphi_{0,-}(x)\\ \nonumber
&+  \sum_{j=1}^{\frac{r}{2}-1} 
\left(-\left(\sum_{k=0}^{2j} \omega^{m (-1)^{k} \lfloor \frac{k+1}{2}\rfloor} \right)
+ \left(\sum_{k=0}^{2j+1} \omega^{m (-1)^{k} \lfloor \frac{k+1}{2}\rfloor} \right)\right. 
\left. +\left(\sum_{k=0}^{2j-1} \omega^{m (-1)^{k} \lfloor \frac{k+1}{2}\rfloor} \right)\right) \varphi_{2j-}(x)\\
&\hspace{4.7cm} + \sum_{j=0}^{\frac{r}{2}-1} \left(\sum_{k=0}^{2j+1} \omega^{m (-1)^{k} \lfloor \frac{k+1}{2}\rfloor} \right) \varphi_{2j+1,-}(x)\\
&= \sum_{j=0}^{\frac{r}{2}-1} \left(\omega^{-m (j+1)}+\sum_{k=0}^{2j-1} \omega^{m (-1)^{k} \lfloor \frac{k+1}{2}\rfloor}\right)\\
&\hspace{4.7cm} + \sum_{j=0}^{\frac{r}{2}-1} \left(\sum_{k=0}^{2j+1} \omega^{m (-1)^{k} \lfloor \frac{k+1}{2}\rfloor} \right) \varphi_{2j+1,-}(x).
\end{align*}
To prove that $f_m$ has no jump for $x\in (0,\infty)\cap O_V$ it then suffices to show the following two identities,
\begin{align} \label{ch4:eq:firstWeirdOmegaIdentity}
 \sum_{k=0}^{2j+1} \omega^{m (-1)^{k} \lfloor \frac{k+1}{2}\rfloor} 
 =  \sum_{k=0}^{2j+1} \omega^{m ((-1)^{k-1} \lfloor \frac{k+1}{2}\rfloor - 1)}
\end{align}
and 
\begin{align} \label{ch4:eq:secondWeirdOmegaIdentity}
\omega^{-m (j+1)}+\sum_{k=0}^{2j-1} \omega^{m (-1)^{k} \lfloor \frac{k+1}{2}\rfloor}
=\sum_{k=0}^{2j} \omega^{m ((-1)^{k-1} \lfloor \frac{k+1}{2}\rfloor - 1)},
\end{align}
because then we would obtain the coefficients as in \eqref{ch4:eq:deffm} for $\operatorname{Im}(z)<0$. To prove the first identity we notice that
\begin{align*} 
\sum_{k=0}^{2j+1} \omega^{m (-1)^{k} \lfloor \frac{k+1}{2}\rfloor} 
&= \sum_{k=0}^j \omega^{2 m k} 
+  \sum_{k=0}^j \omega^{-m (2k+1)}\\
&= \frac{1-\omega^{2m (j+1)}}{1-\omega^{2m}} + \omega^{-m} \frac{1-\omega^{-2m (j+1)}}{1-\omega^{-2m}}\\
&= \frac{1-\omega^m - \omega^{m(2j+2)}+\omega^{-m(2j+1)}}{1-\omega^{2m}}.
\end{align*}
Indeed, then, complex conjugating and multiplying by $\omega^{-m}$, we have
\begin{align} \nonumber
\sum_{k=0}^{2j+1} \omega^{m ((-1)^{k-1} \lfloor \frac{k+1}{2}\rfloor - 1)}
&= \omega^{-m} \frac{1-\omega^{-m} - \omega^{-m(2j+2)}+\omega^{m(2j+1)}}{1-\omega^{-2m}}\\ \nonumber
&= \frac{-\omega^m (1-\omega^{-m} - \omega^{-m(2j+2)}+\omega^{m(2j+1)})}{1-\omega^{2m}}\\ \label{ch4:eq:OmegaSums2}
&= \sum_{k=0}^{2j+1} \omega^{m (-1)^{k} \lfloor \frac{k+1}{2}\rfloor}.
\end{align}
The first identity \eqref{ch4:eq:firstWeirdOmegaIdentity} is proved. To prove the second identity \eqref{ch4:eq:secondWeirdOmegaIdentity} we use \eqref{ch4:eq:OmegaSums2} to see that
\begin{align*}
\sum_{k=0}^{2j-1} & \omega^{m ((-1)^{k-1} \lfloor \frac{k+1}{2}\rfloor -1)}+ \omega^{-m (j+1)}\\
&= \sum_{k=0}^{2j-1} \omega^{m ((-1)^{k-1} \lfloor \frac{k+1}{2}\rfloor -1)}  + \omega^{m (-\lfloor \frac{2j+1}{2}\rfloor -1)}\\
&= \sum_{k=0}^{2j} \omega^{m ((-1)^{k-1} \lfloor \frac{k+1}{2}\rfloor -1)}.
\end{align*}
Then \eqref{ch4:eq:secondWeirdOmegaIdentity} is also proved. We conclude that $f_m$ does not have a jump on $(0,q)\cap O_V$. 

Now we will prove that it also does not have a jump on $(-q,0)\cap O_V$. For $x<0$, we have, again using \eqref{ch4:eq:varphij+=varphij-} and \eqref{ch4:eq:varphijcrelation}, that
\begin{align*} \nonumber
- |x|^\frac{m}{r+1} f_{m+}(x)
&= \sum_{j=0}^{\frac{r}{2}-1} \left(\sum_{k=0}^{2j} \omega^{m (-1)^{k-1} \lfloor \frac{k+1}{2}\rfloor} \right)  (\varphi_{2j-1,-}(x)+\varphi_{2j,-}(x)+\varphi_{2j+1,-}(x))\\ \nonumber
& \quad -  \sum_{j=0}^{\frac{r}{2}-1} \left(\sum_{k=0}^{2j+1} \omega^{m (-1)^{k-1} \lfloor \frac{k+1}{2}\rfloor} \right) \varphi_{2j+1,-}(x)\\ \nonumber
&= \sum_{j=0}^{\frac{r}{2}-1} \left(\sum_{k=0}^{2j} \omega^{m (-1)^{k-1} \lfloor \frac{k+1}{2}\rfloor} \right) \varphi_{2j,-}(x)\\ \nonumber
& \quad + \sum_{j=0}^{\frac{r}{2}-1} \left( \left(\sum_{k=0}^{2j+2} \omega^{m (-1)^{k-1} \lfloor \frac{k+1}{2}\rfloor} \right)
-\left(\sum_{k=0}^{2j+1} \omega^{m (-1)^{k-1} \lfloor \frac{k+1}{2}\rfloor} \right)
\right. \\
&\hspace{4cm}\left. +\left(\sum_{k=0}^{2j} \omega^{m (-1)^{k-1} \lfloor \frac{k+1}{2}\rfloor} \right)
\right) \varphi_{2j+1,-}(x)\\ \nonumber
&= \sum_{j=0}^{\frac{r}{2}-1} \left(\sum_{k=0}^{2j} \omega^{m (-1)^{k-1} \lfloor \frac{k+1}{2}\rfloor} \right) \varphi_{2j,-}(x)\\
&\hspace{2cm}+ \sum_{j=0}^{\frac{r}{2}-1} \left(\sum_{k=0}^{2j} \omega^{m (-1)^{k-1} \lfloor \frac{k+1}{2}\rfloor} + \omega^{-m (j+1) } \right) \varphi_{2j+1,-}(x)
\end{align*}
Here we have used for the second equality that
\begin{align*}
\sum_{j=0}^{2(\frac{r}{2}-1)+2}  \omega^{m (-1)^{k-1} \lfloor \frac{k+1}{2}\rfloor} = \sum_{j=0}^r \omega^{j} = 0. 
\end{align*}On the other hand, we have
\begin{align*}
- |x|^\frac{m}{r+1} f_{m-}(x)
= \sum_{j=0}^{r-1} \left(\sum_{k=0}^j \omega^{m ((-1)^{k} \lfloor \frac{k+1}{2}\rfloor + 1)} \right) \varphi_{j-}(x).
\end{align*}
To see that $f_m$ does not have a jump on $(-q,0)\cap O_V$ we should then have the two identities
\begin{align*}
\sum_{k=0}^{2j} \omega^{m (-1)^{k-1} \lfloor \frac{k+1}{2}\rfloor}
= \sum_{k=0}^{2j} \omega^{m ((-1)^{k} \lfloor \frac{k+1}{2}\rfloor+1)}
\end{align*}
and
\begin{align*}
\sum_{k=0}^{2j+1} \omega^{m (-1)^{k-1} \lfloor \frac{k+1}{2}\rfloor}
= \sum_{k=0}^{2j} \omega^{m ((-1)^{k} \lfloor \frac{k+1}{2}\rfloor+1)} + \omega^{-m (j+1) }.
\end{align*}
These identities are simply the complex conjugate of the identities \eqref{ch4:eq:firstWeirdOmegaIdentity} and \eqref{ch4:eq:secondWeirdOmegaIdentity} that we found before. 
We conclude that $f_m$ has no jumps in $D(0,q)\cap O_V$. Since the $g$ functions (see Proposition \ref{ch4:prop:gfunctionsbounded0}) and $V$ are bounded on $D(0,q)$ we conclude that $f_m$ is analytic. 
\end{proof}

\begin{proposition} \label{ch4:prop:sumfm}
Let $j=0,1,\ldots,r$. We have for $z\in D(0,q)\cap O_V$ and $\pm\operatorname{Im}(z)>0$ that
\begin{align} \label{ch4:eq:sumfm}
\sum_{m=1}^r \omega^{\pm (-1)^{l-1} (\frac{1}{2}+\lfloor\frac{l}{2}\rfloor) m} z^\frac{m}{r+1} f_m(z)
= \sum_{j=0}^{r-1} (j+1) \varphi_j(z) - (r+1) \sum_{j=l}^{r-1} \varphi_j(z).
\end{align}
\end{proposition}

\begin{proof}
One may verify, by considering the different cases of parity, that 
\begin{align*}
(-1)^{l-1} \left(\frac{1}{2}+\left\lfloor\frac{l}{2}\right\rfloor\right) + (-1)^k \left\lfloor\frac{k+1}{2}\right\rfloor + \frac{1}{2} \equiv 0 \mod (r+1)
\end{align*}
only has $k=l$ as a solution (under the assumption that $0\leq k\leq r$). Combining this with Definition \ref{ch4:def:fm}, we get for $\pm\operatorname{Im}(z)>0$ that
\begin{align*}
-&\sum_{m=1}^r \omega^{\pm (-1)^{l-1} (\frac{1}{2}+\lfloor\frac{l}{2}\rfloor) m} z^\frac{m}{r+1} f_m(z)\\
&= \sum_{j=0}^{r-1} \sum_{k=0}^{j} \sum_{m=1}^r \omega^{\pm m ((-1)^{l-1} (\frac{1}{2}+\lfloor\frac{l}{2}\rfloor)+(-1)^k \lfloor \frac{k+1}{2}\rfloor)} \varphi_j(z)\\
&= \sum_{j=0}^{l-1} \sum_{k=0}^j (-1) \varphi_j(z) + \sum_{j=l}^{r-1} \left(r+\sum_{k=0, k\neq l}^{j} (-1)\right) \varphi_j(z)\\
&= -\sum_{j=0}^{l-1} (j+1) \varphi_j(z) + \sum_{j=l}^{r-1} (r-j) \varphi_j(z)\\
&= -\sum_{j=0}^{r-1} (j+1) \varphi_j(z) + (r+1) \sum_{j=l}^{r-1} \varphi_j(z).
\end{align*}
\end{proof}

\subsubsection{A local parametrix problem with constant jumps} \label{ch4:sec:szegoconstantjumps}

We define the following function. 

\begin{definition} \label{ch4:def:D0}
We define the $(r+1)\times (r+1)$ diagonal matrix
\begin{align} \label{ch4:eq:defD0}
D_0(z) &= \exp\left(\frac{2}{r+1} \sum_{j=0}^{r-1} (j+1) \varphi_j(z)\right) \bigoplus_{l=0}^{r} 
\exp\left(-2\sum_{j=l}^{r-1} \varphi_j(z)\right).
\end{align}
\end{definition}

This function is the equivalent of (5.15) in \cite{KuMo}. The function $D_0(z)$ relates RH-$\mathring{\text{P}}$ to a RHP with constant jumps, in the following way. 

\begin{proposition} \label{ch4:prop:RHPfromPtoQ}
Suppose that
\begin{align*}
\widetilde P(z) = \mathring P(z) 
\operatorname{diag}(1,z^{-\beta},\ldots,z^{-\beta}) 
D_0(z)^n.
\end{align*}
Then $\mathring P$ satisfies RHP-$\mathring{\text{P}}$ if and only if $\widetilde P$ satisfies RH-$\widetilde{\text{P}}$, as defined below. 
\end{proposition}

\begin{rhproblem} \label{ch4:RHPforQ} \
\begin{description}
\item[RH-$\widetilde{\text{P}}$1] $\widetilde P$ is analytic on $D(0,r_0) \setminus \Sigma_S$.
\item[RH-$\widetilde{\text{P}}$2] $\widetilde P$ has boundary values for $x\in D(0,r_0) \cap \Sigma_S$
\begin{align} \nonumber
\widetilde P_+(x) &= \widetilde P_-(x)\\ \nonumber
&\hspace{2cm}\times \left\{\begin{array}{ll} 
\bigoplus_{j=1}^{\frac{r}{2}} 
\begin{pmatrix}
0 & 1\\
-1 & 0
\end{pmatrix}
\oplus 1, & r \equiv 0 \mod 2,\\
\bigoplus_{j=1}^{\frac{r+1}{2}} 
\begin{pmatrix}
0 & 1\\
-1 & 0
\end{pmatrix}
, & r \equiv 1 \mod 2,
\end{array}
\right. \\ \label{ch4:RHS2}
&\hspace{5.5cm} x \in (0,r_0),\\ \nonumber
\widetilde P_+(x) &= \widetilde P_-(x)\\ \nonumber
&\times \left\{\begin{array}{ll}
1 \oplus
\bigoplus_{j=1}^{\frac{r}{2}} e^{2\pi i\beta}
\begin{pmatrix}
0 & 1\\
-1 & 0
\end{pmatrix}, & r \equiv 0 \mod 2,\\
1\oplus
\bigoplus_{j=1}^{\frac{r-1}{2}} e^{2\pi i\beta}
\begin{pmatrix}
0 & 1\\
-1 & 0
\end{pmatrix}
\oplus e^{2\pi i\beta}, & r \equiv 1 \mod 2,\\
\end{array}
\right. \\
&\hspace{5.5cm} x \in (-r_0,0),\\ \nonumber
\widetilde P_+(z) &= \widetilde P_-(z)  
\begin{pmatrix}
1 & 0\\
1 & 1
\end{pmatrix} 
\oplus \mathbb I_{(r-1)\times (r-1)}\\
&\hspace{4cm} z \in \Delta_0^\pm \cap D(0,r_0).
\end{align}
\item[RH-$\widetilde{\text{P}}$3] As $z\to 0$
\begin{align}
\nonumber
\widetilde P(z) &= \mathcal{O}\begin{pmatrix} h_{-\alpha-\frac{r-1}{r}}(z) & h_{-\alpha-\frac{r-1}{r}}(z) & \ldots &  h_{-\alpha-\frac{r-1}{r}}(z)\\ 
\vdots & & & \vdots\\
h_{-\alpha-\frac{r-1}{r}}(z) &  h_{-\alpha-\frac{r-1}{r}}(z) & \ldots & h_{-\alpha-\frac{r-1}{r}}(z)\end{pmatrix}\\ \label{ch4:RHQ3a}
&\hspace{4cm} \text{for $z$ to the right of }\Delta_0^\pm,\\ \nonumber
\label{ch4:RHQ3b}
\widetilde P(z) &= \mathcal{O}\begin{pmatrix} 1 & h_{-\alpha-\frac{r-1}{r}}(z) & \ldots &  h_{-\alpha-\frac{r-1}{r}}(z)\\ 
\vdots & & & \vdots\\
1 & h_{-\alpha-\frac{r-1}{r}}(z) & \ldots &  h_{-\alpha-\frac{r-1}{r}}(z)\end{pmatrix}\\
&\hspace{4cm} \text{for $z$ to the left of }\Delta_0^\pm.
\end{align}
\end{description}
\end{rhproblem}

\begin{proof}
So let us suppose that $\mathring P$ satisfies RH-$\mathring{\text{P}}$. For the positive and negative real axis the jumps of $\widetilde P$ follow from Proposition \ref{ch4:prop:sumfm}. For example, for $x<0$ the upper-right component of the $l$-th $2\times 2$ block of the jump matrix equals
\begin{align*}
& e^{\pi i\beta} |x|^\beta \exp\left(\sum_{m=1}^r \omega^{(-1)^{2l-2} (\frac{1}{2}+\lfloor\frac{2l-1}{2}\rfloor m} \omega^\frac{m}{2} |x|^\frac{m}{r+1} f_m(x)\right)\\
&\quad e^{\pi i\beta} |x|^{-\beta} \exp\left(\sum_{m=1}^r \omega^{-(-1)^{2l-1} (\frac{1}{2}+\lfloor\frac{2l}{2}\rfloor m} \omega^{-\frac{m}{2}} |x|^\frac{m}{r+1} f_m(x)\right)\\
&= e^{2\pi i\beta} 
\exp \left(\sum_{m=1}^r \omega^{(\frac{1}{2}+l-1) m} \omega^\frac{m}{2} |x|^\frac{m}{r+1} f_m(x) - \sum_{m=1}^r \omega^{(\frac{1}{2}+l) m} \omega^{-\frac{m}{2}} |x|^\frac{m}{r+1} f_m(x)\right)\\
&= e^{2\pi i\beta}.
\end{align*}
The other components, both for the jump for $x<0$ and $x>0$, follow with an analogous argument. For the jump on the lips of the lens we notice that
\begin{align} \nonumber
\widetilde P_-(z)^{-1} \widetilde P_+(z) &= 
\operatorname{diag}(1,z^\beta,\ldots,z^\beta) \\ \nonumber
& \qquad D_0(z)^{-n}
\left(\begin{pmatrix}
1 & 0\\
z^{-\beta} e^{2 n \varphi(z)} & 1
\end{pmatrix} 
\oplus \mathbb I_{(r-1)\times (r-1)}\right) D_0(z)^n\\ \nonumber
& \hspace{5cm} \operatorname{diag}(1,z^{-\beta},\ldots,z^{-\beta})\\ \label{ch4:eq:jumpstildeliplens}
&= \begin{pmatrix}
1 & 0\\
D_{0,11}(z)^{-1} e^{2 n \varphi(z)} D_{0,00}(z) & 1
\end{pmatrix} 
\oplus \mathbb I_{(r-1)\times (r-1)}
\end{align}
We see that
\begin{align*}
D_{0,11}(z)^{-n} e^{2 n \varphi(z)} D_{0,00}(z)^n
&= \exp 2n\left(\varphi(z) + \sum_{j=1}^{r-1} \varphi_j(z) - \sum_{j=0}^{r-1} \varphi_j(z)\right)\\
&= \exp 2n(\varphi(z) - \varphi_0(z)) = 1. 
\end{align*}
Plugging this in \eqref{ch4:eq:jumpstildeliplens}, we conclude that $\widetilde P$ satisfies RH-$\widetilde{\text{P}}$2. To see that it satisfies RH-$\widetilde{\text{P}}$3 we first notice that the components of $D_0(z)$ are bounded around $z=0$, which follows from Proposition \ref{ch4:prop:fmanalytic} for example. Now RH-$\widetilde{\text{P}}$3 follows from the observation that
\begin{align*}
z^{-\frac{r-1}{r}} h_{\alpha+\frac{r-1}{r}}(z) z^\beta = z^{-\alpha-\frac{r-1}{r}} h_{\alpha+\frac{r-1}{r}}(z) = h_{-\alpha-\frac{r-1}{r}}(z). 
\end{align*}
We conclude that $\widetilde P$ satisfies RH-$\widetilde{\text{P}}$. The opposite implication is more or less analogous. 
\end{proof}

We will eventually find a solution to RH-$\widetilde{\text{P}}$ in the form
\begin{align*}
\widetilde P(z) = \Psi\left(n^{r+1} f(z)\right).
\end{align*} 
Here $\Psi$ is a solution to a so-called bare parametrix problem, that we define explicitly in Section \ref{ch4:sec:bareParametrix}. This problem has the same behavior as in RH-$\widetilde{\text{P}}$, except that the jump contours are extended to infinity, i.e., we have jumps on the positive and negative real axis and on the positive and negative imaginary axis. We denote these by $\mathbb R^\pm$ and $i\mathbb R^\pm$. Additionally, $\Psi$ has a specific asymptotic behavior at $\infty$. The function $f$ is a conformal map with $f(0)=0$ that maps positive numbers to positive numbers. Without loss of generality, by slightly deforming the lips of the lens, we may assume that it maps the jump contours of RH-$\widetilde{\text{P}}$2 into the jump contours of $\Psi$.

\subsubsection{The conformal map $f$} \label{ch4:sec:conformalf}

The conformal map that we use in the construction of the local parametrix is defined as follows. 
\begin{definition} \label{ch4:def:conformalf}
We define the map
\begin{align} \label{ch4:eq:conformalf}
f(z) = z \left(\frac{2 f_1(z)}{(r+1)^2}\right)^{r+1}
\end{align}
\end{definition}

We will be more precise about the domain of $f$ in a moment. 

\begin{proposition} \label{ch4:prop:conformalf}
In a small enough neighborhood of $0$ the function $f$, as defined in \eqref{ch4:eq:conformalf}, is a conformal map with $f(0)=0$ that maps positive numbers to positive numbers. Furthermore, with $c_{0,V}$ as in \eqref{ch4:eq:onecutthetabehav}, we have
\begin{align} \label{ch4:eq:constf0}
f'(0) = \left(\frac{\pi c_{0,V}}{\sin\left(\frac{\pi}{r+1}\right)}\right)^{r+1}.
\end{align}
\end{proposition}

\begin{proof}
By Proposition \ref{ch4:prop:fmanalytic} we already know that $f_1$ is analytic on $D(0,q)\cap O_V$, then $f$, too, is analytic on $D(0,q)\cap O_V$. It suffices to prove that $f_1(0)>0$. If we subtract the formula in Proposition \ref{ch4:prop:sumfm} for $l=0$ from the one for $l=1$, then we obtain for $\operatorname{Im}(z)>0$ that
\begin{align} \label{ch4:eq:behavfr+1phi0}
2i \sum_{m=1}^r \sin\left(\frac{\pi m}{r+1}\right) z^\frac{m}{r+1} f_m(z) = (r+1) \varphi_0(z). 
\end{align}
By \eqref{ch4:eq:varphieven0} and the one-cut $\frac{1}{r}$-regularity of we have for $x>0$ that
\begin{align*}
\varphi_{0,+}(z) = \mu([0,x]) = (r+1)\pi i c_{0,V} x^\frac{1}{r+1} (1+o(1))
\end{align*}
as $x\to 0$. Then it follows for $x>0$ that
\begin{align*}
2 i \sin\left(\frac{\pi}{r+1}\right) x^\frac{1}{r+1} f_1(x) + \mathcal O\left(x^{\frac{2}{r+1}}\right) = (r+1)^2 \pi i c_{0,V} x^\frac{1}{r+1} (1+o(1))
\end{align*}
as $z\to 0$. This can only be true if 
\begin{align*}
f_1(0) = (r+1)^2 \frac{\pi c_{0,V}}{2\sin\left(\frac{\pi}{r+1}\right)}. 
\end{align*}
Then we conclude that $f_1(0)>0$, and furthermore
\begin{align*}
f'(0) = \left(\frac{\pi c_{0,V}}{\sin\left(\frac{\pi}{r+1}\right)}\right)^{r+1}.
\end{align*}
\end{proof}

We infer that there exists an $r_0>0$ such that $f$ is a conformal map if we take it to have domain $D(0,r_0)$. We fix such an $r_0$ and we will use it as the radius of the disk on which our local parametrix is defined. 

\subsection{Bare local parametrix problem} \label{ch4:sec:bareParametrix}

As mentioned before the bare parametrix problem has the same behavior as in RH-$\widetilde{\text{P}}$, but with the jump contours extended to infinity, and with an additional behavior at $\infty$. It is the generalization of the RHP in Section 3 in \cite{KuMo}. The bare parametrix problem takes the following form. 

\begin{rhproblem} \label{ch4:RHPforQ} \
\begin{description}
\item[RH-$\Psi$1] $\Psi$ is analytic on $\mathbb C\setminus (\mathbb R \cup i\mathbb R)$.
\item[RH-$\Psi$2] $\Psi$ has boundary values for $x\in (\mathbb R \cup i\mathbb R)\setminus \{0\}$
\begin{align} \nonumber
\Psi_+(x) &= \Psi_-(x) \times \left\{\begin{array}{ll} 
\bigoplus_{j=1}^{\frac{r}{2}} 
\begin{pmatrix}
0 & 1\\
-1 & 0
\end{pmatrix}
\oplus 1, & r \equiv 0 \mod 2,\\
\bigoplus_{j=1}^{\frac{r+1}{2}} 
\begin{pmatrix}
0 & 1\\
-1 & 0
\end{pmatrix}
, & r \equiv 1 \mod 2,
\end{array}
\right. \\  \label{ch4:RHPsi2}
&\hspace{6cm} x \in \mathbb R^+,\\ \nonumber
\Psi_+(x) &= \Psi_-(x) \\ \nonumber
&\hspace{-0.6cm}\times \left\{\begin{array}{ll}
1 \oplus
\bigoplus_{j=1}^{\frac{r}{2}} e^{2\pi i\beta}
\begin{pmatrix}
0 & 1\\
-1 & 0
\end{pmatrix}, & r \equiv 0 \mod 2,\\
1\oplus
\bigoplus_{j=1}^{\frac{r-1}{2}} e^{2\pi i\beta}
\begin{pmatrix}
0 & 1\\
-1 & 0
\end{pmatrix}
\oplus e^{2\pi i\beta}, & r \equiv 1 \mod 2,\\
\end{array}
\right. \\
&\hspace{6cm} x \in \mathbb R^{-},\\ \nonumber
\Psi_+(x) &= \Psi_-(x)  
\begin{pmatrix}
1 & 0\\
1 & 1
\end{pmatrix} 
\oplus \mathbb I_{(r-1)\times (r-1)}\\ 
&\hspace{6cm} x \in i\mathbb R^{\pm}.
\end{align}
\item[RH-$\Psi$3] As $z\to\infty$ we have for $\pm\operatorname{Im}(z)>0$
\begin{multline} \label{ch4:RHPsi3}
\Psi_\alpha(z) = \left(\mathbb I + \mathcal O\left(\frac{1}{z}\right)\right) L_\alpha(z)\\
\left\{ \begin{array}{l} \displaystyle\bigoplus_{j=1}^\frac{r}{2} 
\begin{pmatrix} e^{-(r+1) \omega^{\pm (\frac{r}{2}-j)}z^\frac{1}{r+1}} & 0\\
0 & e^{-(r+1) \omega^{\mp (\frac{r}{2}-j)}z^\frac{1}{r+1}}
\end{pmatrix} \oplus e^{-(r+1)z^\frac{1}{r+1}}, \\ 
\hspace{7cm}r\equiv 0\mod 2,\\
\displaystyle\bigoplus_{j=1}^\frac{r+1}{2} 
\begin{pmatrix} e^{-(r+1) \omega^{\pm (\frac{r}{2}-j)}z^\frac{1}{r+1}} & 0\\
0 & e^{-(r+1) \omega^{\mp (\frac{r}{2}-j)}z^\frac{1}{r+1}}
\end{pmatrix}, \\ 
\hspace{7cm}r\equiv 1\mod 2, \end{array}\right.
\end{multline}
where $L_\alpha(z)$ is as in Definition \ref{ch4:def:Lalpha} below. 
\item[RH-$\Psi$4] As $z\to 0$
\begin{align}
\label{ch4:RHPsi4a}
\Psi(z) &= \mathcal{O}\begin{pmatrix} h_{-\alpha-\frac{r-1}{r}}(z) & \ldots &  h_{-\alpha-\frac{r-1}{r}}(z)\\ 
\vdots & & \vdots\\
h_{-\alpha-\frac{r-1}{r}}(z) & \ldots & h_{-\alpha-\frac{r-1}{r}}(z)\end{pmatrix},
& \operatorname{Re}(z)>0,\\ \label{ch4:RHPsi4b}
\Psi(z) &= \mathcal{O}\begin{pmatrix} 1 & h_{-\alpha-\frac{r-1}{r}}(z) & \ldots &  h_{-\alpha-\frac{r-1}{r}}(z)\\ 
\vdots & & & \vdots\\
1 & h_{-\alpha-\frac{r-1}{r}}(z) & \ldots &  h_{-\alpha-\frac{r-1}{r}}(z)\end{pmatrix},
& \operatorname{Re}(z) < 0.
\end{align}
\end{description}
\end{rhproblem}

It will be practical to choose a notation for the consecutive powers of $\omega$ in RH-$\Psi$3. We set
\begin{align} \label{ch4:eq:defkj}
k_j &= (-1)^{j-1} \left(\frac{r}{2} - \left\lfloor \frac{j-1}{2}\right\rfloor\right), & j=1,\ldots,r+1.
\end{align}
With this definition we can rewrite RH-$\Psi$3 as 
\begin{align} \label{ch4:RHPsi3rewritten}
\Psi(z) = \left(\mathbb I +\mathcal O\left(\frac{1}{z}\right)\right) L_\alpha(z) \bigoplus_{j=1}^{r+1} e^{-(r+1) \omega^{\pm k_j} z^\frac{1}{r+1}}
\end{align}
as $z\to\infty$, for $\pm\operatorname{Im}(z)>0$. 

\begin{definition} \label{ch4:def:Lalpha}
We define
\begin{align} \label{ch4:eq:defLalpha}
L_\alpha(z) = \frac{(2\pi)^\frac{r}{2}}{\sqrt{r+1}} z^{-\frac{r}{r+1}\beta}
\bigoplus_{j=0}^r z^{-\frac{r}{2(r+1)}+\frac{j}{r+1}} 
\left\{\begin{array}{ll} M_\alpha^+, & \operatorname{Im}(z)>0,\\
M_\alpha^-, & \operatorname{Im}(z)<0,
\end{array} \right.
\end{align}
where $M^+$ and $M^-$ are $(r+1)\times (r+1)$ matrices given by
\begin{align} \label{ch4:eq:defMalpha+}
\hspace{-0.5cm}\left(M_\alpha^+\right)_{kl} &= 
\operatorname{diag}(1,-1,1,\ldots) 
\begin{pmatrix}
1 & 1 & \cdots\\
\omega^{k_1} & \omega^{k_2} & \cdots & \omega^{k_{r+1}}\\
\omega^{2 k_1} & \omega^{2 k_2} & \cdots & \omega^{2 k_{r+1}}\\
\vdots & \vdots & & \vdots\\
\omega^{r k_1} & \omega^{r k_2} & \cdots & \omega^{r k_{r+1}}
\end{pmatrix}
\left(\bigoplus_{j=1}^{r+1} e^{2\pi i(\beta+\eta) k_j}\right)
\times \left(\operatorname{diag}\left(1, -1, 1,\ldots\right)\right)^r
\end{align}
\begin{align} \nonumber
\hspace{-0.5cm}\left(M_\alpha^-\right)_{kl} &= 
\operatorname{diag}(1,-1,1,\ldots) 
\begin{pmatrix}
1 & 1 & \cdots\\
\omega^{k_1} & \omega^{k_2} & \cdots & \omega^{k_{r+1}}\\
\omega^{2 k_1} & \omega^{2 k_2} & \cdots & \omega^{2 k_{r+1}}\\
\vdots & \vdots & & \vdots\\
\omega^{r k_1} & \omega^{r k_2} & \cdots & \omega^{r k_{r+1}}
\end{pmatrix}
\left(\bigoplus_{j=1}^{r+1} e^{2\pi i(\beta+\eta) k_j}\right)
\left(\operatorname{diag}\left(1, -1, 1,\ldots\right)\right)^r\\ \label{ch4:eq:defMalpha-}
& \hspace{2cm}\times\left\{\begin{array}{ll} 
\displaystyle\bigoplus_{j=1}^{\frac{r}{2}} 
\begin{pmatrix}
0 & -1\\
1 & 0
\end{pmatrix}
\oplus 1, & r \equiv 0 \mod 2,\\
\displaystyle\bigoplus_{j=1}^{\frac{r+1}{2}} 
\begin{pmatrix}
0 & -1\\
1 & 0
\end{pmatrix}
, & r \equiv 1 \mod 2,
\end{array}
\right.
\end{align}
where
\begin{align} \label{ch4:eq:defTheta}
\eta = -\frac{r}{r+1} \left(\beta+\frac{1}{2}\right). 
\end{align}
\end{definition}

We will eventually prove that RH-$\Psi$ has a unique solution. In the next section we construct the solution. 
It follows by standard arguments from Riemann-Hilbert theory that $\det \Psi(z)$ is a constant multiple of $z^{-r\beta}$.

\begin{remark}
RH-$\Psi$ shows a great similarity with the bare Meijer-G parametrix for $p$-chain from \cite[p. 34]{BeBo} (in our case $p=r$). Its solution is constructed with Meijer G-functions, and our solution will also be constructed with Meijer G-functions, as we shall see in the next section. It is natural to wonder if these bare parametrix problems are somehow related. In \cite{BeBo} the bare parametrix problem was obtained after a RH analysis where two lenses, rather than one lens, had to be opened. Another difference is that the jumps on the lenses are not simply a direct product of a $2\times 2$ block and an $(r+1)\times (r+1)$ unit matrix, as in our case. Initially, we suspected that it is possible to map RH-$\Psi$ to the bare Meijer-G parametrix problem, by artificially adding jumps. It appears that this does not work. Then perhaps, these two problems are actually inherently different.  
\end{remark}


\subsubsection{Definition of $\Psi$ with Meijer G-functions}
We will find a solution of RHP-$\Psi$ in terms of Meijer G-functions. See \cite{BeSm} for an introduction on Meijer G-functions. One may also consult Appendix \ref{ch:appendixB} for general information on the Meijer G-function. We will use the particular Meijer G-function defined by
\begin{align} \label{ch4:eq:defMeijerG}
G_{0,r+1}^{r+1,0}\left(\left. \begin{array}{c} -\\ 0,-\alpha,-\alpha-\frac{1}{r},\ldots,-\alpha-\frac{r-1}{r}\end{array}\right| z\right)
= \frac{1}{2\pi i} \int_L \Gamma\left(s\right) \prod_{j=0}^{r-1} \Gamma\left(s-\alpha-\frac{j}{r}\right) z^{-s} ds,
\end{align}
where $L$ encircles the interval $(-\infty,\max(0,\alpha+\frac{r-1}{r})]$ in the complex $s$-plane. 
We shall simply abbreviate the left-hand side of \eqref{ch4:eq:defMeijerG} with the notation $G_{0,r+1}^{r+1,0}(z)$, suppressing the parameters. $G_{0,r+1}^{r+1,0}(z)$ can be viewed as a multi-valued function. By \cite[p. 144]{Lu} the definition \eqref{ch4:eq:defMeijerG} makes sense for $z$ with argument between $-\frac{r+1}{2}\pi$ and $\frac{r+1}{2}\pi$. We will consider $G_{0,r+1}^{r+1}(z)$ as a function where $z$ has argument between $-\pi$ and $\pi$ though, to preserve our convention that (non-integer) powers of $z$ have a cut at $(-\infty,0]$, and that the argument should lie between $-\pi$ and $\pi$. If we want to consider values of $G_{0,r+1}^{r+1,0}$ for $z$ with argument outside this range, then we use a notation that we introduce in a moment. 

The function in \eqref{ch4:eq:defMeijerG} is a solution to the $(r+1)$-th order linear differential equation
\begin{align} \label{ch4:eq:MeijerGdvgl}
\vartheta \prod_{j=0}^{r-1} \left(\vartheta+\alpha+\frac{j}{r}\right) \psi(z) + (-1)^r z \psi(z) &= 0, & \vartheta = z\frac{d}{dz}.
\end{align}
Around $z=0$ a basis of solutions can be expressed in terms of generalized hypergeometric functions. A particular basis of solutions, practical for our purposes, to \eqref{ch4:eq:MeijerGdvgl} is given by
\begin{align} \label{ch4:eq:defpsij}
\psi_j(z) &= \gamma_j  G_{0,r+1}^{r+1,0}\left(z e^{2\pi i k_j}\right), & j=1,2,\ldots, r+1,
\end{align}
where
\begin{align} \nonumber
k_j &= (-1)^{j-1} \left(\frac{r}{2} - \left\lfloor \frac{j-1}{2}\right\rfloor\right),\\ \label{ch4:eq:defgammaj}
\gamma_j &= (-1)^{r (j-1)} e^{2\pi i\beta k_j}.
\end{align}
We have repeated the definition for $k_j$ (see \eqref{ch4:eq:defkj}) as a convenience to the reader. The notation $z e^{2\pi i k_j}$ means that we have analytically continued $G_{0,r+1}^{r+1,0}$ along a circle a total of $|k_j|$ times, in positive direction if $k_j>0$ and in negative direction if $k_j<0$. Notice that $k_j$ may be a half-integer. We remark that the functions $\psi_j$ have a cut at $(-\infty,0]$. 
It will in some cases be convenient to notice that we can write $k_j$ alternatively as
\begin{align} \label{ch4:eq:defkjAlternative}
k_j = \frac{(-1)^{j-1}\left(r+1 - 2\left\lfloor\frac{j}{2}\right\rfloor\right)-1}{2}.
\end{align}
Additionally, we define 
\begin{align} \label{ch4:eq:defpsi0}
\psi_0(z) &= \psi_1(z) + \psi_2(z)
= e^{\pi i r\beta} G_{0,r+1}^{r+1,0}(z e^{\pi i r}) + (-1)^r e^{-\pi i r\beta} G_{0,r+1}^{r+1,0}(z e^{-\pi i r}).
\end{align}

\begin{lemma} \label{ch4:prop:psi0entire}
$\psi_0$ is an entire function.
\end{lemma}

\begin{proof}
Using \eqref{ch4:eq:defpsi0} and \eqref{ch4:eq:defMeijerG} we find that
\begin{align} \nonumber
\psi_0(z) &= 
\frac{1}{2 \pi i} \int_L \Gamma(s) \prod_{j=0}^{r-1} \Gamma\left(s-\alpha-\frac{j}{r}\right)
\left(e^{\pi i r\beta} e^{-\pi i r s} z^{-s} + (-1)^r e^{-\pi i r\beta} e^{\pi i r s} z^{-s}\right) ds\\ \label{ch4:eq:residueCalc}
&=-\frac{i^{r-1}}{\pi} \int_L \Gamma(s) \prod_{j=0}^{r-1} \Gamma\left(s-\alpha-\frac{j}{r}\right)
\sin\left(\pi r (s-\alpha)\right) z^{-s} ds.
\end{align}
The sine factor removes all the poles of the Gamma functions to the right of the product symbol. That means that only the poles of $\Gamma(s)$ survive, and these are located at $\ldots, -2, -1, 0$. Then, using Euler's reflection formula for the Gamma function and the well-known trigonometric identity 
\begin{align*}
\sin(r x) = 2^{r-1} \prod_{j=0}^{r-1} \sin\left(x+\frac{j \pi}{r}\right),
\end{align*}
we find with a residue calculation of \eqref{ch4:eq:residueCalc} that
\begin{align*}
\psi_0(z) &= -(2\pi i)^r \sum_{m=0}^\infty \frac{(-1)^{m(r+1)}}{m!} \left(\prod_{j=0}^{r-1} \frac{1}{\Gamma\left(1+m+\alpha+\frac{j}{r}\right)}\right) z^m\\
&= (-1)^{r+1} (2\pi i)^r \prod_{j=0}^{r-1} \frac{1}{\Gamma\left(1+\alpha+\frac{j}{r}\right)} 
{}_{0}F_{r}\left(\left. \begin{array}{c} -\\ 1+\alpha,1+\alpha+\frac{1}{r},\ldots,1+\alpha+\frac{r-1}{r}\end{array}\right| (-1)^{r+1} z\right),
\end{align*}
where the function ${}_{0}F_{r}$ is a generalized hypergeometric function. Since hypergeometric functions of this type are entire, we conclude that $\psi_0$ is an entire function. 
\end{proof}

\begin{definition} \label{ch4:def:Psi}
We define, with $\psi_0,\psi_1,\ldots,\psi_{r+1}$ as above,
\begin{align} \label{ch4:eq:defPsialpha}
\Psi_\alpha(z) = \left\{\begin{array}{ll}
\begin{pmatrix}
\psi_1(z) & \psi_2(z) & \psi_3(z) & \psi_4(z) & \ldots\\
\vartheta\psi_1(z) & \vartheta\psi_2(z) & \vartheta\psi_3(z) & \vartheta\psi_4(z) & \ldots\\
\vdots & & & & \vdots\\
\vartheta^r\psi_1(z) & \vartheta^r\psi_2(z) & \vartheta^r\psi_3(z) & \vartheta^r\psi_4(z) & \ldots
\end{pmatrix},
& 0<\operatorname{arg}(z) < \frac{\pi}{2},\\
\begin{pmatrix}
\psi_0(z) & \psi_2(z) & \psi_3(z) & \psi_4(z) & \ldots\\
\vartheta\psi_0(z) & \vartheta\psi_2(z) & \vartheta\psi_3(z) & \vartheta\psi_4(z) & \ldots\\
\vdots & & & & \vdots\\
\vartheta^r\psi_0(z) & \vartheta^r\psi_2(z) & \vartheta^r\psi_3(z) & \vartheta^r\psi_4(z) & \ldots\\
\end{pmatrix},
& \frac{\pi}{2} <\operatorname{arg}(z) < \pi,\\
\begin{pmatrix}
\psi_2(z) & -\psi_1(z) & \psi_{4}(z) & -\psi_3(z) & \ldots\\
\vartheta\psi_2(z) & -\vartheta\psi_1(z) &\vartheta \psi_4(z) & -\vartheta \psi_3(z) & \ldots\\
\vdots & & & & \vdots\\
\vartheta^r\psi_2(z) & -\vartheta^r\psi_1(z) &\vartheta^r \psi_4(z) & -\vartheta^r \psi_3(z) & \ldots
\end{pmatrix},
& -\frac{\pi}{2} <\operatorname{arg}(z) < 0,\\
\begin{pmatrix}
\psi_0(z) & -\psi_1(z) & \psi_{4}(z) & -\psi_3(z) & \ldots\\
\vartheta\psi_0(z) & -\vartheta\psi_1(z) &\vartheta \psi_4(z) & -\vartheta \psi_3(z) & \ldots\\
\vdots & & & & \vdots\\
\vartheta^r\psi_0(z) & -\vartheta^r\psi_1(z) &\vartheta^r \psi_4(z) & -\vartheta^r \psi_3(z) & \ldots
\end{pmatrix}
& -\pi<\operatorname{arg}(z) < -\frac{\pi}{2}.
\end{array}\right. 
\end{align}
\end{definition}

We will eventually prove that $\Psi_\alpha$ is the unique solution to RH-$\Psi$ (see Theorem \ref{ch4:thm:PsiaexistUnique}). 

\begin{lemma} \label{ch4:prop:PsiandQ}
$\Psi_\alpha$ satisfies RHP-$\Psi$1, RH-$\Psi$2 and RH-$\Psi$4.
\end{lemma}

\begin{proof}
For RHP-$\Psi$2 we only have to check that the jump on the negative axis is satisfied.  The other jumps are satisfied by construction. By Lemma \ref{ch4:prop:psi0entire} we know that $\psi_0$ does not have a jump on the negative axis. Hence there is no jump in the $1\times 1$ block in the upper-left corner, i.e., the corresponding component equals $1$. The last block, i.e., in the lower-right corner, is $2\times 2$ when $r\equiv 0\mod 2$ and $1\times 1$ when $r\equiv 1\mod 2$. Let us verify the jumps for the $2\times 2$ blocks, excluding the lower-right $2\times 2$ block when $r\equiv 0\mod 2$.  In the next few arguments we use that $G_{0,r+1}^{r+1,0}$ can be viewed as a multi-valued function (with argument between $-\frac{r+1}{2}\pi$ and $\frac{r+1}{2}$). Now let $x<0$. For every even $2\leq j\leq r$ we have 
\begin{align} \nonumber
-\psi_{j-1,+}(x) &= - \gamma_{j-1} G_{0,r+1}^{r+1,0}(|x| e^{-\pi i +2\pi i k_{j-1}})
= -\frac{\gamma_{j-1}}{\gamma_{j+1}} C_{j+1} G_{0,r+1}^{r+1,0}(|x| e^{\pi i +2\pi i k_{j+1}})\\ \label{ch4:eq:psij-1+psij+1}
&=  -\frac{\gamma_{j-1}}{\gamma_{j+1}} \psi_{j+1,-}(x),
\end{align}
and for every even $2\leq j\leq r-1$ we have
\begin{align} \nonumber
\psi_{j+2,+}(x) &= - \gamma_{j+2} G_{0,r+1}^{r+1,0}(|x| e^{-\pi i+ 2\pi i k_{j+2}})
= \frac{\gamma_{j+2}}{\gamma_j} C_j G_{0,r+1}^{r+1,0}(|x| e^{\pi i +2\pi i k_{j}})\\ \label{ch4:eq:psij-1+psij+1b}
&= \frac{\gamma_{j+2}}{\gamma_j} \psi_{j,-}(x).
\end{align}
Here we have used for the first equality that $k_{j+1}+1 = k_{j-1}$ and for the second equality that $k_{j+2}=k_{j}+1$, which is clear from \eqref{ch4:eq:defkj}. Indeed, using these two identities for the $k_j$, and \eqref{ch4:eq:defgammaj}, we also have
\begin{align} \label{ch4:eq:gammaj+1etc}
\frac{\gamma_{j-1}}{\gamma_{j+1}} = \frac{\gamma_{j+2}}{\gamma_j} = e^{2\pi i\beta}.
\end{align}
Combining \eqref{ch4:eq:psij-1+psij+1}, \eqref{ch4:eq:psij-1+psij+1b} and \eqref{ch4:eq:gammaj+1etc}, we conclude that we get the correct jump in the $2\times 2$ blocks, i.e., for every $1\leq j\leq r-1$
\begin{align*}
\begin{pmatrix} 
-\psi_{j-1,+}(x) & \psi_{j+2,+}(x)
\end{pmatrix}
=  
\begin{pmatrix}
\psi_{j,-}(x) & \psi_{j+1,-}(x)
\end{pmatrix}
\begin{pmatrix}
0 & e^{2\pi i\beta}\\
-e^{2\pi i\beta} & 0
\end{pmatrix}
\end{align*}
Let us move on to the lower-right block. Let us first consider the case where $r\equiv 0\mod 2$.
Then the last block is a $2\times 2$ block. We should have
\begin{align} \label{ch4:eq:psilast2x2block}
\begin{pmatrix} 
-\psi_{r-1,+}(x) & \psi_{r+1,+}(x)
\end{pmatrix}
= 
\begin{pmatrix}
\psi_{r,-}(x) & \psi_{r+1,-}(x)
\end{pmatrix}
\begin{pmatrix}
0 & e^{2\pi i\beta}\\
-e^{2\pi i\beta} & 0
\end{pmatrix}
\end{align}
Then it suffices (because the other case has already been treated in \eqref{ch4:eq:psij-1+psij+1}) to show that
\begin{align*}
\psi_{r+1,+}(x) = e^{2\pi i\beta} \psi_{r,-}(x).
\end{align*}
Indeed, we have
\begin{align*}
\psi_{r+1,+}(x) &= \gamma_{r+1} G_{0,r+1}^{r+1,0}(|x| e^{-\pi i+2\pi i k_{r+1}})
= \frac{\gamma_{r+1}}{\gamma_r} \gamma_r G_{0,r+1}^{r+1,0}(|x| e^{\pi i+2\pi i k_{r}})\\
&= \frac{\gamma_{r+1}}{\gamma_r} \psi_{r,-}(x),
\end{align*}
where we have used that $k_{r+1}=0=k_r + 1$ according to \eqref{ch4:eq:defkj}. Indeed, using \eqref{ch4:eq:defgammaj}, we also have
\begin{align*}
\frac{\gamma_{r+1}}{\gamma_r} = \frac{1}{e^{-2\pi i\beta}} = e^{2\pi i\beta},
\end{align*}
as it should be, and we obtain \eqref{ch4:eq:psilast2x2block}. Now we consider the case where $r\equiv 1\mod 2$. Then the last block is a $1\times 1$ block. Indeed, we have
\begin{align*}
-\psi_{r,+}(x) &= -\gamma_{r} G_{0,r+1}^{r+1,0}(|x| e^{-\pi i+2\pi i k_{r}})
= -\frac{\gamma_{r}}{\gamma_{r+1}} \gamma_{r+1} G_{0,r+1}^{r+1,0}(|x| e^{\pi i+2\pi i k_{r+1}})\\
&=  -\frac{\gamma_{r}}{\gamma_{r+1}} \psi_{r+1,+}(x).
\end{align*}
Here we have used that $k_{r+1}=-\frac{1}{2}$ and $k_r=\frac{1}{2}$, as follows from \eqref{ch4:eq:defkj}. Indeed, by \eqref{ch4:eq:defgammaj}, we have that
\begin{align*}
\frac{\gamma_{r}}{\gamma_{r+1}} = \frac{e^{\pi i\beta}}{-e^{-\pi i\beta}} = -e^{2\pi i\beta}. 
\end{align*}
This proves that RH-$\Psi$2 is satisfied. We should still prove that RH-$\Psi$4 is satisfied. The $\psi_j$ are solutions to the linear differential equation \eqref{ch4:eq:MeijerGdvgl} of order $r+1$. We can write them in the corresponding Frobenius basis. Since the indicial equation has roots $0,-\alpha, -\alpha-\frac{1}{r},\ldots,-\alpha-\frac{r-1}{r}$, we infer that all solutions behave as $\mathcal O(z^{-\alpha-\frac{r-1}{r}})$ as $z\to 0$. It remains to show that the behavior in the first column for $\operatorname{Re}(z)<0$ is only $\mathcal O(1)$ as $z\to 0$. This is a direct consequence of Lemma \ref{ch4:prop:psi0entire} however. 
\end{proof}

\subsubsection{Asymptotics of $\Psi$ as $z\to\infty$}

We also investigate the asymptotics of $\Psi_\alpha$ as $z\to\infty$. To this end we first collect some properties of $L_\alpha$ as defined in Definition \ref{ch4:def:Lalpha}. 

\begin{proposition} \label{ch4:eq:sameJumpsLPsi}
$L_\alpha$ has the same jumps as $\Psi_\alpha$ has on the positive and negative real axis. 
\end{proposition}

\begin{proof}
The jump on the positive real axis is satisfied by construction, so let us focus on the jump on the negative ray. Set $x<0$. We have that 
\begin{align} \nonumber
L_{\alpha,-}(x)^{-1} L_{\alpha,+}(x) &= (M_\alpha^+)^{-1}
e^{2\pi i\frac{r}{r+1}\beta}  
\left(\bigoplus_{j=0}^r e^{\pi i \frac{r}{r+1}-\frac{2\pi i j}{r+1}}\right)
M_\alpha^-\\ \label{ch4:eq:LadirectprodLa}
&= e^{2\pi i \frac{r}{r+1}(\beta+\frac{1}{2})} (M_\alpha^+)^{-1}
\operatorname{diag}\left(1, \omega^{-1}, \omega^{-2}, \ldots\right)
M_\alpha^-.
\end{align}
The factor $\operatorname{diag}(1,-1,1,\ldots)$ in the left of \eqref{ch4:eq:defMalpha+} and \eqref{ch4:eq:defMalpha-} has no effect on the jump. We notice that
\begin{align*}
\begin{pmatrix}
1 & 1 & \cdots & 1\\
\omega^{k_1} & \omega^{k_2} & \cdots & \omega^{k_{r+1}}\\
\omega^{2 k_1} & \omega^{2 k_2} & \cdots & \omega^{2 k_{r+1}}\\
\vdots & \vdots & & \vdots\\
\omega^{r k_1} & \omega^{r k_2} & \cdots & \omega^{r k_{r+1}}
\end{pmatrix}^{-1}
= \frac{1}{r+1}
\begin{pmatrix}
1 & \omega^{-k_1} & \omega^{-2 k_1} & \cdots & \omega^{-r k_1}\\
1 & \omega^{-k_2} & \omega^{-2 k_2} & \cdots & \omega^{-r k_2}\\
\vdots & \vdots & \vdots & & \vdots\\
1 & \omega^{-k_{r+1}} & \omega^{-2 k_{r+1}} & \cdots & \omega^{-r k_{r+1}}
\end{pmatrix}.
\end{align*}
Then we see that
\begin{multline}
\begin{pmatrix}
1 & 1 & \cdots & 1\\
\omega^{k_1} & \omega^{k_2} & \cdots & \omega^{k_{r+1}}\\
\omega^{2 k_1} & \omega^{2 k_2} & \cdots & \omega^{2 k_{r+1}}\\
\vdots & \vdots & & \vdots\\
\omega^{r k_1} & \omega^{r k_2} & \cdots & \omega^{r k_{r+1}}
\end{pmatrix}^{-1}
\operatorname{diag}\left(1, \omega^{-1}, \omega^{-2}, \ldots\right)
\begin{pmatrix}
1 & 1 & \cdots & 1\\
\omega^{k_1} & \omega^{k_2} & \cdots & \omega^{k_{r+1}}\\
\omega^{2 k_1} & \omega^{2 k_2} & \cdots & \omega^{2 k_{r+1}}\\
\vdots & \vdots & & \vdots\\
\omega^{r k_1} & \omega^{r k_2} & \cdots & \omega^{r k_{r+1}}
\end{pmatrix}\\ \label{ch4:eq:weirdPermutationMatrix}
= \frac{1}{r+1}
\begin{pmatrix}
1 & \omega^{-k_1} & \omega^{-2 k_1} & \cdots & \omega^{-r k_1}\\
1 & \omega^{-k_2} & \omega^{-2 k_2} & \cdots & \omega^{-r k_2}\\
\vdots & \vdots & \vdots & & \vdots\\
1 & \omega^{-k_{r+1}} & \omega^{-2 k_{r+1}} & \cdots & \omega^{-r k_{r+1}}
\end{pmatrix}
\begin{pmatrix}
1 & 1 & \cdots\\
\omega^{k_1-1} & \omega^{k_2-1} & \cdots & \omega^{k_{r+1}-1}\\
\omega^{2(k_1-1)} & \omega^{2 (k_2-1)} & \cdots & \omega^{2 (k_{r+1}-1)}\\
\vdots & \vdots & & \vdots\\
\omega^{r(k_1-1)} & \omega^{r (k_2-1)} & \cdots & \omega^{r (k_{r+1}-1)}
\end{pmatrix}
\end{multline}
Looking at \eqref{ch4:eq:LadirectprodLa} and the definitions \eqref{ch4:eq:defMalpha+} and \eqref{ch4:eq:defMalpha-} of $M_{\alpha}^\pm$, it is important to investigate the matrix in \eqref{ch4:eq:weirdPermutationMatrix}. The $(j,l)$ component of the matrix in \eqref{ch4:eq:weirdPermutationMatrix} is given by
\begin{align*}
\frac{1}{r+1}\sum_{m=0}^{r} \omega^{-m k_j} \omega^{m (k_l-1)}
= \frac{1}{r+1} \sum_{m=0}^{r} \omega^{(k_l-k_j-1) m}
\end{align*}
This sum is $1$ when $k_j-k_l-1\equiv 0\mod (r+1)$ and $0$ otherwise. For every $1\leq j\leq r+1$ there is exactly one $1\leq l\leq r+1$ such that $k_j-k_l-1\equiv 0\mod (r+1)$ is satisfied. This is because $\{k_1,k_2,\ldots,k_{r+1}\}=\{-\frac{r}{2},-\frac{r}{2}+1,\ldots,\frac{r}{2}\}$, i.e., there cannot be two different solutions $l$ since the difference between the two corresponding $k_j$'s would have to be more than $r$ apart. Then we infer that \eqref{ch4:eq:weirdPermutationMatrix} is a permutation matrix. In fact, we claim that \eqref{ch4:eq:weirdPermutationMatrix} equals the permutation matrix
\begin{align} \label{ch4:eq:weirdPermutationMatrix1}
\left(1 \oplus \bigoplus_{j=1}^\frac{r}{2} \begin{pmatrix} 0 & 1\\ 1 & 0\end{pmatrix}\right)
\left(\bigoplus_{j=1}^\frac{r}{2} \begin{pmatrix} 0 & 1\\ 1 & 0\end{pmatrix} \oplus 1\right)
\end{align}
when $r \equiv 0\mod 2$ and 
\begin{align} \label{ch4:eq:weirdPermutationMatrix2}
\left(1 \oplus \bigoplus_{j=1}^\frac{r-1}{2} \begin{pmatrix} 0 & 1\\ 1 & 0\end{pmatrix} \oplus 1\right)
\left(\bigoplus_{j=1}^\frac{r+1}{2} \begin{pmatrix} 0 & 1\\ 1 & 0\end{pmatrix}\right)
\end{align}
when $r \equiv 1\mod 2$. To prove this we remark that $k_j-k_l-1\equiv 0\mod (r+1)$ happens exactly when
\begin{align} \label{ch4:eq:weirdRelation}
(-1)^l \left\lfloor\frac{l}{2}\right\rfloor - (-1)^j \left\lfloor\frac{j}{2}\right\rfloor \equiv 1 \mod (r+1). 
\end{align}
We have used \eqref{ch4:eq:defkjAlternative} to arrive at \eqref{ch4:eq:weirdRelation}.

We first consider the case that $r \equiv 0\mod 2$. The way to argue this is as follows. Since both \eqref{ch4:eq:weirdPermutationMatrix} and \eqref{ch4:eq:weirdPermutationMatrix1} are permutation matrices, it suffices to show that their non-zero elements, i.e., the ones, are in the same position. So let us consider a non-zero entry of \eqref{ch4:eq:weirdPermutationMatrix1}. Let us assume that it is in row $j$, with $j$ ranging from $1$ to $r+1$.

If $j$ is even and $j<r$, then our non-zero entry is made by the $(j,j+1)$ component of the left matrix in \eqref{ch4:eq:weirdPermutationMatrix1}. Then it has to be multiplied by a component in the $j+1$-th row of the right matrix in \eqref{ch4:eq:weirdPermutationMatrix}. The only non-zero component there is at position $(j+1,j+2)$. So we have $l=j+2$. Hence we see that \eqref{ch4:eq:weirdRelation} is satisfied. If $j=r$, then we should have that $l=r+1$, and then \eqref{ch4:eq:weirdRelation} is also satisfied. 

Let us now consider the case where $j$ is odd and $j\geq 3$. Then the non-zero component in the $j$-th row of the left matrix in \eqref{ch4:eq:weirdPermutationMatrix1} is at position $(j,j-1)$. Then this component is multiplied by the non-zero component in the $j-1$-th row of the right matrix in \eqref{ch4:eq:weirdPermutationMatrix1}, which is at position $(j-1,j-2)$. Thus we infer that $l=j-2$ and again we see that \eqref{ch4:eq:weirdRelation} is satisfied. In the case that $j=1$ we find that $l=2$ and then \eqref{ch4:eq:weirdRelation} is also satisfied. 

The claim is proved for $r\equiv 0\mod 2$. Replacing \eqref{ch4:eq:weirdPermutationMatrix} by \eqref{ch4:eq:weirdPermutationMatrix1}, and then inserting it into \eqref{ch4:eq:LadirectprodLa} with the help of \eqref{ch4:eq:defMalpha+} and \eqref{ch4:eq:defMalpha-}, we must conclude that
\begin{align} \nonumber
& L_{\alpha,-}(x)^{-1} L_{\alpha,+}(x)\\ \nonumber
&= e^{2\pi i \frac{r}{r+1}(\beta+\frac{1}{2})} 
\left(\bigoplus_{j=1}^{r+1} e^{2\pi i(\beta+\eta) k_j}\right)^{-1}
\left(1 \oplus \bigoplus_{j=1}^\frac{r}{2} \begin{pmatrix} 0 & 1\\ -1 & 0\end{pmatrix}\right)\\ \nonumber
& \hspace{2.2cm}\left(\bigoplus_{j=1}^\frac{r}{2} \begin{pmatrix} e^{2\pi i (\beta+\eta) k_{2j}} & 0\\ 0 & e^{2\pi i (\beta+\eta) k_{2j-1}}\end{pmatrix} \oplus e^{2\pi i (\beta+\eta) k_{r+1}}\right)\\ \label{ch4:eq:jumpwithkjs}
&= e^{2\pi i \frac{r}{r+1}(\beta+\frac{1}{2})}  
\left(e^{2\pi i(\beta+\eta) (k_2-k_1)} \oplus \bigoplus_{j=1}^\frac{r}{2} \begin{pmatrix} 0 \hspace{1.4cm} e^{2\pi i(\beta+\eta) (k_{2j+2}-k_{2j})}\\ -e^{2\pi i(\beta+\eta) (k_{2j-1}-k_{2j+1})} \hspace{0.8cm} 0\end{pmatrix}\right)
\end{align}
We have $k_2-k_1 = -\frac{r}{2}-\frac{r}{2} = -r$ and, as one can check with \eqref{ch4:eq:defkj}, we have for all $j=1,\ldots,r-1$
\begin{align*}
(-1)^{j} ( k_{j+2}-k_j) = 1. 
\end{align*}
Indeed, using the definition of $\eta$ in \eqref{ch4:eq:defTheta} we have
\begin{align*}
\frac{r}{r+1} (\beta+\frac{1}{2}) - r (\beta+ \eta) = \frac{r}{2}.
\end{align*}
and we have
\begin{align*}
\frac{r}{r+1} (\beta+\frac{1}{2}) + \beta+ \eta = \beta,
\end{align*}
Thus, plugging these in \eqref{ch4:eq:jumpwithkjs}, we obtain the correct jump \eqref{ch4:RHPsi2} when $r \equiv 0 \mod 2$. 
A similar argument works for $r \equiv 1\mod 2$, we omit the details. 
\end{proof}

\begin{lemma} \label{ch4:prop:behavInftyPsi}
For $\pm\operatorname{Im}(z)>0$ we have as $z\to\infty$ the asymptotic expansion
\begin{align} \label{ch4:eq:prop:behavInftyPsi}
\Psi_\alpha(z) 
\sim  L_\alpha(z) \sum_{m=0}^\infty A_{\alpha,m} z^{-\frac{m}{r+1}}
\bigoplus_{j=1}^{r+1} e^{-(r+1) \omega^{\pm k_j} z^\frac{1}{r+1}} 
\end{align}
for $(r+1)\times (r+1)$ matrices $A_{\alpha,0}=\mathbb I$ and $A_{\alpha,1}, A_{\alpha,2}, \ldots$ that depend only on $\alpha$. 
\end{lemma}

\begin{proof}
Due to \cite{Lu} (Theorem 5 on p. 179)  we have as $z\to\infty$ that
\begin{align} \label{ch4:eq:MeijerGasympLuke}
G_{0,r+1}^{r+1,0}(z) \sim \frac{(2\pi)^\frac{r}{2}}{\sqrt{r+1}} e^{-(r+1) z^\frac{1}{r+1}} z^\eta \sum_{m=0}^\infty M_m z^{-\frac{m}{r+1}},
\end{align}
where $\eta$ is as in \eqref{ch4:eq:defTheta} and where $M_m$ are some coefficients that depend on $\alpha$, and $M_0=1$. 
The $M_m$ can be calculated with a cumbersome procedure, in principle. 
The asymptotic expansion \eqref{ch4:eq:MeijerGasympLuke} allows us to find the asymptotic behavior of the $\psi_j$. Namely, we have as $z\to\infty$
\begin{align} \label{ch4:eq:notvarthetapsij}
\psi_j(z) \sim (-1)^{r(j-1)} e^{2\pi i(\beta+\eta) k_j} \frac{(2\pi)^\frac{r}{2}}{\sqrt{r+1}} 
 e^{-(r+1) \omega^{k_j} z^\frac{1}{r+1}} z^\eta \sum_{m=0}^\infty M_m \omega^{-k_j m} z^{-\frac{m}{r+1}}.
\end{align}
These asymptotics follow simply by analytic continuation of \eqref{ch4:eq:MeijerGasympLuke}. Let us now look at the derivatives. Applying powers of the operator $\vartheta$ to \eqref{ch4:eq:notvarthetapsij}, we infer that there exist coefficients $M^{[l]}_m$ such that
\begin{align} \label{ch4:eq:varthetapsij}
\vartheta^l\psi_j(z) \sim
(-1)^{r(j-1)} e^{2\pi i(\beta+\eta) k_j} \frac{(2\pi)^\frac{r}{2}}{\sqrt{r+1}}
e^{-(r+1) \omega^{k_j} z^\frac{1}{r+1}} \omega^{k_j l} z^{\eta+\frac{l}{r+1}} \sum_{m=0}^\infty M_m^{[l]} \omega^{-k_j m} z^{-\frac{m}{r+1}}.
\end{align}
In fact, it is not hard to see that we should have $M_0^{[l]}=(-1)^l$. A closed form expression for the $M^{[l]}_m$ with $m>0$ seems hard to find, but we will not need it anyway.
Plugging \eqref{ch4:eq:varthetapsij} in the definition of $\Psi_\alpha$ of the first quadrant, we have as $z\to\infty$
\begin{multline*}
\Psi_\alpha(z) \sim \frac{(2\pi)^\frac{r}{2}}{\sqrt{r+1}} z^\eta\\
\hspace{-0.25cm}\begin{pmatrix}
\sum_{m=0}^\infty M^{[0]}_m \omega^{-k_1 m} z^{-\frac{m}{r+1}} & \sum_{m=0}^\infty M^{[0]}_m \omega^{-k_2 m} z^{-\frac{m}{r+1}} & \cdots\\
z^\frac{1}{r+1}\sum_{m=0}^\infty M^{[1]}_m \omega^{k_1 (1-m)} z^{-\frac{m}{r+1}} & z^\frac{1}{r+1}\sum_{m=0}^\infty M^{[1]}_m \omega^{k_2 (1-m)} z^{-\frac{m}{r+1}} & \cdots\\
z^\frac{2}{r+1}\sum_{m=0}^\infty M^{[2]}_m \omega^{k_1 (2-m)} z^{-\frac{m}{r+1}} & z^\frac{2}{r+1}\sum_{m=0}^\infty M^{[2]}_m \omega^{k_2 (2-m)} z^{-\frac{m}{r+1}} & \cdots\\
\vdots & \vdots &  \\
z^\frac{2}{r+1}\sum_{m=0}^\infty M^{[r]}_m \omega^{k_1 (r-m)} z^{-\frac{m}{r+1}} & z^\frac{2}{r+1}\sum_{m=0}^\infty M^{[r]}_m \omega^{k_2 (r-m)} z^{-\frac{m}{r+1}} & \cdots
\end{pmatrix}\\
\left(\operatorname{diag}\left(1, -1, 1,\ldots, (-1)^r\right)\right)^r 
\left(\bigoplus_{j=1}^{r+1} e^{2\pi i(\beta+\eta) k_j}\right)
\bigoplus_{j=1}^{r+1} e^{-(r+1) \omega^{k_j} z^\frac{1}{r+1}}.
\end{multline*}
We can rewrite this as
\begin{multline} \label{ch4:eq:PsialphaMatrix}
\Psi_\alpha(z) \sim \frac{(2\pi)^\frac{r}{2}}{\sqrt{r+1}} z^{\eta+\frac{r}{2(r+1)}} 
\left(\bigoplus_{j=0}^r z^{-\frac{r}{2(r+1)}+\frac{j}{r+1}}\right)\\
\begin{pmatrix}
\sum_{m=0}^\infty M^{[0]}_m \omega^{-k_1 m} z^{-\frac{m}{r+1}} & \sum_{m=0}^\infty M^{[0]}_m \omega^{-k_2 m} z^{-\frac{m}{r+1}} & \cdots & \sum_{m=0}^\infty M^{[0]}_m \omega^{-k_{r+1} m} z^{-\frac{m}{r+1}}\\
\sum_{m=0}^\infty M^{[1]}_m \omega^{k_1 (1-m)} z^{-\frac{m}{r+1}} & \sum_{m=0}^\infty M^{[1]}_m \omega^{k_2 (1-m)} z^{-\frac{m}{r+1}} & \cdots & \sum_{m=0}^\infty M^{[1]}_m \omega^{k_{r+1} (1-m)} z^{-\frac{m}{r+1}}\\
\sum_{m=0}^\infty M^{[2]}_m \omega^{k_1 (2-m)} z^{-\frac{m}{r+1}} & \sum_{m=0}^\infty M^{[2]}_m \omega^{k_2 (2-m)} z^{-\frac{m}{r+1}} & \cdots & \sum_{m=0}^\infty M^{[2]}_m \omega^{k_{r+1} (2-m)} z^{-\frac{m}{r+1}}\\
\vdots & \vdots & & \vdots\\
\sum_{m=0}^\infty M^{[r]}_m \omega^{k_1 (r-m)} z^{-\frac{m}{r+1}} & \sum_{m=0}^\infty M^{[r]}_m \omega^{k_2 (r-m)} z^{-\frac{m}{r+1}} & \cdots & \sum_{m=0}^\infty M^{[r]}_m \omega^{k_{r+1} (r-m)} z^{-\frac{m}{r+1}}
\end{pmatrix}\\
\left(\operatorname{diag}\left(1, -1, 1,\ldots, (-1)^r\right)\right)^r
\left(\bigoplus_{j=1}^{r+1} e^{2\pi i(\beta+\eta) k_j}\right)
\bigoplus_{j=1}^{r+1} e^{-(r+1) \omega^{k_j} z^\frac{1}{r+1}}.
\end{multline}
Using $M_0^{[l]}=(-1)^l$ we see that the matrix in the second line of \eqref{ch4:eq:PsialphaMatrix} equals
\begin{multline*}
\operatorname{diag}(1,-1,1,\ldots, (-1)^r) 
\begin{pmatrix}
1 & 1 & \cdots & 1\\
\omega^{k_1} & \omega^{k_2} & \cdots & \omega^{k_{r+1}}\\
\omega^{2 k_1} & \omega^{2 k_2} & \cdots & \omega^{2 k_{r+1}}\\
\vdots & \vdots & & \vdots\\
\omega^{r k_1} & \omega^{r k_2} & \cdots & \omega^{r k_{r+1}}
\end{pmatrix}\\
+ 
\operatorname{diag}(M^{[0]}_1,M^{[1]}_1,\ldots) 
\begin{pmatrix}
\omega^{-k_1} & \omega^{-k_2} & \cdots & \omega^{-k_{r+1}}\\
1 & 1 & \cdots & 1\\
\omega^{k_1} & \omega^{k_2} & \cdots & \omega^{k_{r+1}}\\
\vdots & \vdots & & \vdots\\
\omega^{(r-1) k_1} & \omega^{(r-1) k_2} & \cdots & \omega^{(r-1) k_{r+1}}
\end{pmatrix} z^{-\frac{1}{r+1}} + \ldots
\end{multline*}
Then there exist coefficients $\mathring A_{\alpha,0}=\mathbb I$ and $\mathring A_{\alpha,1}, \mathring A_{\alpha,2},\ldots$ such that the matrix in the second line of \eqref{ch4:eq:PsialphaMatrix} equals as a formal series
\begin{align*}
\operatorname{diag}(1,-1,1,\ldots, (-1)^r) 
\begin{pmatrix}
1 & 1 & \cdots & 1\\
\omega^{k_1} & \omega^{k_2} & \cdots & \omega^{k_{r+1}}\\
\omega^{2 k_1} & \omega^{2 k_2} & \cdots & \omega^{2 k_{r+1}}\\
\vdots & \vdots & & \vdots\\
\omega^{r k_1} & \omega^{r k_2} & \cdots & \omega^{r k_{r+1}}
\end{pmatrix}
\sum_{m=0}^\infty \mathring A_{\alpha,m} z^{-\frac{m}{r+1}}.
\end{align*}
Using that $\eta+\frac{r}{2(r+1)} = -\frac{r}{r+1}\beta$ (see \eqref{ch4:eq:defTheta}) and comparing with Definition \ref{ch4:def:Lalpha}, we see now that \eqref{ch4:eq:PsialphaMatrix} turns into
\begin{align*}
\Psi_\alpha(z) \sim L_\alpha(z) \sum_{m=0}^\infty A_{\alpha,m} z^{-\frac{m}{r+1}}
\bigoplus_{j=1}^{r+1} e^{-(r+1) \omega^{k_j} z^\frac{1}{r+1}},
\end{align*}
as $z\to\infty$, where
\begin{multline*}
A_{\alpha,m} = \left(\operatorname{diag}\left(1, -1, 1,\ldots, (-1)^r\right)\right)^r
\left(\bigoplus_{j=1}^{r+1} e^{2\pi i(\beta+\eta) k_j}\right)^{-1} \mathring A_{\alpha,m}\\
\left(\bigoplus_{j=1}^{r+1} e^{2\pi i(\beta+\eta) k_j}\right) 
\left(\operatorname{diag}\left(1, -1, 1,\ldots, (-1)^r\right)\right)^r.
\end{multline*}
As one can verify, we get a similar expression in the three other quadrants, with the same $A_{\alpha,m}$. 
\end{proof}

In fact, Lemma \ref{ch4:prop:behavInftyPsi} can be improved upon. The next proposition, combined with Lemma \ref{ch4:prop:behavInftyPsi2}, shows us that $\Psi$ solves RH-$\Psi$.

\begin{lemma} \label{ch4:prop:behavInftyPsi2}
For $\pm\operatorname{Im}(z)>0$ we have as $z\to\infty$ the asymptotic expansion
\begin{align} \label{ch4:eq:propbehavInftyPsi2}
\Psi_\alpha(z) \sim 
\sum_{m=0}^\infty \frac{C_{\alpha,m}}{z^m}
L_\alpha(z)
\bigoplus_{j=1}^{r+1} e^{-(r+1) \omega^{\pm k_j} z^\frac{1}{r+1}}
\end{align}
for $(r+1)\times (r+1)$ matrices $C_{\alpha,0}=\mathbb I$ and $C_{\alpha,1}, C_{\alpha,2}, \ldots$ that depend only on $\alpha$. 
\end{lemma}

\begin{proof}
It's a simple exercise to verify that
\begin{align*}
\Psi_\alpha(z) &\bigoplus_{j=1}^{r+1} e^{(r+1) \omega^{\pm k_j} z^\frac{1}{r+1}}, & \pm\operatorname{Im}(z)>0,
\end{align*}
has the same jumps as $\Psi_\alpha(z)$ has on the positive and negative real axis. Then it follows from \text{Proposition \ref{ch4:eq:sameJumpsLPsi}} that
\begin{align*}
L_\alpha(z)^{-1} & \Psi_\alpha(z) \bigoplus_{j=1}^{r+1} e^{(r+1) \omega^{\pm k_j} z^\frac{1}{r+1}}, & \pm\operatorname{Im}(z)>0,
\end{align*}
has no jumps on the positive and negative real axis. This can only be true if the expansion in \eqref{ch4:eq:prop:behavInftyPsi} consists solely of integer powers of $z$, i.e., we have for $\pm\operatorname{Im}(z)>0$
\begin{align*}
\Psi_\alpha(z) \sim L_\alpha(z) \sum_{m=0}^\infty \frac{A_{\alpha,(r+1) m}}{z^{m}}
\bigoplus_{j=1}^{r+1} e^{-(r+1) \omega^{\pm k_j} z^\frac{1}{r+1}}
\end{align*}
as $z\to\infty$. We can rewrite this as
\begin{align} \nonumber
\Psi_\alpha(z) \sim& L_\alpha(z) \sum_{m=0}^\infty \frac{A_{\alpha,(r+1) m}}{z^{m}} L_\alpha(z)^{-1}\\ \label{ch4:eq:PsialphaSimintegerpowers}
& \hspace{2cm} L_\alpha(z)
\bigoplus_{j=1}^{r+1} e^{-(r+1) \omega^{\pm k_j} z^\frac{1}{r+1}}
\end{align}
as $z\to\infty$, for $\pm\operatorname{Im}(z)>0$. 
For any positive integer $k$ we know that $L_\alpha(z)^{-1}$ and 
\begin{align*}
\left(\mathbb I + \frac{A_{\alpha,r+1}}{z}+ \frac{A_{\alpha,2(r+1)}}{z^2} +\ldots+\frac{A_{\alpha,k(r+1)}}{z^k}\right) L_\alpha(z)^{-1}
\end{align*}
have the same jumps on the positive and negative real axis. Thus it follows that
\begin{align*}
L_\alpha(z) \left(\mathbb I + \frac{A_{\alpha,r+1}}{z}+ \frac{A_{\alpha,2(r+1)}}{z^2} +\ldots+\frac{A_{\alpha,k(r+1)}}{z^k}\right) L_\alpha(z)^{-1}
\end{align*}
has no jumps. This means that there exist $C_{\alpha,0}=\mathbb I$ and $C_{\alpha,1}, C_{\alpha,2},\ldots$ such that
\begin{align*}
L_\alpha(z) \sum_{m=0}^\infty \frac{A_{\alpha,(r+1) m}}{z^{m}} L_\alpha(z)^{-1}
= \sum_{m=0}^\infty \frac{C_{\alpha,m}}{z^m}
\end{align*}
as formal series, and, inserting this in \eqref{ch4:eq:PsialphaSimintegerpowers}, we are done. 
\end{proof}

\subsubsection{Asymptotic behavior of $\Psi^{-1}$ as $z\to 0$} \label{ch4:sec:inversePsi}

In principle, we may use the argumentation from \cite{KuMo} to find an expression for the inverse of $\Psi_\alpha$. The explicit construction then uses
\begin{align*} 
G_{0,r+1}^{r+1,0}\left(\left. \begin{array}{c} -\\ 0,\alpha,\alpha+\frac{1}{r},\ldots,\alpha+\frac{r-1}{r}\end{array}\right| -z\right)
\end{align*}
and its analytic continuations along circular arcs. In \cite{KuMo} we used the explicit form of $\Psi_\alpha(z)^{-1}$ to eventually show that the scaling limit for the correlation kernel coincides with \eqref{ch4:eq:scalingLimitIK}. This was technically not necessary (although it was a nice way of verification) as there is a faster way to show this. Since the associated formulae tend to become big for $r>2$, we opt to omit the explicit form of $\Psi_\alpha(z)^{-1}$. 

All that really turns out to be important to us about the inverse of $\Psi_\alpha$, is the asymptotic behavior of $\Psi_\alpha(z)^{-1}$ as $z\to 0$ in the left half-plane.

\begin{lemma}
For $\operatorname{Re}(z)<0$ we have as $z\to 0$ that
\begin{align} \label{ch4:eq:behavPsiinv0}
\Psi_\alpha(z)^{-1} = \left\{\begin{array}{ll} 
\mathcal O
\begin{pmatrix}
1 & 1 & \hdots & 1\\
z^\alpha & z^\alpha & \hdots & z^\alpha\\
\vdots & & & \vdots\\
z^\alpha & z^\alpha & \hdots & z^\alpha
\end{pmatrix}, & \alpha\neq 0,\\
\mathcal O
\begin{pmatrix}
1 & 1 & \hdots & 1\\
\log z & \log z & \hdots & \log z\\
\vdots & & & \vdots\\
\log z & \log z & \hdots & \log z
\end{pmatrix}, & \alpha= 0.
\end{array}\right.
\end{align}
\end{lemma}
\begin{proof}
We only prove it for $\alpha\neq 0$. In the $j$-th quadrant we have a connection matrix $\Gamma_j$ such that
\begin{align} \label{ch4:eq:behavPsiinv0pre}
\Psi_\alpha(z) \Gamma_j = 
M_j(z)
\operatorname{diag}(1,z^{-\alpha},z^{-\alpha-\frac{1}{r}},\ldots,z^{-\alpha-\frac{r-1}{r}})
\end{align}
for some non-singular analytic function $M_j(z)$. This follows from the fact that the indicial equation associated to \eqref{ch4:eq:MeijerGdvgl} has solutions $0, -\alpha, -\alpha-\frac{1}{2},\ldots,-\alpha-\frac{r-1}{r}$ at $z=0$. In fact it follows from the asymptotics of $\Psi_\alpha(z)$ as $z\to \infty$, the fact that its jumps all have determinant $1$, and an application of Liouville's theorem that $M_j$ should have a constant determinant.  For $z$ in the left half-plane, i.e., for $j=2$ and $j=3$, we have by Lemma \ref{ch4:prop:psi0entire} that
\begin{align*}
\Gamma_j = 1 \oplus \frac{1}{r+1} U,
\end{align*}
where $U$ is some $r\times r$ matrix consisting of powers of $\Omega$ and powers of $e^{2\pi i\alpha}$, that we could determine explicitly in principle. We conclude that as $z\to 0$ we have
\begin{align*}
\Psi_\alpha(z)^{-1} &= 
\Gamma_j
\operatorname{diag}(1,z^{\alpha},z^{\alpha+\frac{1}{r}},\ldots,z^{\alpha+\frac{r-1}{r}})
M_j(z)^{-1}\\
&= \left(1 \oplus \frac{1}{r+1} U\right) \mathcal O\begin{pmatrix}
1 & 1 & \hdots & 1\\
z^\alpha & z^\alpha & \hdots & z^\alpha\\
z^{\alpha+\frac{1}{r}} & z^{\alpha+\frac{1}{r}} & \hdots & z^{\alpha+\frac{1}{r}}\\
\vdots & & & \vdots\\
z^{\alpha+\frac{r-1}{r}} & z^{\alpha+\frac{r-1}{r}} & \hdots & z^{\alpha+\frac{r-1}{r}}
\end{pmatrix}\\
&= \mathcal O\begin{pmatrix}
1 & 1 & \hdots & 1\\
z^\alpha & z^\alpha & \hdots & z^\alpha\\
\vdots & & & \vdots\\
z^\alpha & z^\alpha & \hdots & z^\alpha
\end{pmatrix}.
\end{align*}
\end{proof}

\subsubsection{Uniqueness of $\Psi_\alpha$}

In this section we prove that $\Psi_\alpha$ is the only solution to RH-$\Psi$.

\begin{theorem} \label{ch4:thm:PsiaexistUnique}
$\Psi_\alpha$ is the unique solution to RH-$\Psi$.
\end{theorem}

\begin{proof}
It follows from Lemma \ref{ch4:prop:PsiandQ} and Lemma \ref{ch4:prop:behavInftyPsi2} (see \eqref{ch4:RHPsi3rewritten} also) that $\Psi_\alpha$ does indeed solve RH-$\Psi$. To prove uniqueness, suppose that $\Psi(z)$ is a solution to RH-$\Psi$. Then $\Psi(z)\Psi_\alpha(z)^{-1}$ has no jumps and it behaves like $\mathbb I+\mathcal O(1/z)$ as $z\to \infty$. For $z$ in the left half-plane we have by \eqref{ch4:eq:behavPsiinv0} and RH-$\Psi$4 that 
\begin{align*}
\Psi(z)\Psi_\alpha(z)^{-1}
= \left\{\begin{array}{lr}
\mathcal O(z^\alpha), & -1<\alpha<-1+\frac{1}{r},\\
\mathcal O(z^{-1+\frac{1}{r}} \log z), & \alpha = -1+\frac{1}{r},\\
\mathcal O(z^{-1+\frac{1}{r}}), & \alpha>-1+\frac{1}{r}, \alpha\neq 0,\\
\mathcal O(z^{-1+\frac{1}{r}}\log z), & \alpha=0,
\end{array}\right.
\end{align*}
as $z\to 0$. Either way, we conclude that the singularity at $z=0$ is removable. Then Liouville's theorem shows that $\Psi(z)\Psi_\alpha(z)^{-1}=\mathbb I$, and we are done. 
\end{proof}

\subsection{Definition of the local parametrix at the hard edge} \label{ch4:sec:defLocalP}

\subsubsection{Definition of the initial local parametrix $\mathring P$}

In what follows, we will assume that the lips of the lens are slightly deformed around $z=0$, such that,  in $D(0,r_0)$, $f$ maps the lips of the lens into the positive and negative imaginary axis. Notice that we indeed have the freedom to do this. Our initial local parametrix is defined as follows.  
\begin{definition} \label{ch4:def:mathringP}
With $\Psi, f(z)$ and $D_0$ as in Definition \ref{ch4:def:Psi}, Definition \ref{ch4:def:conformalf} and Definition \ref{ch4:def:D0}, we define the initial local parametrix by
\begin{align} \label{ch4:eq:defMathringP}
\mathring P(z) &= \mathring E(z) \Psi\left(n^{r+1} f(z)\right) \operatorname{diag}(1,z^\beta,\ldots,z^\beta) D_0(z)^{-n}, & z\in D(0,r_0),
\end{align}
where we take
\begin{align} \label{ch4:def:mathringE}
\mathring E(z) = n^{-r\beta} \left(\frac{f(z)}{z}\right)^{-\frac{r\beta}{r+1}}.
\end{align}
\end{definition}

\begin{proposition}
$\mathring P$, as in Definition \ref{ch4:def:mathringP}, satisfies RH-$\mathring{\text{P}}$.
\end{proposition}

\begin{proof}
By Proposition \ref{ch4:prop:conformalf} we know that composing $\Psi$ with $f$ does not change its jump matrices or asymptotic behavior as $z\to 0$. Then by Lemma \ref{ch4:prop:PsiandQ} and Proposition \ref{ch4:prop:RHPfromPtoQ} we know that 
\begin{align*}
\Psi_\alpha(n^b f(z)) \operatorname{diag}(1,z^\beta,\ldots,z^\beta) D_0(z)^{-n}
\end{align*}
satisfies RH-$\mathring{\text{P}}$. By Proposition \ref{ch4:prop:conformalf} we also know that $\mathring E(z)$ is analytic on $D(0,r_0)$. Then multiplication on the left with $\mathring E(z)$ does not change any of the conditions in RH-$\mathring{\text{P}}$. 
\end{proof}

\subsubsection{The double matching} \label{ch4:sec:matching}

As indicated in Section \ref{ch4:sec:localParamSetUp} we will apply the double matching procedure from \cite{Mo}. In this section we will show that the conditions of \cite[Theorem 1.2]{Mo} can be met, we repeat this theorem for convenience. 

\begin{theorem}\label{lem:matching}
Let $\mathring P$ and $N$ be defined in a neighborhood of $\overline{D(0,\rho)}$ for some $\rho>0$. These are matrix-valued functions of size $m\times m$ that may vary with $n$. 
Let $a, b, c, d, e \geq 0$ satisfy
\begin{align} \label{eq:assumpabcde}
a\leq e < b\quad\quad\text{ and }\quad\quad  d<\min(b,c).
\end{align}
Suppose that uniformly for $z\in\partial D(0,n^{-a})$ as $n\to\infty$
\begin{align}\label{eq:almostMatching}
\mathring P(z) N(z)^{-1} E(z) &= \mathbb I + \frac{C(z)}{n^b z} + \mathcal O\left(n^{-c}\right),
\end{align}
where $C$ and $E$ are $m\times m$ functions in a neighborhood of $\overline{D(0,\rho)}$ that may vary with $n$, and
\begin{itemize}
\item[(i)] $C$ is meromorphic with only a possible pole at $z=0$, whose order is bounded by some non-negative integer $p$ for all $n$, and $C$ is uniformly bounded for $z\in \partial D(0,n^{-a})$ as $n\to\infty$,
\item[(ii)] $E$ is non-singular, analytic, and uniformly for $z,w\in \partial D(0,n^{-a})$ we have as $n\to\infty$
\begin{align} \label{eq:assumpE}
E(z) = \mathcal O(n^\frac{d}{2}), \quad\quad E(z)^{-1} = \mathcal O(n^{\frac{d}{2}}),\quad\quad \text{ and }\quad E(z)^{-1} E(w) &= \mathbb I+\mathcal O(n^e (z-w)).
\end{align}
\end{itemize}
Then there are non-singular analytic functions  ${E_n^0:\overline{D(0,n^{-a})}\to\mathbb C^{m\times m}}$, ${E_n^\infty:\overline{A(0;n^{-a},\infty)}\to\mathbb C^{m\times m}}$ such that as $n\to\infty$ 
\begin{align*} 
E_n^0(z) \mathring P(z) &= \left(\mathbb I + \mathcal O(n^{d-c})\right) E_n^\infty(z) N(z), &\text{uniformly for }z\in \partial D(0,n^{-a}),\\ 
E_n^\infty(z) &= \mathbb I + \mathcal O(n^{d-b}), &\text{uniformly for }z\in \partial D(0,\rho).
\end{align*}
\end{theorem}

Obviously, our present situation requires that one takes $m=r+1$. The main objective of the current section is to show that the assumptions of Theorem \ref{lem:matching}, and in particular the estimates in (ii), hold for some choice of the constants $a, b, c, d$ and $e$. We shall determine the explicit values of the constants $a, b, c, d$ and $e$ as we go along. As indicated before in \text{Section \ref{ch4:sec:localParamSetUp}}, we want the double matching on the circle with radius $\rho = r_0$ and the circle of radius $r_n$. We remind the reader that
\begin{align*}
r_n &= n^{-\frac{r+1}{2}}, & n=1,2,\ldots
\end{align*}
As mentioned in Section \ref{ch4:sec:localParamSetUp}, we can use the analytic prefactors $E_n^0$ and $E_n^\infty$ from Theorem \ref{lem:matching} and construct the local parametrix $P$ as
\begin{align*}
P(z) = \left\{\begin{array}{ll}
E_n^0(z) \mathring P(z), & z\in D(0,r_n),\\
E_n^\infty(z) N(z), & z\in A(0;r_n,r_0),
\end{array}\right.
\end{align*}
and $P$ will then satisfy a matching condition on both the inner and the outer circle.\\

We know that $|n^{r+1} f(z)|\to \infty$ uniformly for $z\in \partial D(0,r_n)$ as $n\to\infty$. Then, by Lemma \ref{ch4:prop:behavInftyPsi2} and \eqref{ch4:eq:defMathringP}, we have uniformly for $z\in\partial D(0,r_n)$ that
\begin{align} \label{ch4:eq:almostMatching}
\mathring P(z) N(z)^{-1} \sim 
\left(\mathbb I + \frac{C_{\alpha,1}}{n^{r+1} f(z)} +\frac{C_{\alpha,2}}{(n^{r+1}f(z))^2} + \ldots\right) E(z)^{-1},
\end{align}
as $n\to\infty$, where $E$ is defined as follows.
\begin{definition} \label{ch4:def:E}
For $z\in D(0,r_0)$, we define the function
\begin{align} \label{ch4:eq:defE}
E(z) = \mathring E(z)^{-1} N(z) \operatorname{diag}(1,z^{-\beta},\ldots,z^{-\beta}) D_0(z)^n 
\left(\bigoplus_{j=1}^{r+1} e^{n (r+1) \omega^{k_j} f(z)^\frac{1}{r+1}}\right)
L_\alpha\left(n^{r+1}f(z)\right)^{-1}.
\end{align}
\end{definition}
Notice that $E$ depends on $n$. When $r=1$, i.e., when we have a $2\times 2$ RH analysis, the factor
\begin{align*}
D_0(z)^n 
\left(\bigoplus_{j=1}^{r+1} e^{n (r+1) \omega^{k_j} f(z)^\frac{1}{r+1}}\right)
\end{align*}
in \eqref{ch4:eq:defE} equals the unit matrix and, as it turns out, $E$ can then be used as an analytic prefactor (for an ordinary matching, that is). When the size of the RHP is larger, $E$ cannot serve as an analytic prefactor, unfortunately. See Section 1.3 in \cite{Mo} for more on where the difficulty of the matching comes from in larger size RHPs. The function $E$ is a central ingredient for the double matching procedure though. We prove some properties for $E$ in the next three propositions. These statements together with their proofs are straightforward generalizations of their counterparts for $r=1$ in \cite{KuMo} (see Lemma 5.10(c) and \text{Lemma 5.13}). 

\begin{proposition} \label{ch4:prop:Eanalytic}
$E(z)$ is an analytic non-singular function.
\end{proposition}

\begin{proof}
It is clear that all the factors in the right-hand side of \eqref{ch4:eq:defE} are non-singular, hence $E$ is singular as well. As one can easily verify, $L_\alpha\left(n^{r+1}f(z)\right)$ and 
\begin{align*}
L_\alpha(n^{r+1}f(z)) \bigoplus_{j=1}^{r+1} e^{-n (r+1) \omega^{k_j} f(z)^\frac{1}{r+1}}
\end{align*}
have the same jumps, namely those in RH-$\widetilde{\text{P}}$2 according to Proposition \ref{ch4:eq:sameJumpsLPsi}. 
On the other hand, we also know that
\begin{align*}
N(z) \operatorname{diag}(1,z^{-\beta},\ldots,z^{-\beta}) D_0(z)^n
\end{align*}
must have the same jumps on the positive and negative real axis as in RH-$\widetilde{\text{P}}$2. This is due to Proposition \ref{ch4:prop:RHPfromPtoQ} and the fact that $N$ has the same jumps as $\mathring P$ has on the positive and negative real axis. With these two insights, it follows that 
\begin{align*}
N(z) \operatorname{diag}(1,z^{-\beta},\ldots,z^{-\beta}) D_0(z)^n 
\left(\bigoplus_{j=1}^{r+1} e^{n (r+1) \omega^{k_j} f(z)^\frac{1}{r+1}}\right)
L_\alpha\left(n^{r+1}f(z)\right)^{-1}
\end{align*}
does not have any jumps. Then $E$ also does not have any jumps, by Proposition \ref{ch4:prop:Eanalytic} and \eqref{ch4:eq:defE}. This implies that $E$ has a Laurent series around $z=0$. It is clear from Definition \ref{ch4:def:Lalpha} that as $z\to 0$
\begin{align} \label{ch4:eq:behavzLalphaO}
z^{-\frac{r}{r+1}\beta} L_\alpha(z)^{-1} = \mathcal O\left(z^{-\frac{r}{2(r+1)}}\right).
\end{align}
The factors $D_0(z)^n$ and
\begin{align*}
\bigoplus_{j=1}^{r+1} e^{n (r+1) \omega^{k_j} f(z)^\frac{1}{r+1}}
\end{align*}
are bounded due to Proposition \ref{ch4:prop:fmanalytic}, Proposition \ref{ch4:prop:sumfm} and Definition \ref{ch4:def:conformalf} as $z\to 0$. Combining this with \eqref{ch4:eq:behavzLalphaO} and the asymptotic behavior of $N$ in \eqref{ch4:eq:behavNasz0}, we infer that
\begin{align*}
E(z) = \mathcal O\left(z^{-\frac{r}{2(r+1)}} z^{-\frac{r}{2(r+1)}}\right) = \mathcal O\left(z^{-1+\frac{1}{r+1}}\right)
\end{align*}
as $z\to 0$. Hence $E$ has a removable singularity at $z=0$ and the proposition follows. 
\end{proof}

\begin{lemma} \label{ch4:eq:estimateEd}
Uniformly for $z\in \overline{D(0,r_n)}$ we have as $n\to\infty$
\begin{align} \label{ch4:eq:eq:estimateEd}
E(z) = \mathcal O(n^\frac{r}{2}) \quad \text{and} \quad E(z)^{-1} = \mathcal O(n^\frac{r}{2}).
\end{align}
\end{lemma}

\begin{proof}
We define the auxilliary function
\begin{align} \label{ch4:eq:defwidetildeLalpha}
\widetilde L_\alpha(z) = z^{\frac{r}{r+1}\beta} L_\alpha(z). 
\end{align}
Then we have
\begin{align} \label{ch4:eq:identitywidetildeLalpha3}
L_\alpha(n^{r+1} f(z))^{-1} = n^{r\beta} \left(\frac{f(z)}{z}\right)^{\frac{r}{r+1}\beta} z^{\frac{r}{r+1}\beta} \widetilde L_\alpha(n^{r+1} f(z))^{-1}. 
\end{align}
Furthermore, $\widetilde L_\alpha$ satisfies the identities
\begin{align} \label{ch4:eq:identitywidetildeLalpha1}
\widetilde L_\alpha\left(n^\frac{r+1}{2} z\right) &= \left(\bigoplus_{j=0}^r n^{-\frac{r}{4}+\frac{j}{2}}\right)
\widetilde L_\alpha(z),\\ \label{ch4:eq:identitywidetildeLalpha2}
\widetilde L_\alpha\left(n^{r+1} f(z)\right) 
&=  \left(\bigoplus_{j=0}^r n^{-\frac{r}{4}+\frac{j}{2}}\right)
\widetilde L_\alpha\left(n^\frac{r+1}{2} f(z)\right)
\end{align}
which is clear from Definition \ref{ch4:def:Lalpha}. Plugging \eqref{ch4:eq:identitywidetildeLalpha3}, \eqref{ch4:eq:identitywidetildeLalpha1} and \eqref{ch4:eq:identitywidetildeLalpha1} into the definition \eqref{ch4:eq:defE} of $E$, we can express $E$ as
\begin{multline} \label{ch4:eq:rewriteE}
E(z)=M_\alpha(z) 
\left(\bigoplus_{j=0}^r n^{-\frac{r}{4}+\frac{j}{2}}\right)
\widetilde L_\alpha\left(n^\frac{r+1}{2} z\right)
D_0(z)^n \\
\left(\bigoplus_{j=1}^{r+1} e^{n (r+1) \omega^{k_j} f(z)^\frac{1}{r+1}}\right)
\widetilde L_\alpha\left(n^\frac{r+1}{2} f(z)\right)^{-1}
 \left(\bigoplus_{j=0}^r n^{\frac{r}{4}-\frac{j}{2}}\right),
\end{multline}
where
\begin{align} \label{ch4:eq:defanaMalpha}
M_\alpha(z) = N(z) \operatorname{diag}\left(z^\frac{r\beta}{r+1},z^{-\frac{\beta}{r+1}}, \ldots, z^{-\frac{\beta}{r+1}}\right) \widetilde L_\alpha(z)^{-1}.
\end{align}
Notice that $M_\alpha$ is an analytic function that does not depend on $n$. Indeed, this is because $N(z)$ and
\begin{align*}
\widetilde L_\alpha(z) \operatorname{diag}\left(z^{-\frac{r\beta}{r+1}},z^{\frac{\beta}{r+1}}, \ldots,z^{\frac{\beta}{r+1}}\right)
= L_\alpha(z) \operatorname{diag}\left(1,z^\beta,\ldots,z^\beta\right)
\end{align*}
have the same jumps, as one may verify. Then it is trivial that $M_\alpha$ is $\mathcal O(1)$ uniformly for $z\in \partial D(0,r_n)$ as $n\to\infty$.  Using \eqref{ch4:eq:defkj} it follows after some straightforward algebra that
\begin{align*}
\omega^{k_{l+1}} = -\omega^{(-1)^{l-1} (\frac{1}{2} + \lfloor \frac{l}{2}\rfloor)}
\end{align*}
for all $l=0,1,\ldots, r$. Then it follows from Proposition \ref{ch4:prop:fmanalytic}, Proposition \ref{ch4:prop:sumfm} and Definition \ref{ch4:def:D0} that
\begin{align*}
D_0(z)^n 
\left(\bigoplus_{j=1}^{r+1} e^{n (r+1) \omega^{k_j} f(z)^\frac{1}{r+1}}\right) 
&= \bigoplus_{j=0}^r \exp \left(n \sum_{m=2}^r \omega^{\pm (-1)^{l-1} (\frac{1}{2}+\lfloor\frac{l}{2}\rfloor) m} z^\frac{m}{r+1} f_m(z)\right)\\
&= \exp{\mathcal O\left(n z^\frac{2}{r+1}\right)} = \mathcal O(1)
\end{align*}
uniformly for $z\in\partial D(0,r_n)$ as $z\to 0$. We conclude that the factor
\begin{align} \label{ch4:eq:defmathcalLalpha}
\mathcal L_\alpha(z) = \widetilde L_\alpha\left(n^\frac{r+1}{2} z\right)
D_0(z)^n
\left(\bigoplus_{j=1}^{r+1} e^{n (r+1) \omega^{k_j} f(z)^\frac{1}{r+1}}\right)
\widetilde L_\alpha\left(n^\frac{r+1}{2} f(z)\right)^{-1}
\end{align}
in \eqref{ch4:eq:rewriteE} is uniformly bounded for $z\in\partial D(0,r_n)$, which follows from Definition \ref{ch4:def:Lalpha} and the fact that both $n^\frac{r+1}{2} z$ and $n^\frac{r+1}{2} f(z)$ are of order $1$ on $\partial D(0,r_n)$. Then in view of \eqref{ch4:eq:rewriteE} and the boundedness of $M_\alpha(z)$ we have that
\begin{align*}
E(z) = \mathcal O\left(1\cdot n^\frac{r}{4} \cdot 1 \cdot n^\frac{r}{4}\right) = \mathcal O\left(n^\frac{r}{2}\right)
\end{align*}
uniformly for $z\in\partial D(0,r_n)$ as $z\to\infty$. By the maximum modulus principle the estimate also holds on $D(0,r_n)$. Hence we have the first estimate in \eqref{ch4:eq:eq:estimateEd}. The estimate for the inverse of $E$ follows in similar fashion. 
\end{proof}

\begin{lemma} \label{ch4:eq:estimateEe}
Uniformly for $z,w\in \overline{D(0,r_n)}$ we have as $n\to\infty$
\begin{align} \label{ch4:eq:eq:estimateEe}
E(z)^{-1} E(w) = \mathbb I + \mathcal O\left(n^{r+\frac{1}{2}}(z-w)\right).
\end{align}
\end{lemma}

\begin{proof}
Using \eqref{ch4:eq:rewriteE} we find that
\begin{multline} \label{ch4:eq:EzEwnn}
E(z)^{-1} E(w)
= \left(\bigoplus_{j=0}^r n^{-\frac{r}{4}+\frac{j}{2}}\right) \mathcal L_\alpha(z)^{-1}
\left(\bigoplus_{j=0}^r n^{\frac{r}{4}-\frac{j}{2}}\right)\\
M_\alpha(z)^{-1} M_\alpha(w) 
\left(\bigoplus_{j=0}^r n^{-\frac{r}{4}+\frac{j}{2}}\right)
\mathcal L_\alpha(w)
\left(\bigoplus_{j=0}^r n^{\frac{r}{4}-\frac{j}{2}}\right),
\end{multline}
with $\mathcal L_\alpha$ as in \eqref{ch4:eq:defmathcalLalpha} and $M_\alpha$ as in \eqref{ch4:eq:defanaMalpha} above. Due to the analyticity of $M_\alpha$ we have the estimate
\begin{align} \label{ch4:eq:nMMnIO}
\left(\bigoplus_{j=0}^r n^{\frac{r}{4}-\frac{j}{2}}\right)
M_\alpha(z)^{-1} M_\alpha(w) 
\left(\bigoplus_{j=0}^r n^{-\frac{r}{4}+\frac{j}{2}}\right)
=\mathbb I + \mathcal O\left(n^\frac{r}{2}(z-w)\right),
\end{align}
uniformly for $z,w \in\partial D(0,r_n)$ as $n\to\infty$. As we proved before, $\mathcal L_\alpha$ is uniformly bounded on $\partial D(0,r_n)$. 
Then we have that
\begin{align} \nonumber
\mathcal L_\alpha(z)^{-1} &\left(\mathbb I + \mathcal O(n^\frac{r}{2}(z-w))\right) \mathcal L_\alpha(w)\\
&= \mathcal L_\alpha(z)^{-1} \mathcal L_\alpha(w) + \mathcal L_\alpha(z)^{-1} \mathcal O(n^\frac{r}{2}(z-w)) \mathcal L_\alpha(w)\\ \nonumber
&= \mathbb I + \mathcal O\left(n^\frac{r+1}{2}(z-w)\right) + \mathcal O\left(n^\frac{r}{2}(z-w)\right)\\ \label{ch4:eq:LIOL}
&= \mathbb I + \mathcal O\left(n^\frac{r+1}{2}(z-w)\right)
\end{align}
uniformly for $z,w\in\partial D(0,r_n)$ as $n\to\infty$. Here we have used a standard argument using Cauchy's integral formula to estimate $\mathcal L_\alpha(z)^{-1} \mathcal L_\alpha(w)$. 
Plugging \eqref{ch4:eq:nMMnIO} and \eqref{ch4:eq:LIOL} into \eqref{ch4:eq:EzEwnn}, we conclude that
\begin{align*}
E(z)^{-1} E(w) &= \left(\bigoplus_{j=0}^r n^{-\frac{r}{4}+\frac{j}{2}}\right)
\left(\mathbb I + \mathcal O\left(n^\frac{r+1}{2}(z-w)\right)\right)
\left(\bigoplus_{j=0}^r n^{\frac{r}{4}-\frac{j}{2}}\right)\\ \nonumber
&= \mathbb I + \mathcal O\left(n^\frac{r}{4} n^\frac{r+1}{2}(z-w) n^\frac{r}{4}\right)\\ \nonumber
&= \mathbb I + \mathcal O(n^{r+\frac{1}{2}}(z-w))
\end{align*}
uniformly for $z,w\in\partial D\left(0,r_n\right)$ as $n\to\infty$, and we have arrived at \eqref{ch4:eq:eq:estimateEe}. By a double application of the maximum principle, applied to the analytic function
$$(z,w)\mapsto \frac{E(z)^{-1} E(w) - \mathbb I}{z-w},$$
the estimate holds for $z,w\in\overline{D(0,r_n)}$. 
\end{proof}

In line with Theorem \ref{lem:matching} we define the constants
\begin{align} \label{ch4:eq:abcde}
a= \frac{r+1}{2}, \quad b = r+1, \quad d =  r, \quad \text{ and }\quad e = r + \frac{1}{2},
\end{align}
and the meromorphic function
\begin{align} \label{ch4:eq:defmeromorphicC}
C(z) = \frac{z}{f(z)} \left(C_{\alpha,1} + \frac{C_{\alpha,2}}{(n^b f(z))} + \frac{C_{\alpha,3}}{(n^b f(z))^2}\right).
\end{align}
When $r=1$ or $r=2$ we are allowed to take fewer terms in the expansion in \eqref{ch4:eq:defmeromorphicC}. We will nevertheless fix the definition of $C$ with three terms as in \eqref{ch4:eq:defmeromorphicC}. Then by \eqref{ch4:eq:almostMatching} we have uniformly for $z\in\partial D(0,r_n)$ that
\begin{align*}
\mathring P(z) N(z)^{-1} = \left(\mathbb I + \frac{C(z)}{n^{r+1} z} + \mathcal O\left(n^{-c}\right)\right) E(z)^{-1}
\end{align*}
as $n\to\infty$, where
\begin{align} \label{ch4:eq:defabc}
c=2 (r+1).
\end{align}
Now all the requirements for Theorem \ref{lem:matching} are met. Hence we obtain analytic prefactors ${E_n^0 : D(0,r_n)\to \mathbb C}$ and $E_n^\infty : A(0;r_n,r_0)\to \mathbb C$
such that
\begin{align} \label{ch4:eq:doubleMatching0}
E_n^0(z) \mathring P(z)
= \left(\mathbb I + \mathcal O\left(\frac{1}{n^{r+2}}\right)\right) E_n^\infty(z) N(z)
\end{align}
uniformly for $z\in \partial D(0,r_n)$ as $n\to\infty$, and
\begin{align} \label{ch4:eq:doubleMatchingInfty}
E_n^\infty(z) = \mathbb I + \mathcal O\left(\frac{1}{n}\right)
\end{align}
uniformly for $z\in \partial D(0,r_0)$ as $n\to\infty$. 

Analytic prefactors with the properties as in Theorem \ref{lem:matching} are not unique, but we fix them to be defined as in Section 2.1 in \cite{Mo}. The reason for this particular choice, is that it will allow us to apply Theorem 3.1 of \cite{Mo} later on when we calculate the scaling limit of the correlation kernel in Section \ref{sec:proofOfMainThm}. We omit the explicit formulae for $E_n^0(z)$ and $E_n^\infty(z)$ though, since such formulae are not insightful in my opinion, and they will not be relevant to us. 

\subsubsection{Definition of the local parametrix $P$ at the hard edge}

We are now ready to fix the definition of the local parametrix $P$ at the origin. 

\begin{definition} \label{ch4:def:P}
We define the local parametrix at the hard edge $z=0$ by
\begin{align} \label{ch4:eq:defP}
P(z) = \left\{\begin{array}{ll}
E_n^0(z) \mathring P(z), & z\in D(0,r_n),\\
E_n^\infty(z) N(z), & z\in A(0;r_n,r_0).
\end{array}\right.
\end{align}
\end{definition}

Considering our discussion in Section \ref{ch4:sec:matching}, and \eqref{ch4:eq:doubleMatching0} and \eqref{ch4:eq:doubleMatchingInfty} in particular, we have the following corollary. 

\begin{corollary} \label{ch4:cor:doubleMatching}
$P$, as defined in Definition \ref{ch4:def:P}, satisfies a double matching of the form:
\begin{align*}
P(z) N(z)^{-1} &= \mathbb I + \mathcal O\left(\frac{1}{n}\right), & \text{uniformly for }z\in \partial D(0,r_0),\\
P_+(z) P_-(z)^{-1} &= \mathbb I +\mathcal O\left(\frac{1}{n^{r+2}}\right), &\text{uniformly for }z\in \partial D(0,r_n),
\end{align*}
as $n\to\infty$. 
\end{corollary}

\subsection{The local parametrix $Q$ at the soft edge $z=q$}

The local parametrix problem around $q$ is defined on a disk around $q$. Without loss of generality we may assume that its radius is $r_0$, as before. 

\begin{rhproblem} \label{ch4:RHPfortildeP} \
\begin{description}
\item[RH-Q1] $Q$ is analytic on $D(q,r_0) \setminus \Sigma_S$.
\item[RH-Q2] $Q$ has the same jumps as $S$ has on $D(q,r_0) \setminus \Sigma_S$.
\item[RH-Q3] $Q$ is bounded around $q$. 
\item[RH-Q4] $Q$ satisfies the matching condition:
\begin{align} \label{ch4:eq:matchingQ}
Q(z) N(z)^{-1} = \mathbb I + \mathcal O\left(\frac{1}{n}\right)
\end{align}
uniformly for $z\in\partial D(q,r_0)$ as $n\to \infty$. 
\end{description}
\end{rhproblem}

The construction of the local parametrix, with Airy functions, is standard and we omit the details. See \cite{Ku2} for an example were this is done for size $3\times 3$, it easily generalizes to size $(r+1)\times (r+1)$. 

\section{Final transformation and proof of the main theorem}

\subsection{The final transformation $S\mapsto R$}

With the global and local parametrices as before, we define the final transformation as 
\begin{align} \label{ch4:eq:defR}
R(z) = \left\{\begin{array}{ll} S(z) P(z)^{-1} & z\in D(0,r_0),\\
S(z) Q(z)^{-1} & z\in D(q,r_0),\\
S(z) N(z)^{-1} & \text{elsewhere}.
\end{array}\right.
\end{align}
We remark that $R$ has a jump on $\partial D(0,r_n)$ and on the lips of the lens inside $A(0;r_n,r_0)$. This is a difference with the usual case, where an ordinary matching is used. Notice that there are no jumps inside $D(0,r_n)$ because the jumps of $S$ and $\mathring P$, and thus $P$, are the same there. Similarly, there are no jumps on $(-r_0,r_n)$ and $(r_n,r_0)$ because the jumps of $S$ and $N$, and thus $P$, are the same there. See Figure \ref{ch4:FigR} for the corresponding jump contour $\Sigma_R$. 

\begin{figure}[t]
\begin{center}
\resizebox{10.5cm}{6cm}{
\begin{tikzpicture}[>=latex]
	\draw[-] (-7,0)--(7,0);
	\draw[-] (-4,-4)--(-4,4);
	\draw[fill] (-4,0) circle (0.1cm);
	\draw[fill] (3,0) circle (0.1cm);
	\draw[ultra thick] (-4,0) circle (1.5cm);
	\draw[ultra thick] (-4,0) circle (0.5cm);
	\draw[ultra thick] (3,0) circle (1.5cm);
	\node[above] at (-4.2,-0.075) {\large 0};	
	\node[above] at (2.8,0) {\large $q$};	
	
	\draw[->, ultra thick] (-4,0.5) -- (-4,1.2);
	\draw[-, ultra thick] (-4,0.5) -- (-4,3);
	
	\draw[->, ultra thick] (-4,-0.5) -- (-4,-1.2);
	\draw[-, ultra thick] (-4,-0.5) -- (-4,-3);
	
	\draw[->, ultra thick] (4.5,0) -- (5.5,0);
	\draw[-, ultra thick] (4.5,0) -- (7,0);
	
	\draw[-, ultra thick] (-4,3) to [out=0, in=135] (1.939,1.061);
	\draw[-, ultra thick] (-4,-3) to [out=0, in=225] (1.939,-1.061);

	\draw[->, ultra thick] (-2.86,0.96) to (-2.96,1.08);
	\draw[->, ultra thick] (-3.65,0.35) to (-3.75,0.47);	
	\draw[->, ultra thick] (4.14,0.96) to (4.04,1.08);
	
	\draw[->, ultra thick] (-1.5,2.85) to (-1.3,2.82);
	\draw[->, ultra thick] (-1.5,-2.85) to (-1.3,-2.82);
\end{tikzpicture}
}
\caption{Contour $\Sigma_{R}$ for the RHP for $R$. \label{ch4:FigR}}
\end{center}
\end{figure}
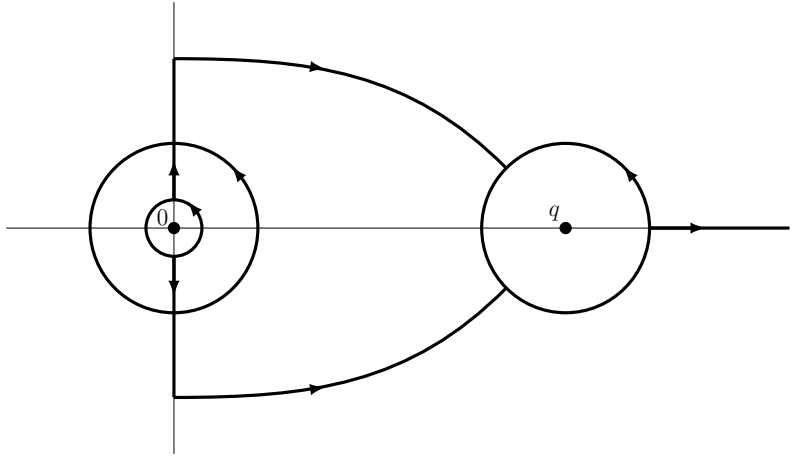

\begin{lemma}
The singularity of $R$ at $z=0$ is removable. 
\end{lemma}

\begin{proof}
In the left half-plane we have, using \eqref{ch4:eq:defP} and \eqref{ch4:eq:defMathringP}, that
\begin{align} \label{ch4:eq:Rformula}
R(z) = S(z) \operatorname{diag}(1,z^{-\beta},\ldots,z^{-\beta}) D_0(z)^{-n} \Psi_\alpha(z)^{-1} E_n^0(z)^{-1}.
\end{align}
Also, using Proposition \ref{ch4:prop:sumfm} and the asymptotics in RH-S4, we have as $z\to 0$
\begin{align*}
S(z) \operatorname{diag}(1,z^{-\beta},\ldots,z^{-\beta}) D_0(z)^{-n}
= \mathcal O\begin{pmatrix} 1 & h_{-\alpha-\frac{r-1}{r}}(z) & h_{-\alpha-\frac{r-1}{r}}(z)\\
\vdots & & \vdots\\
1 & h_{-\alpha-\frac{r-1}{r}}(z) & h_{-\alpha-\frac{r-1}{r}}(z)\end{pmatrix}.
\end{align*}
We remind the reader that $h_\alpha$ is defined in RH-Y4. Combining the above with \eqref{ch4:eq:Rformula} and \eqref{ch4:eq:behavPsiinv0} we infer, for $\alpha\neq 0$, that as $z\to 0$
\begin{align*}
R(z) &= \mathcal O\begin{pmatrix} 1 & h_{-\alpha-\frac{r-1}{r}}(z) & h_{-\alpha-\frac{r-1}{r}}(z)\\
\vdots & & \vdots\\
1 & h_{-\alpha-\frac{r-1}{r}}(z) & h_{-\alpha-\frac{r-1}{r}}(z)\end{pmatrix}
\mathcal O\begin{pmatrix} 
h_\alpha(z) & \hdots & h_\alpha(z)\\
z^\alpha & \hdots & z^\alpha\\
\vdots & & \vdots\\
z^\alpha & \hdots & z^\alpha
\end{pmatrix}\\
&= \mathcal O(h_\alpha(z)+z^\alpha h_{-\alpha-\frac{r-1}{r}}(z))\\
&= \mathcal O(h_\alpha(z) + z^{-\frac{r-1}{r}} h_{\alpha+\frac{r-1}{r}}(z)).
\end{align*}
A slightly different behavior holds for $\alpha=0$. By considering the cases separately, one finds that as $z\to 0$
\begin{align*}
R(z) 
= \left\{\begin{array}{lr}
\mathcal O(z^\alpha), & -1<\alpha<-1+\frac{1}{r},\\
\mathcal O(z^{-1+\frac{1}{r}} \log z), & \alpha = -1+\frac{1}{r},\\
\mathcal O(z^{-1+\frac{1}{r}}), & \alpha>-1+\frac{1}{r}, \alpha\neq 0,\\
\mathcal O(z^{-1+\frac{1}{r}}\log z), & \alpha=0.
\end{array}\right.
\end{align*}
Either way, we must conclude that the singularity at $z=0$ is removable. 
\end{proof}

The same is true for the singularity at $z=q$, although we omit the details. We conclude that $R$ is analytic on $\mathbb C\setminus \Sigma_R$. 

\begin{theorem} \label{ch4:thm:RtoI}
(a) As $n\to\infty$ we have uniformly on the indicated contours that
\begin{align} \label{ch4:eq:estimatesJumpsR1}
R_+(z) &= R_-(z) \left(\mathbb I+\mathcal O\left(\frac{1}{n}\right)\right), \hspace{0.1cm}z\in \partial D(0,r_0) \cup \partial D(0,r_n) \cup \partial D(q,r_0),\\ \label{ch4:eq:estimatesJumpsR2}
R_+(z) &= R_-(z) \left(\mathbb I + \mathcal O(e^{-c_1 \sqrt n})\right), \hspace{2.2cm} z\in \Delta_0^\pm \cap A(0;r_n,r_0),\\ \label{ch4:eq:estimatesJumpsR3}
R_+(z) &= R_-(z) \left(\mathbb I + \mathcal O(e^{-c_2 n})\right), \hspace{2cm} z\text{ in the remaining parts.}
\end{align}
where $c_1, c_2$ are positive constants (see Figure \ref{ch4:FigR} also).\\ 
(b) We have as $n\to\infty$ that
\begin{align} \label{ch4:eq:asympRn}
R(z) = \mathbb I + \mathcal O\left(\frac{1}{n}\right)
\end{align}
uniformly for $z\in\mathbb C\setminus \Sigma_R$. 
\end{theorem}

\begin{proof}
The estimates \eqref{ch4:eq:estimatesJumpsR1} for the jumps on the circles follows from the double matching around $z=0$ and the matching around $z=q$, see Corollary \ref{ch4:cor:doubleMatching} and \eqref{ch4:eq:matchingQ}. The jump on the lens inside $A(0;r_n,r_0)$ is according to Definition \ref{ch4:def:P} and RH-S2 given by
\begin{align*} \nonumber
R_-(z)^{-1} R_+(z) &= E_n^\infty(z) N(z) S_-(z)^{-1} S_+(z) N(z)^{-1} E_n^\infty(z)^{-1}\\
&= \mathbb I + E_n^\infty(z) N(z) z^{-\beta} e^{2n\varphi(z)} E_{21} N(z)^{-1} E_n^\infty(z)^{-1},
\end{align*}
where $E_{21}$ is the matrix with a $1$ on the component in the first column and second row and all other components equal to $0$. The factors $E_n^\infty(z), N(z)$ and their inverses depend polynomially on $n$. To get the correct behavior \eqref{ch4:eq:estimatesJumpsR2}, it then suffices to show that there exists a $c>0$ such that
\begin{align} \label{ch4:eq:nRephic}
n \operatorname{Re}(\varphi(z)) \leq -c \sqrt n
\end{align}
uniformly for $z\in \Delta_0^\pm\cap A(0;r_n,r_0)$ as $n\to \infty$. To prove this, we remember from \eqref{ch4:eq:behavfr+1phi0} that for $\pm\operatorname{Im}(z)>0$
\begin{align*}
\pm 2i \sum_{m=1}^r \sin\left(\frac{\pi m}{r+1}\right) z^\frac{m}{r+1} f_m(z) = (r+1) \varphi_0(z).
\end{align*}
Actually, \eqref{ch4:eq:behavfr+1phi0} was only stated for $z$ in the upper half-plane, but it is easy to see how to extend it. Then, using $\varphi_0(z) = \varphi(z)\pm \pi i$, we get uniformly for $z\in \Delta_0^\pm\cap A(0;r_n,r_0)$ that
\begin{align*}
(r+1) \operatorname{Re}(\varphi(z)) &= \pm 2\sin\left(\frac{1}{r+1}\right) \operatorname{Re}\left(i z^\frac{1}{r+1}\right) + \mathcal O\left(|z|^\frac{2}{r+1}\right)\\
&= - 2 \sin\left(\frac{1}{r+1}\right) \sin\left(\frac{1}{2(r+1)}\right) |z|^\frac{1}{r+1} + \mathcal O\left(|z|^\frac{2}{r+1}\right)\\
&\leq  - \sin\left(\frac{1}{r+1}\right) \sin\left(\frac{1}{2(r+1)}\right) \frac{1}{\sqrt n}
\end{align*}
as $n\to\infty$. Hence we get \eqref{ch4:eq:nRephic} for a particular choice of $c$ and we conclude that we get the estimate \eqref{ch4:eq:estimatesJumpsR2} for the jump on $\Delta_0^\pm\cap A(0;r_n,r_0)$ for some $c_1>0$. 

The estimate \eqref{ch4:eq:estimatesJumpsR3} for the jump on $(q,\infty)$ follows from the variational equations and Assumption \ref{ch4:assump:varStrict}, and for the estimate on the remaining parts of the lips of the lens one uses a standard argument with the Cauchy-Riemann equations.\\

\noindent (b) This follows from (a) with standard arguments from Riemann-Hilbert theory. One may use arguments similar to those from Appendix A in \cite{BlKu2}. 
\end{proof}

Theorem \ref{ch4:thm:RtoI}(b) is usually sufficient to obtain the scaling limit of the correlation kernel. In our case it will not be enough though, as we shall see in the next section. We present a stronger result in \text{Section \ref{sec:proofOfMainThm}}, that will allow us to calculate the scaling limit. This result comes from Theorem 3.1 of \cite{Mo}. 

\subsection{Rewriting of the correlation kernel}

In this section we invert all the transformations of our RH analysis, with the goal of finding a relation between the correlation kernel and the bare parametrix $\Psi$ in particular. 

\begin{lemma} \label{lem:corKerRewrite}
For $x,y\in (0,r_n)$ the correlation kernel can be written as
\begin{multline} \label{eq:justbeforexnyn}
K_{V,n}^{\alpha,\frac{1}{r}}(x,y) = \frac{e^{\frac{r n}{r+1} (V(x)-V(y))}}{2\pi i(x-y)} \\
\begin{pmatrix} - 1 & 1 & 0 & \cdots & 0\end{pmatrix} 
\Psi_{\alpha,+}\left(n^{r+1} f(y)\right)^{-1} E_n^0(y)^{-1} R(y)^{-1}
R(x) E_n^0(x) \Psi_{\alpha,+}\left(n^{r+1} f(x)\right)
\begin{pmatrix} 1 \\ 1 \\ 0 \\ \vdots \\ 0\end{pmatrix}.
\end{multline}
\end{lemma}

\begin{proof}
A simple calculation, where we invert the transformations $Y\mapsto X\mapsto T \mapsto S$, shows that for $z$ in $D(0,r_n)$ in the first quadrant
\begin{align*}
Y(z) \begin{pmatrix} 1 \\ 0 \\ \vdots \\ 0\end{pmatrix}
= e^{-n\frac{r\ell}{r+1}} r^{-\frac{r}{2r+2}} e^{n g_{0}(z)} (1\oplus C_n) L S(z) \begin{pmatrix} 1 \\ z^{-\beta} e^{2 n\varphi(z)} \\ 0 \\ \vdots \\ 0\end{pmatrix} 
\end{align*}
and
\begin{multline*}
\begin{pmatrix} 0 & w_\alpha(z) & w_{\alpha+\frac{1}{r}}(z) & \cdots \end{pmatrix} Y(z)^{-1} = \\
r^\frac{r}{2r+2} e^{-n\frac{\ell}{r+1}} w_\alpha(z) z^\frac{r-1}{2 r} e^{n (g_0(z)-g_1(z))}
\begin{pmatrix} -z^{-\beta} e^{2n\varphi(z)} & 1 & 0 & \cdots & 0\end{pmatrix} S(z)^{-1} L^{-1} (1\oplus C_n^{-1}). 
\end{multline*}
Then with the help of \eqref{ch4:eq:KinY} we can write the correlation kernel for $x,y\in (0,r_n)$ as
\begin{multline*}
K_{V,n}^{\alpha,\frac{1}{r}}(x,y) = \frac{1}{2\pi i(x-y)} 
|y|^\frac{r-1}{2r} w_\alpha(y) e^{-n \ell} e^{n (g_{0+}(x)+g_{0+}(y) - g_{1+}(y))} \\
\begin{pmatrix} -|y|^{-\beta} e^{2n\varphi_+(y)} & 1 & 0 & \cdots & 0\end{pmatrix} P_+(y)^{-1} R(y)^{-1}
R(x) P_+(x) \begin{pmatrix} 1 \\ |x|^{-\beta} e^{2 n\varphi_+(x)} \\ 0 \\ \vdots \\ 0\end{pmatrix}
\end{multline*}
(This formula is also valid when $n$ is not divisible by $r$, see Proposition \ref{prop:corKerWidetildeY} in Appendix \ref{ch:appendixA}). Using \eqref{ch4:eq:defP} and \eqref{ch4:eq:defMathringP}, we can express this as
\begin{multline*}
K_{V,n}^{\alpha,\frac{1}{r}}(x,y) = \frac{1}{2\pi i(x-y)} 
|y|^{-\alpha} w_\alpha(y) e^{-n \ell} e^{n (g_{0+}(x)+g_{0+}(y) - g_{1+}(y))} \\
\begin{pmatrix} - D_{0+,00}(y)^n e^{2n\varphi_+(y)} & D_{0+,11}(y)^n & 0 & \cdots & 0\end{pmatrix} \Psi_{\alpha,+}\left(n^{r+1} f(y)\right)^{-1} E_n^0(y)^{-1} R(y)^{-1}
R(x) E_n^0(x)\\
\Psi_{\alpha,+}\left(n^{r+1} f(x)\right) \begin{pmatrix} D_{0+,00}(x)^{-n} \\ D_{0+,11}(x)^{-n} e^{2 n\varphi_+(x)} \\ 0 \\ \vdots \\ 0\end{pmatrix}.
\end{multline*}
Now we use that $D_{0+,11}(z)=D_{0+,00}(z) e^{2 \varphi_{0,+}(z)}=D_{0+,00}(z) e^{2 \varphi_+(z)}$, which follows from Definition \ref{ch4:def:D0} and the relation $\varphi_0(z) = \varphi(z) \pm \pi i$. Then we obtain
\begin{multline*}
K_{V,n}^{\alpha,\frac{1}{r}}(x,y) = \frac{1}{2\pi i(x-y)} \frac{D_{0+,00}(y)^n}{D_{0+,00}(x)^n}
w_0(y) e^{-n \ell} e^{n (g_{0+}(x)+g_{0+}(y) - g_{1+}(y)+2\varphi_+(y))} \\ 
\begin{pmatrix} - 1 & 1 & 0 & \cdots & 0\end{pmatrix} 
\Psi_{\alpha,+}\left(n^{r+1} f(y)\right)^{-1} E_n^0(y)^{-1} R(y)^{-1}
R(x) E_n^0(x) \Psi_{\alpha,+}\left(n^{r+1} f(x)\right)
\begin{pmatrix} 1 \\ 1 \\ 0 \\ \vdots \\ 0\end{pmatrix}.
\end{multline*}
By \eqref{ch4:eq:defgfunctions} and \eqref{ch4:eq:defvarphi0} we have that
\begin{align*}
g_{0+}(y) - g_{1+}(y) + 2\varphi_{0,+}(y) = -g_{0+}(y) + V(y) + \ell.
\end{align*}
Hence we have
\begin{multline*}
K_{V,n}^{\alpha,\frac{1}{r}}(x,y) = \frac{1}{2\pi i(x-y)} \frac{D_{0+,00}(y)^n}{D_{0+,00}(x)^n}
e^{n (g_{0+}(x)-g_{0+}(y))} \\ 
\begin{pmatrix} - 1 & 1 & 0 & \cdots & 0\end{pmatrix} 
\Psi_{\alpha,+}\left(n^{r+1} f(y)\right)^{-1} E_n^0(y)^{-1} R(y)^{-1}
R(x) E_n^0(x) \Psi_{\alpha,+}\left(n^{r+1} f(x)\right)
\begin{pmatrix} 1 \\ 1 \\ 0 \\ \vdots \\ 0\end{pmatrix}.
\end{multline*}
Now the Lemma follows, if we can show that
\begin{align} \label{eq:D0+gVl}
D_{0+,00}(x) = g_{0,+}(x) - \frac{r}{r+1} (V(x)+\ell). 
\end{align}
Indeed, using \eqref{ch4:eq:defvarphi0}, \eqref{ch4:eq:defvarphij} and \eqref{ch4:eq:defvarphir-1} we find that
\begin{align*}
-2 \sum_{j=0}^{r-1} (r-j) \varphi_j(z) =&
-2 r\varphi_0(z) + \sum_{j=1}^{r-1} (r-j) (-g_{j-1}(z) + 2 g_j(z) - g_{j+1}(z))\\
=& -2 r\varphi_0(z) - (r-1) g_0(z) +r g_1(z)\\
&+ \sum_{j=1}^{r-1} \left(-(r-j-1) + 2(r-j) - (r-j+1)\right) g_j(z)\\
=& - 2r \varphi_0(z) - (r-1) g_0(z) + r g_1(z)\\
=& (r+1) g_0(z) - r (V(z) + \ell)
\end{align*}
for any $z\in O_V$. Now using Definition \ref{ch4:def:D0}, we get \eqref{eq:D0+gVl}, and we are done. 
\end{proof}

\subsection{Proof of the main theorem} \label{sec:proofOfMainThm}

To obtain the scaling limit of the correlation kernel at the hard edge $z=0$ it will be convenient to introduce, for any $x,y>0$, the notation
\begin{align} \label{eq:defxnyn}
x_n = \frac{x}{f'(0) n^{r+1}}\quad\quad\text{and}\quad\quad y_n = \frac{y}{f'(0) n^{r+1}}.
\end{align}
Remember (see Proposition \ref{ch4:prop:conformalf}) that
\begin{align*}
f'(0) = \left(\frac{\pi c_{0,V}}{\sin\left(\frac{\pi}{r+1}\right)}\right)^{r+1}.
\end{align*}
Hence \eqref{eq:defxnyn} can also be written as
\begin{align} \label{eq:defxnyn2}
x_n = \frac{x}{(c n)^{r+1}}\quad\quad\text{and}\quad\quad y_n = \frac{y}{(c n)^{r+1}},
\hspace{2cm} c = \frac{\pi c_{0,V}}{\sin\left(\frac{\pi}{r+1}\right)}.
\end{align}
Notice that $x_n, y_n$ are in $(0,r_n)$ for $n$ big enough. In view of \eqref{eq:justbeforexnyn} we would like
\begin{align*}
E_n^0(y_n)^{-1} R(y_n)^{-1} R(x_n) E_n^0(x_n)
\end{align*}
to be close to the identity matrix. In \cite{KuMo} we used a method to show this when $r=2$. Following this method, we find using standard arguments with Cauchy's formula (see \cite[Lemma 6.5]{KuMo}) that
\begin{align*}
R(y_n)^{-1} R(x_n) = \mathbb I + \mathcal O\left(n^{-\frac{r+3}{2}}(x-y)\right)
\end{align*}
uniformly for $x,y$ in compact sets as $n\to\infty$. Then we find, using Lemma \ref{ch4:eq:estimateEd} and Lemma \ref{ch4:eq:estimateEe},  that uniformly for $x,y$ in compact sets
\begin{align*}
E_n^0(y_n)^{-1} R(y_n)^{-1} R(x_n) E_n^0(x_n)
&= \mathbb I + \mathcal O\left(n^{r+\frac{1}{2}}(x_n-y_n)\right) + \mathcal O\left(n^\frac{r}{2} n^{-\frac{r+3}{2}} (x-y) n^\frac{r}{2}\right)\\
&= \mathbb I + \mathcal O\left(\frac{x-y}{\sqrt n}\right) + \mathcal O\left(n^\frac{r-3}{2}(x-y)\right)
\end{align*}
as $n\to\infty$. For $r=1$ and $r=2$, this is enough, but for $r\geq 3$ we run into a problem. We conclude that we can not use the approach from \cite{KuMo}, unfortunately. Instead we will use \cite[Theorem 3.1]{Mo}, which is specifically designed for scaling limits of correlation kernels.

\begin{lemma}
Uniformly for $x,y>0$ in compact sets we have as $n\to\infty$ that
\begin{align} \label{eq:estimateERRE}
E_n^0(y_n)^{-1} R(y_n)^{-1} R(x_n) E_n^0(x_n) = \mathbb I + \mathcal O\left(\frac{x-y}{\sqrt n}\right).
\end{align}
\end{lemma} 

\begin{proof}
We argue that the conditions of the theorem are met. With $a, b, c, d, e$ as in \eqref{ch4:eq:abcde} and \eqref{ch4:eq:defabc}, we should have the inequality
\begin{align*}
2(r+1) = c \geq \min\left(\frac{3}{2}a+d,\frac{3}{2}a+2d-e\right) = \frac{7}{4}r+\frac{1}{4},
\end{align*}
which indeed holds. It is clear that $C$ is uniformly bounded on $\partial D(0,n^{-e})$. 
Another condition for \cite[Theorem 3.1]{Mo} to hold, is that the jumps of $R$ should satisfy specific estimates, and that $R\to \mathbb I$ uniformly as $n\to\infty$. Indeed, these conditions are provided by Theorem \ref{ch4:thm:RtoI}, the estimates on the jumps by \eqref{ch4:eq:estimatesJumpsR1}-\eqref{ch4:eq:estimatesJumpsR3} and the large $n$ behavior of $R$ by \eqref{ch4:eq:asympRn}. That the inversion $s\mapsto s^{-1}$ is bounded in $L^2\left(\Sigma_R\setminus D(0,r_0)\right)$ as $n\to\infty$ is obvious. The remaining conditions for \cite[Theorem 3.1]{Mo} follow from standard Riemann-Hilbert theory (see Appendix A from \cite{BlKu2} for example). Hence we may apply \cite[Theorem 3.1]{Mo} and the lemma follows. 
\end{proof}

We are ready to give the proof of the main result.\\

\noindent\textit{Proof of Theorem \ref{ch4:mainThm}.} By standard analysis arguments we have that 
\begin{align} \label{eq:err+1VV}
e^{\frac{r n}{r+1} (V(x_n)-V(y_n))} = 1 + \mathcal O\left(\frac{x-y}{n^r}\right)
\end{align}
and
\begin{align} \label{Psir+1fPsi}
\Psi_{\alpha,+}\left(n^{r+1} f(y_n)\right)^{-1} \Psi_{\alpha,+}\left(n^{r+1} f(x_n)\right)
= \Psi_{\alpha,+}(y)^{-1} \Psi_{\alpha,+}(x) + \mathcal O\left(\frac{x-y}{n^{r+1}}\right)
\end{align}
uniformly for $x,y$ in compact sets as $n\to\infty$. 
Plugging \eqref{eq:estimateERRE}, \eqref{eq:err+1VV} and \eqref{Psir+1fPsi} into \eqref{eq:justbeforexnyn} for $x=x_n$ and $y=y_n$, we get
\begin{multline*}
\frac{1}{f'(0) n^{r+1}}K_{V,n}^{\alpha,\frac{1}{r}}(x_n,y_n) = \frac{1}{2\pi i(x-y)} 
\begin{pmatrix} - 1 & 1 & 0 & \cdots & 0\end{pmatrix} 
\Psi_{\alpha,+}(y)^{-1} \Psi_{\alpha,+}(x)
\begin{pmatrix} 1 \\ 1 \\ 0 \\ \vdots \\ 0\end{pmatrix}+\mathcal O\left(\frac{1}{\sqrt n}\right).
\end{multline*}
We know the value of $f'(0)$ from Proposition \ref{ch4:prop:conformalf} and we arrive at
\begin{multline} \label{eq:scalingLimEnd}
\lim_{n\to\infty} \frac{1}{\mathfrak (c n)^{r+1}}K_{V,n}^{\alpha,\frac{1}{r}}\left(\frac{x}{(c n)^{r+1}},\frac{y}{(c n)^{r+1}}\right) = \frac{1}{2\pi i(x-y)} 
\begin{pmatrix} - 1 & 1 & 0 & \cdots & 0\end{pmatrix} 
\Psi_{\alpha,+}(y)^{-1} \Psi_{\alpha,+}(x)
\begin{pmatrix} 1 \\ 1 \\ 0 \\ \vdots \\ 0\end{pmatrix}
\end{multline}
uniformly for $x,y>0$ in compact sets, where $c$ is as in \eqref{eq:defxnyn2} (and as in Theorem \ref{ch4:mainThm}).

We know that this scaling limit must coincide with the one for $V(x)=x$, a case which has already been treated by Borodin \cite{Bo} (for general $\theta>0$ actually). Since the limit, i.e., the right-hand side of \eqref{eq:scalingLimEnd}, is independent of $V$, the limit must hold for all one-cut $\frac{1}{r}$-regular external fields $V$. Theorem \ref{ch4:mainThm} is proved. \qed

\phantomsection
\cleardoublepage
 
\appendix

\section{Removal of the restriction that $r$ divides $n$}\label{ch:appendixA}

We will show that the restriction that $n$ is divisible by $r$ can be removed. Let us fix an integer $p\in\{1,2,\ldots,r-1\}$. We consider all $n = r m+p$, where $m$ runs over the natural numbers. Then the corresponding RHP is the same as RH-Y from Section \ref{ch4:sec:theRHP}, but with RH-Y3 replaced by
\begin{itemize}
\item[RH-Y3] As $|z|\to\infty$
\begin{align} \label{RHY2widetilde}
Y(z) = \left(\mathbb{I}+\mathcal{O}\left(\frac{1}{z}\right)\right) 
\left(1\oplus z^{-\frac{n+r-p}{r}} \mathbb{I}_{p\times p} \oplus z^{-\frac{n-p}{r}} \mathbb{I}_{(r-p)\times (r-p)}\right).
\end{align}
\end{itemize}
Instead of defining the first transformation as in Section \ref{ch4:sec:firstTransformation} directly, we first apply an intermediate transformation $Y\mapsto \widetilde Y$. 
\begin{definition}
We define
\begin{align} \label{eq:defWidetildeY}
\widetilde Y(x) = (1\oplus \sigma) Y(z) 
\left(1\oplus z^{\frac{r-p}{r}} \mathbb{I}_{p\times p} \oplus z^{-\frac{p}{r}} \mathbb{I}_{(r-p)\times (r-p)}\right)
(1\oplus \sigma)
\end{align}
where $\sigma$ is the cyclic permutation matrix whose components are given by
\begin{align}
\sigma_{kj} = \left\{\begin{array}{ll}
1, & \text{if }k \equiv j+p \mod r,\\
0, & \text{otherwise,}
\end{array}\right.
\end{align}
where the indices range from $1$ to $r$.
\end{definition}

\begin{proposition}
$\widetilde Y$ satisfies RH-Y from Section \ref{ch4:sec:theRHP}, but with an additional jump for $x<0$, given by
\begin{align*}
\widetilde Y_+(x) &= \widetilde Y_-(x) \left(1\oplus \Omega^p \mathbb{I}_{r\times r}\right).
\end{align*}
\end{proposition}

\begin{proof}
RH-Y3 is clear. Let $x>0$. We let the indices of the matrices range from $0$ to $1$. Then for $j\geq 1$ we have
\begin{align} \nonumber
&\left(\widetilde Y_-(x)^{-1} \widetilde Y_+(x)\right)_{0j}\\
&= \sum_{k=1}^{p} \left(Y_-(x)^{-1} Y_+(x)\right)_{0k} z^\frac{r-p}{r} \sigma_{kj}
+  \sum_{k=p+1}^r \left(Y_-(x)^{-1} Y_+(x)\right)_{0k} z^{-\frac{p}{r}}\sigma_{kj}\\ \label{eq:YsigmaJump0}
&= \sum_{k=1}^{p} w_{\alpha+\frac{k-1+r-p}{r}}(x) \sigma_{kj}
+ \sum_{k=p+1}^r w_{\alpha+\frac{k-1-p}{r}}(x) \sigma_{kj}.
\end{align}
The term $w_{\alpha+\frac{j-1}{r}}(x)$ corresponds to either $k=j+p-r$ or $k=j+p$ depending on which one of the two is among $1,\ldots,r$. Then by the definition of $\sigma$ we infer that \eqref{eq:YsigmaJump0} equals $w_{\alpha+\frac{j-1}{r}}(x)$. It is clear that other non-zero components of the jump for $x>0$ must be on the diagonal. Then we have for $j=1,\ldots,r$
\begin{align*}
\left(\widetilde Y_-(x)^{-1} \widetilde Y_+(x)\right)_{jj} &= \sum_{k=1}^{r} (\sigma^{-1})_{jk} \left(Y_-(x)^{-1} Y_+(x)\right)_{kl} \sigma_{kj}\\
&=  \sum_{k=1}^{r} (\sigma^{-1})_{jk} \sigma_{kj} = 1. 
\end{align*}
where we have ignored the factors $z^\frac{r-p}{r}$ and $z^{-\frac{p}{r}}$ from the beginning, because they must come from the same block and thus cancel each other. One can also verify that the upper-left component of the jump equals $1$. We conclude that we get the jump from RH-Y2 for $x>0$. 

Let us now look at $x<0$. Indeed, then we get
\begin{align*}
\widetilde Y_-(&x)^{-1} \widetilde Y_+(x)\\
&= 
(1\oplus \sigma)^{-1} \left(1\oplus e^{\frac{r-p}{r} \pi i}|x|^\frac{r-p}{r} \mathbb I_{p\times p} \oplus e^{-\frac{p}{r}\pi i} |x|^{-\frac{p}{r}}\mathbb I_{(r-p)\times (r-p)}\right)^{-1}\\
&\hspace{1cm}\left(1\oplus e^{-\frac{r-p}{r} \pi i}|x|^\frac{r-p}{r} I_{p\times p} \oplus e^{\frac{p}{r}\pi i} |x|^{-\frac{p}{r}} I_{(r-p)\times (r-p)}\right) (1\oplus \sigma)\\
&= (1\oplus \sigma)^{-1} \left(1\oplus \Omega^{p-r} \mathbb I_{p\times p}\oplus \Omega^p \mathbb I_{(r-p)\times (r-p)}\right) (1\oplus \sigma)\\
&= 1\oplus \Omega^p \sigma^{-1} \sigma\\
&= 1 \oplus \Omega^p \mathbb I_{r\times r}. 
\end{align*}
\end{proof}

\begin{proposition} \label{prop:corKerWidetildeY}
The correlation kernel can be expressed as
\begin{align*}
K_{V,n}^{\alpha,\theta}(x,y)=\frac{1}{2\pi i(x-y)}
\begin{pmatrix}
0 & w_\alpha(y) & w_{\alpha+\frac{1}{r}}(y) & \hdots & w_{\alpha+\frac{r-1}{r}}(y)
\end{pmatrix}
\widetilde Y_+^{-1}(y) \widetilde Y_+(x)
\begin{pmatrix}
1 \\ 0 \\ \vdots \\ 0
\end{pmatrix}
\end{align*}
\end{proposition}

\begin{proof}
We should show that the formula coincides with \eqref{ch4:eq:KinY}. It is clear that
\begin{align} \label{eq:secondPartCorKer}
\widetilde Y_+(x)
\begin{pmatrix}
1 \\ 0 \\ \vdots \\ 0
\end{pmatrix}
= Y_+(x)
\begin{pmatrix}
1 \\ 0 \\ \vdots \\ 0
\end{pmatrix}.
\end{align}
Let us look at
\begin{align*}
\begin{pmatrix}
0 & w_\alpha(y) & w_{\alpha+\frac{1}{r}}(y) & \hdots & w_{\alpha+\frac{r-1}{r}}(y)
\end{pmatrix}
(1\oplus \sigma)^{-1}
\left(1\oplus z^{\frac{r-p}{r}} \mathbb{I}_{p\times p} \oplus z^{-\frac{p}{r}} \mathbb{I}_{(r-p)\times (r-p)}\right)^{-1}.
\end{align*}
We label its components with indices from $0$ to $r$. Then its component with index $j$, when $j>0$, is given by
\begin{align*}
\sum_{k=1}^r w_{\alpha+\frac{k-1}{r}}(y) (\sigma^{-1})_{kj} \times
\left\{\begin{array}{rl} 
z^\frac{p-r}{r}, & j\leq p,\\
z^\frac{p}{r}, & j>p
\end{array}\right.
= \left\{\begin{array}{rl} 
\sum_{k=1}^r w_{\alpha+\frac{k-1+t-r}{r}}(y) \sigma_{jk}, & j\leq p,\\
\sum_{k=1}^r w_{\alpha+\frac{k-1+t}{r}}(y) \sigma_{jk}, & j>p.
\end{array}\right.
\end{align*}
Now, considering both cases $j\leq p$ and $j>p$, and using the definition of $\sigma$, we find that the component equals $w_{\alpha+\frac{j-1}{r}}(y)$. We conclude that
\begin{multline} \label{eq:firstPartCorKer}
\begin{pmatrix}
0 & w_\alpha(y) & w_{\alpha+\frac{1}{r}}(y) & \hdots & w_{\alpha+\frac{r-1}{r}}(y)
\end{pmatrix}
\widetilde Y_+^{-1}(y)\\
=\begin{pmatrix}
0 & w_\alpha(y) & w_{\alpha+\frac{1}{r}}(y) & \hdots & w_{\alpha+\frac{r-1}{r}}(y)
\end{pmatrix}
Y_+^{-1}(y)
\end{multline}
and the proposition now follows by inserting \eqref{eq:firstPartCorKer} and \eqref{eq:secondPartCorKer} in \eqref{ch4:eq:KinY}. 
\end{proof}

From this point onwards the transformations of the RHP are carried out in the same way as before, i.e., we apply the transformation to $\widetilde Y \mapsto X$ as in Definition \ref{ch4:def:X}, but with $Y$ replaced by $\widetilde Y$. After that we appy the transformation $X\mapsto T$, as in Definition \ref{ch4:def:T}. The factor $\Omega^p$ for the jump on the negative real axis will then drop out due to \eqref{ch4:g1g2jump}. Namely, for $x<0$, i.e., for odd $j$, we have
\begin{align*}
n(- g_{j-1,-}(x)+g_{j,-}(x)+g_{j,+}(x) - g_{j+1,+}(x)) &= -\frac{2\pi i n}{r}
\equiv -\frac{2\pi i p}{r} \mod 2\pi i
\end{align*}
when $r \equiv 0 \mod 2$. When $r \equiv 1 \mod 2$ one also has to use \eqref{ch4:grjump2}. From this point onwards there are no differences with the case where $n$ is divisible by $r$ anymore and the analysis of the RHP is identical.

\phantomsection
\cleardoublepage
\section{The Meijer G-function}\label{ch:appendixB}

The Meijer G-function is defined by the following 
contour integral:
\begin{align}
\label{appendixA:MeijerG}
\MeijerG{m}{n}{p}{q}{a_1, \ldots, a_p}{b_1, \ldots, b_q}{z} 
= 
\frac{1}{2\pi i} \int_{L} \frac{\prod_{j=1}^{m} \Gamma(b_j +s) \prod_{j=1}^{n} \Gamma(1-a_j -s)}{\prod_{j=m+1}^{q} \Gamma(1-b_j -s) \prod_{j=n+1}^{p} \Gamma(a_j +s)} z^{-s} ds,
\end{align}
where $\Gamma$ denotes the gamma function 
and empty products in \eqref{appendixA:MeijerG} should be interpreted as $1$, as usual. The expressions involved satisfy the following conditions (e.g., see \cite[section 5.2]{Lu}). 
\begin{itemize}
 \item $m, n, p$, and $q$ are integers with $0\leq m\leq q$ and $0\leq n \leq p$.
 \item $a_i - b_j$ is not a positive integer, for all $i=1,\ldots, p$ and $j=1,\ldots,q$. 
 \item There are three possible options for the path of integration $L$.
 \begin{itemize}
 \item[(i)] $L$ is a path from $+i\infty$ to $-i\infty$ so that all the poles of $\Gamma(b_j+s)$ lie to the left of the path, and all poles of $\Gamma(1-a_i-s)$ lie to the right of the path. This option works when $\delta= m+n-\frac{1}{2}(p+q)>0$ for $|\operatorname{arg}(z)|<\delta \pi$. 
 \item[(ii)] $L$ is a loop, starting and ending at $-\infty$, and encircling all the poles of $\Gamma(b_j+s)$ in negative direction, but no poles of $\Gamma(1-a_i-s)$. This option works when $q\geq 1$ and either $p<q$ or $p=q$, for $|z|<1$. 
 \item[(iii)] $L$ is a loop, starting and ending at $+\infty$, and encircling all the poles of $\Gamma(1-a_i-s)$ in positive direction, but no poles of $\Gamma(b_j+s)$. This option works when $p\geq 1$ and either $p>q$ or $p=q$, for $|z|>1$. 
 \end{itemize}
\end{itemize}
Variations on the three possibilities for $L$ are possible. In this paper we have $p=n=0$ and $q=m=r+1$ for a positive integer $r$. Then we are in the situation of option (i) and option (ii). Because $p=0$, we may actually consider the contour of option (ii) for all $|\operatorname{arg}(z)<\frac{r+1}{2} \pi|$.

The Meijer G-function satisfies the following higher order linear differential equation, known as the \textit{generalized hypergeometric equation}.
\begin{align*}
\left[\vartheta \prod_{j=1}^{q-1} \left(\vartheta + b_j -1\right) - (-1)^{p-m-n} z \prod_{i=1}^p \left(\vartheta+a_i\right)\right] \psi(z) = 0,
\quad\quad \vartheta = z \frac{d}{dz}.
\end{align*}
Here, an empty product is to be read as $1$. The Meijer G-function is related to the generalized hypergeometric function via 
\begin{multline*}
\MeijerG{m}{n}{p}{q}{a_1, \ldots, a_p}{b_1, \ldots, b_q}{z} 
= \sum_{h=1}^m \frac{\prod_{j=1}^m \Gamma(b_j - b_h)^* \prod_{i=1}^n \Gamma(1+b_h-a_i)}{\prod_{j=m+1}^q \Gamma(1+b_h-b_j) \prod_{j=n+1}^p \Gamma(a_j-b_h)}\\
 z^{b_h} {_{q} F_{p-1}}\left( {1+b_h - a_1, \ldots 1+b_h-a_p \atop 1+b_h - b_1, \ldots, 1+b_h - b_h^*, \ldots, 1+b_h - b_q} ; (-1)^{p-m-n} z\right)
\end{multline*}
when all $b_j$ are pairwise distinct ($\log$ terms will enter if this is not the case) \cite{Lu}. Here the asterisk $^*$ denotes that the factor with $j=h$ should be suppressed in the product in the first line, and similarly for the parameters of the generalized hypergeometric functions in the last line. 

Indeed, these generalized hypergeometric functions, multiplied with their factor $z^{b_h}$, form a basis of solutions to the generalized hypergeometric equation around $z=0$. In particular, the indicial equation at $z=0$ has solutions $b_1,\ldots,b_q$. 

\section*{Acknowledgements}
The author is supported by a PhD fellowship of the Flemish Science Foundation (FWO) and partly by
the long term structural funding-Methusalem grant of the Flemish Government. The author is grateful
to Thomas Bothner, Tom Claeys, Arno Kuijlaars and Walter Van Assche for proofreading and suggestions.  

\phantomsection
\cleardoublepage
\addcontentsline{toc}{section}{\hspace{0.5cm}References}

\end{document}